\documentclass[arxiv]{imsarxiv}

\RequirePackage{amsthm,amsmath,amsfonts,amssymb}
\RequirePackage[numbers]{natbib}
\RequirePackage[colorlinks,citecolor=blue,urlcolor=blue]{hyperref}%

\usepackage{amsbsy,amsgen,amscd,mathrsfs}
\usepackage{pifont,dsfont}

\usepackage{mathtools}

\usepackage{xcolor} %

\definecolor{cit}{rgb}{0.91,0.39,0.16}	%

\definecolor{dark-gray}{gray}{0.3}

\definecolor{dkgray}{rgb}{.3,.3,.3}
\definecolor{medgray}{rgb}{.5,.5,.5}
\definecolor{ltgray}{rgb}{.7,.7,.7}
\definecolor{dkblue}{rgb}{0,0,.5}
\definecolor{medblue}{rgb}{0,0,.75}
\definecolor{ltblue}{rgb}{0.97,0.97,1}
\definecolor{rust}{rgb}{0.5,0.1,0.1}
\definecolor{ltyellow}{rgb}{1, 1, 0.9}

\usepackage{soul}
\sethlcolor{ltyellow}
\newcommand{\hilite}[1]{#1}

\usepackage[normalem]{ulem}

\usepackage{bm} %

\usepackage[utf8]{inputenc} %
\usepackage[T1]{fontenc} %

\newcommand{\econst}{\mathrm{e}}
\newcommand{\iunit}{\mathrm{i}}

\newcommand{\eps}{\varepsilon}

\renewcommand{\phi}{\varphi}

\newcommand{\onecirc}{\text{\ding{192}}}
\newcommand{\twocirc}{\text{\ding{193}}}

\newcommand{\vct}[1]{\bm{#1}}
\newcommand{\mtx}[1]{\bm{#1}}
\newcommand{\set}[1]{\mathsf{#1}}

\newcommand{\coll}[1]{\mathcal{#1}}

\newcommand{\term}[1]{\hl{\textit{#1}}}

\newcommand{\N}{\mathbb{N}}
\newcommand{\Z}{\mathbb{Z}}

\newcommand{\R}{\mathbb{R}}

\newcommand{\C}{\mathbb{C}}
\newcommand{\F}{\mathbb{F}}

\newcommand{\M}{\mathbb{M}}
\newcommand{\Sym}{\mathbb{H}}
\newcommand{\comp}{\textsf{c}}

\newcommand{\indicator}{\mathds{1}}

\renewcommand{\Re}{\operatorname{Re}}

\newcommand{\range}{\operatorname{range}}

\newcommand{\trace}{\operatorname{Tr}}

\newcommand{\Id}{\mathbf{I}}

\newcommand{\transp}{\mathsf{t}}

\newcommand{\psdle}{\preccurlyeq}
\newcommand{\psdge}{\succcurlyeq}

\newcommand{\abs}[1]{\vert {#1} \vert}
\newcommand{\norm}[1]{\Vert {#1} \Vert}
\newcommand{\ip}[2]{\langle {#1}, \ {#2} \rangle}

\newcommand{\abssq}[1]{\abs{#1}^2}

\newcommand{\fnorm}[1]{\norm{#1}_{\mathrm{F}}}
\newcommand{\fnormsq}[1]{\norm{#1}_{\mathrm{F}}^2}

\newcommand{\labs}[1]{\left\vert {#1} \right\vert}
\newcommand{\lnorm}[1]{\left\Vert {#1} \right\Vert}

\newcommand{\diff}{\mathrm{d}}
\newcommand{\Diff}{\mathrm{D}}
\newcommand{\idiff}{\,\diff}

\newcommand{\Expect}{\operatorname{\mathbb{E}}}
\newcommand{\Var}{\operatorname{Var}}

\newcommand{\Probe}{\mathbb{P}}
\newcommand{\Prob}[1]{\Probe\left\{ #1 \right\}}

\newcommand{\condbar}{\, \vert \,}

\newcommand{\Varo}{\mathsf{Var}}

\newcommand{\normal}{\textsc{normal}}
\newcommand{\uniform}{\textsc{uniform}}

\newcommand{\indep}{\perp \!\!\!\! \perp}

\newcommand{\goe}{\mathrm{goe}}
\newcommand{\gue}{\mathrm{gue}}
\newcommand{\rect}{\mathrm{rect}}

\DeclareFontFamily{U}{matha}{\hyphenchar\font45}
\DeclareFontShape{U}{matha}{m}{n}{
  <-6> matha5 <6-7> matha6 <7-8> matha7
  <8-9> matha8 <9-10> matha9
  <10-12> matha10 <12-> matha12
  }{}

\DeclareSymbolFont{matha}{U}{matha}{m}{n}
\DeclareMathSymbol{\abscont}{3}{matha}{"CE}

\makeatletter
\def\paragraph{\@startsection{paragraph}{4}%
  \z@\z@{-\fontdimen2\font}%
  {\normalfont\scshape}}
\makeatother

\startlocaldefs
\numberwithin{equation}{section}
\theoremstyle{plain}

\newtheorem{theorem}{Theorem}[section]

\newtheorem{proposition}[theorem]{Proposition}
\newtheorem{fact}[theorem]{Fact}

\newtheorem{conjecture}[theorem]{Conjecture}
\newtheorem{corollary}[theorem]{Corollary}

\theoremstyle{definition}

\newtheorem{definition}[theorem]{Definition}

\newtheorem{remark}[theorem]{Remark}

\endlocaldefs

\begin{document}

\begin{frontmatter}
\title{Comparison theorems for the extreme eigenvalues \\ of a random symmetric matrix}
\runtitle{Comparison theorems for extreme eigenvalues}

\begin{aug}
\author[A]{\fnms{Joel A.}~\snm{Tropp}\ead[label=e1]{jtropp@caltech.edu}\orcid{0000-0003-1024-1791}}
\address[A]{Department of Computing and Mathematical Sciences, Caltech, Pasadena, CA, USA\printead[presep={,\ }]{e1}}

\end{aug}

\begin{abstract}
This paper establishes a comparison theorem for the maximum
eigenvalue of a sum of independent random symmetric matrices.
The theorem states that the maximum eigenvalue of the matrix sum
is dominated by the maximum eigenvalue of a Gaussian random matrix
whose statistics match the sum, and it strengthens
previous results of this type.
Corollaries address the minimum eigenvalue and the spectral norm;
the proof strategy also extends to matrix martingale sequences.

The comparison methodology is powerful because of the vast arsenal of tools
for treating Gaussian random matrices.
As applications, the paper improves on existing eigenvalue bounds for random matrices arising
in spectral graph theory, quantum information theory, high-dimensional statistics, and numerical linear algebra.
In particular, these techniques deliver the first complete proof that
a sparse random dimension reduction map has the injectivity properties
conjectured by Nelson \& Nguyen in 2013.
\end{abstract}

\begin{keyword}[class=MSC]
\kwd[Primary ]{15-B52, 60-B20}
\kwd{random matrices}
\end{keyword}

\begin{keyword}
\kwd{Comparison theorem}
\kwd{high-dimensional probability}
\kwd{high-dimensional statistics}
\end{keyword}

\end{frontmatter}

\section{Main results}

Random symmetric matrices have abundant applications in high-dimensional probability,
high-dimensional statistics, and computational mathematics.
For example, they model the adjacency matrix of a random graph,
the empirical covariance of a random vector, and
the action of a randomized subspace embedding.
The most important statistics of these random matrix models
are the extreme eigenvalues, which quantify how much the
matrix can dilate a vector.

We can study the spectral properties of a random matrix
by comparison with a Gaussian random matrix model that
shares the same first- and second-order moments.
It is fruitful to pass to the Gaussian matrix %
because we gain access to a richer cornucopia
of technical methods.
This idea has a long and productive history in
random matrix theory (RMT), summarized in Section~\ref{sec:related}.
Our implementation %
involves a novel application
of Stahl's theorem~\cite{Sta13:Proof-BMV} on the structure
of the trace exponential function of a matrix.
The technique allows us to establish general-purpose theorems
that improve over existing results along several axes.

As evidence, we present stronger guarantees for
random matrices that arise in some contemporary problems.
Our main application is the first complete proof that
sparse random dimension reduction maps enjoy the injectivity
properties conjectured by Nelson \& Nguyen~\cite{NN13:OSNAP-Faster,NN14:Lower-Bounds};
see Section~\ref{sec:sparsestack-appl}.
As compared with existing results, our bounds reduce
the dependence on the dimension of the random matrix model,
so the methods are particularly useful for treating the exponentially
large random matrices that arise in quantum information theory (Section~\ref{sec:quantum-appl}).
Other applications include a new analysis of the expansion properties
of random regular graphs with growing degree (Section~\ref{sec:rdm-graph-appl})
and new bounds for the minimum eigenvalue of a sample covariance matrix (Section~\ref{sec:scov-appl}).

This paper also includes new tail bounds for the maximum
eigenvalue of a matrix martingale sequence.  The bounds
compare the trajectory of the martingale against a fixed
Gaussian reference model.  These results do not appear
to have precedents in the RMT literature.  See Section~\ref{sec:freedman}.

The companion paper~\cite{Tro26:Comparison-Theorems-Minimum} was
the first work to recognize that Stahl's theorem has
implications for RMT, but the arguments there are quite different.
Our analysis addresses both the real and complex setting,
and we use the term \term{self-adjoint} to encompass
both real symmetric and complex Hermitian matrices.
The notation is largely standard; see Section~\ref{sec:notation}. %

\subsection{The independent sum model}

Consider an independent family $(\mtx{W}_1, \dots, \mtx{W}_n)$
of random \hilite{self-adjoint} matrices, either real or complex,
with common dimension $d$.
We place the standing assumption that each random matrix
has two finite moments: $\Expect \norm{\mtx{W}_i}^2 < + \infty$,
where $\norm{\cdot}$ denotes the {spectral norm}.
Construct the sum of the random matrices:
\begin{equation} \label{eqn:indep-sum-main}
\mtx{Y} \coloneqq \sum_{i=1}^n \mtx{W}_i.
\end{equation}
We call this the \term{independent sum model}.  It serves as a versatile
template for describing random matrices that arise in theoretical, applied,
and computational mathematics.
For some examples, see the applications in Section~\ref{sec:applications}
or the expository works~\cite{Tro15:Introduction-Matrix,Tro26:Applied-Random}. %

For the independent sum model~\eqref{eqn:indep-sum-main},
the most important statistics are its extreme eigenvalues,
$\lambda_{\max}(\mtx{Y})$ and $\lambda_{\min}(\mtx{Y})$.
\term{Matrix concentration inequalities} provide probability bounds
for the extreme eigenvalues %
in terms of simple, accessible statistics of the summands~\cite{Tro15:Introduction-Matrix}. %
These results have made an enormous impact over the last 15 years
because they reduce (formerly) challenging problems in random matrix theory
to short linear algebra calculations.
Nevertheless, the basic results, such as the matrix Bernstein inequality,
do not provide very precise information,
and researchers have continued to seek more refined results.

\subsection{The Gaussian proxy model}

This paper pursues the idea that we can obtain better matrix concentration inequalities
by passing from the independent sum~\eqref{eqn:indep-sum-main} to another random matrix
that is easier to analyze.
We compare the independent sum $\mtx{Y}$ %
with a \hilite{Gaussian} self-adjoint matrix $\mtx{Z}$
that has the same first- and second-order moments.
In symbols,
\begin{equation} \label{eqn:gauss-proxy-main}
\mtx{Z} \sim \normal( \Expect[\mtx{Y}], \ \Varo[\mtx{Y}] ).
\end{equation}
Section~\ref{sec:gaussian} contains more background on Gaussian matrices.
For intuition, the multivariate central limit theorem~\cite[Thm.~9.6.1]{Dud02:Real-Analysis}
suggests that slowly varying real-valued functions of $\mtx{Y}$ and $\mtx{Z}$ should have similar expectations,
provided that no summand $\mtx{W}_i$ has an outsize influence on the sum $\mtx{Y}$.
We will identify conditions under which the \hilite{extreme eigenvalues}
of $\mtx{Z}$ control the extreme eigenvalues of $\mtx{Y}$
in the nonasymptotic setting.

This approach has deep roots in the RMT literature, where it is common
to replace a random matrix model with an associated Gaussian model~\cite{Tao19:Least-Singular}.
Recently, Brailovskaya \& van Handel~\cite{BvH24:Universality-Sharp}
developed new methods for comparing the independent sum~\eqref{eqn:indep-sum-main}
with the Gaussian model~\eqref{eqn:gauss-proxy-main}, and they derived a family of nonasymptotic universality results for the spectral statistics.
In a similar spirit, this paper establishes stronger one-sided bounds for the extreme eigenvalues via simpler arguments. %
See Section~\ref{sec:related} for discussion of related work.

\subsection{Matrix fluctuation statistics}
\label{sec:norm-flux}

Our comparison theorems are stated in terms of two summary statistics of the
Gaussian proxy model~\eqref{eqn:gauss-proxy-main}.  Define the \term{matrix fluctuation}:
\begin{equation} \label{eqn:maxeig-flux}
\phi(\mtx{Z}) \coloneqq \Expect \lambda_{\max}( \mtx{Z} - \Expect \mtx{Z} )  \geq 0.
\end{equation}
The quantity~\eqref{eqn:maxeig-flux} describes the scale on which the Gaussian matrix $\mtx{Z}$ varies about
its expectation matrix.  Matrix Khinchin inequalities, such as \eqref{eqn:mki-eig} and \eqref{eqn:mki-2nd} below,
provide actionable upper bounds for the matrix fluctuation~\eqref{eqn:maxeig-flux}
in terms of the second-order moments of the Gaussian matrix,
but it can be challenging to evaluate the matrix fluctuation precisely.

Second, define the \term{weak variance}:
\begin{equation} \label{eqn:weak-var-gauss}
\sigma_*^2(\mtx{Z}) \coloneqq \sup\nolimits_{\norm{\vct{u}}=1} \Var[ \vct{u}^* \mtx{Z} \vct{u} ].
\end{equation}
This quantity controls the concentration of the real random variable $\lambda_{\max}(\mtx{Z})$ %
about its expectation.
The weak variance depends only on the second-order moments of
$\mtx{Z}$, so we can often evaluate~\eqref{eqn:weak-var-gauss} without undue effort.
As a \hilite{heuristic}, the weak variance is often much smaller
than the square of the matrix fluctuation:
$\sigma_*(\mtx{Z}) \ll \phi(\mtx{Z})$.
Section~\ref{sec:gaussian} offers more commentary on these statistics.

\subsection{Comparison: Maximum eigenvalue of a self-adjoint matrix}

The main result of this paper provides a stochastic comparison
between the maximum eigenvalue of an independent sum
of random self-adjoint matrices and its Gaussian proxy.

\begin{theorem}[Comparison: Maximum eigenvalue of an independent sum] \label{thm:maxeig-main}
Consider an independent family $(\mtx{W}_1, \dots, \mtx{W}_n)$ of
random \hilite{self-adjoint} matrices with common dimension $d$
and with two finite moments.
Assume the random matrices satisfy the uniform upper bound
\begin{equation} \label{eqn:maxeig-R+}
\lambda_{\max}(\mtx{W}_i - \Expect \mtx{W}_i) \leq R_+
\quad\text{for $i = 1, \dots, n$.}
\end{equation}
Form the sum and its Gaussian proxy:
\[
\mtx{Y} \coloneqq \sum_{i=1}^n \mtx{W}_i
\quad\text{and}\quad
\mtx{Z} \sim \normal(\Expect[\mtx{Y}], \Varo[\mtx{Y}]).
\]
Then the maximum eigenvalues admit the comparison
\begin{equation} \label{eqn:maxeig-expect}
\Expect \lambda_{\max}(\mtx{Y})
	\leq \Expect \lambda_{\max}(\mtx{Z})
		+ \sqrt{\left(\tfrac{1}{3} R_+ \phi(\mtx{Z}) + \sigma_*^2(\mtx{Z}) \right) \cdot 2\log d}
		+ \tfrac{1}{3} R_+ \log d.
\end{equation}
Furthermore, for a parameter $s \geq 0$, the tail probability satisfies the bound
\begin{equation} \label{eqn:maxeig-tail}
\Prob{ \lambda_{\max}(\mtx{Y}) \geq \Expect \lambda_{\max}(\mtx{Z})  +
	\sqrt{\left(\tfrac{1}{3} R_+ \phi(\mtx{Z}) + \sigma_*^2(\mtx{Z})\right) \cdot 2s} + \tfrac{1}{3} R_+ s }
	\leq d \cdot \econst^{-s}.
\end{equation}
The statistics $\phi(\mtx{Z})$ and $\sigma_*^2(\mtx{Z})$ are defined in Section~\ref{sec:norm-flux}.
\end{theorem}

To prove Theorem~\ref{thm:maxeig-main}, we employ Lindeberg's method to compare the \term{trace moment
generating function (trace mgf)} of the random matrix $\mtx{Y}$ with the
trace mgf of the Gaussian proxy $\mtx{Z}$.
This argument crucially exploits Stahl's theorem~\cite{Sta13:Proof-BMV},
a deep result on the structure of the trace exponential function.
See Section~\ref{sec:maxeig-proof} for the details, along with more refined bounds.

\subsubsection{Remarks}

The bound~\eqref{eqn:maxeig-expect} from Theorem~\ref{thm:maxeig-main} confirms
the intuition we distilled from the central limit theorem.
The expected maximum eigenvalue $\Expect \lambda_{\max}(\mtx{Y})$
of the sum is controlled by the expected maximum eigenvalue
$\Expect \lambda_{\max}(\mtx{Z})$ of the Gaussian proxy, plus some error terms.
The comparison is informative when none of the summands $\mtx{W}_i$
has an undue influence on the sum,
as quantified by the upper bound statistic $R_+$. %
Roughly speaking, when $R_+ \log d \ll \phi(\mtx{Z})$,
the error terms have a smaller scale than the expected maximum
eigenvalue $\Expect \lambda_{\max}(\mtx{Z})$ of the Gaussian proxy.

Let us emphasize that the hypothesis~\eqref{eqn:maxeig-R+} only
places a bound on
the maximum eigenvalue $\lambda_{\max}(\mtx{W}_i - \Expect \mtx{W}_i)$,
rather than %
the spectral norm.
This feature is critical because it yields more precise information about instances,
such as a sample covariance matrix,
where the maximum and minimum eigenvalues of the summands have different scales.
Under the sole assumption~\eqref{eqn:maxeig-R+},
we cannot establish a two-sided exponential concentration inequality for
the maximum eigenvalue $\lambda_{\max}(\mtx{Y})$, %
but only the one-sided tail bound~\eqref{eqn:maxeig-tail}.

By the stability of the Gaussian distribution,
we can also express the Gaussian proxy $\mtx{Z}$ as a
sum of independent terms:
\[
\mtx{Z} = \sum_{i=1}^n \mtx{X}_i
\quad\text{where}\quad \mtx{X}_i \sim \normal(\Expect[\mtx{W}_i], \Varo[\mtx{W}_i]),
\]
and the family $(\mtx{X}_1, \dots, \mtx{X}_n)$ is statistically independent.
By monotonicity %
of the Gaussian distribution (Fact~\ref{fact:gauss-mono}),
Theorem~\ref{thm:maxeig-main} holds true if we replace the Gaussian matrix $\mtx{Z}$
by another Gaussian matrix $\mtx{Z}'$ with greater expectation and with
greater variability:
\begin{equation} \label{eqn:gauss-proxy-dom}
\Expect[ \mtx{Z}' ] \psdge \Expect[ \mtx{Z} ]
\quad\text{and}\quad
\Varo[ \mtx{Z}' ] \psdge \Varo[ \mtx{Z} ].
\end{equation}
The symbol $\psdge$ denotes the \term{psd partial order} on self-adjoint operators,
i.e., the \term{Loewner partial order}.
The two
devices in this paragraph can simplify applications of the theorem.

At a formal level, the bounds~\eqref{eqn:maxeig-expect} and~\eqref{eqn:maxeig-tail}
have some parallels with the matrix Bernstein
inequality~\cite[Thm.~6.1]{Tro12:User-Friendly}.
Indeed, the upper tail of the maximum eigenvalue $\lambda_{\max}(\mtx{Y})$
contains both a subgaussian component and a subexponential component.
The subgaussian component has variance on the scale
$R_+ \phi(\mtx{Z}) + \sigma_*^2(\mtx{Z})$; the
subexponential component has scale $R_+$.
We also pay a weak dimensional factor $\log d$,
which is necessary to describe worst-case instances.

The probability inequalities from Theorem~\ref{thm:maxeig-main}
improve uniformly over similar results from the literature.
Modulo constants, the new bounds are always sharper than the matrix Bernstein
inequality~\cite[Thm.~6.1]{Tro12:User-Friendly} because the dimensional
factors now appear only in the error terms.
The new bounds also strengthen the upper bound for the
maximum eigenvalue~\cite[Cor.~2.7]{BvH24:Universality-Sharp}
due to Brailovskaya \& van Handel because we pose weaker
hypotheses %
and obtain a better dependence on the dimension.
In addition, the proofs are strikingly different.
See~Section~\ref{sec:related} for more discussion.

This paper contains a second main result (Theorem~\ref{thm:freedman})
that gives tail bounds for the maximum eigenvalue of a matrix martingale sequence
by comparison with the maximum eigenvalue of a Gaussian reference model.
When specialized to independent sums, Theorem~\ref{thm:freedman}
strengthens Theorem~\ref{thm:maxeig-main} by providing
uniform control on the sequence of partial sums.

\subsubsection{Example: Wigner's matrix}

As a first example, we can apply Theorem~\ref{thm:maxeig-main} to obtain a very good
bound for the maximum eigenvalue of a Wigner matrix~\cite{Wig55:Characteristic-Vectors}.
For each dimension $d$, this random matrix is real and symmetric,
and it has independent and identically distributed (\term{iid}) Rademacher entries above its diagonal:
\begin{equation} \label{eqn:rad-wigner}
\mtx{Y} \coloneqq %
	\sum_{1 \leq j < k \leq d} \eps_{jk} \cdot (\mathbf{E}_{jk} + \mathbf{E}_{kj})
	\quad\text{where $\eps_{jk} \sim \uniform\{\pm 1\}$ iid.}
\end{equation}
Here, $\mathbf{E}_{jk}$ denotes the nonsymmetric $d \times d$ basis matrix
with a one in the $(j, k)$ position and zeros elsewhere.
The definition~\eqref{eqn:rad-wigner} expresses $\mtx{Y}$
as a sum of independent centered random self-adjoint matrices,
each with maximum eigenvalue bounded by $R_+ \coloneqq 1$.

The Gaussian proxy for the Wigner matrix $\mtx{Y}$ is a
random self-adjoint matrix of the form
\begin{equation} \label{eqn:gauss-wigner}
\mtx{Z} \coloneqq %
	\sum_{1 \leq j < k \leq d} \gamma_{jk} \cdot (\mathbf{E}_{jk} + \mathbf{E}_{kj}) %
	\quad\text{where $\gamma_{jk} \sim \normal_{\R}(0,1)$ iid.}
\end{equation}
By direct calculation, the weak variance statistic $\sigma_*^2(\mtx{Z}) < 2$,
regardless of the dimension $d$.
Meanwhile, we can obtain precise bounds for the maximum eigenvalue
of the Gaussian matrix~\eqref{eqn:gauss-wigner}
using Slepian's lemma~\cite[Sec.~7.3 and Exer.~7.11]{Ver25:High-Dimensional-Probability}:
\[
\phi(\mtx{Z}) = \Expect \lambda_{\max}(\mtx{Z}) %
\leq 2\sqrt{d}.
\]
An application of Theorem~\ref{thm:maxeig-main} yields
\begin{align*}
\Expect \lambda_{\max}(\mtx{Y})
	&\leq \Expect \lambda_{\max}(\mtx{Z}) + \sqrt{\left(\tfrac{1}{3} R_+ \phi(\mtx{Z}) + \smash{\sigma_*^2}(\mtx{Z})\right) \cdot 2\log d} + \tfrac{1}{3} R_+ \log d \\
	&\leq 2 \sqrt{d} + \sqrt{ \left(\tfrac{1}{3} \cdot 2\sqrt{d} + 2\right) \cdot 2 \log d } + \tfrac{1}{3} \log d
	= 2 \sqrt{d} + \mathcal{O}(d^{1/4} \log^{1/2} d).
\end{align*}
The leading term $2\sqrt{d}$ is numerically sharp.
The error term is negligible but does not match the optimal order $\mathcal{O}(d^{-1/6})$
predicted by the Tracy--Widom law~\cite[Thm.~3.4]{BvH26:Extremal-Random}.
For comparison,
the general-purpose maximum eigenvalue bound
of Brailovskaya \& van Handel~\cite[Cor.~2.7]{BvH24:Universality-Sharp}
produces a larger error term $\mathcal{O}(d^{1/3} \log^{2/3} d)$.
It is not possible to prove a similar result using classic matrix concentration tools;
cf.~\cite[Sec.~4.2.1]{Tro15:Introduction-Matrix}.

To summarize: we can quickly reach accurate estimates for
the statistics of the Gaussian proxy $\mtx{Z}$ using
specialized techniques that exploit its Gaussian distribution.
The comparison theorem (Theorem~\ref{thm:maxeig-main}) %
transfers these bounds back to the independent sum $\mtx{Y}$.
This approach is interesting because it is flexible enough
to address a variety of problems with economical arguments.

\subsubsection{Unbounded summands}

We stated the comparison theorem (Theorem~\ref{thm:maxeig-main})
under the strict assumption that the maximum eigenvalue of
each summand admits a uniform upper bound.  There is an analogous
result that comprehends %
unbounded summands.

\begin{corollary}[Comparison: Maximum eigenvalue, unbounded summands] \label{thm:maxeig-unbdd}
Consider an independent family $(\mtx{W}_1, \dots, \mtx{W}_n)$ of
random {self-adjoint} matrices with common dimension $d$
and with two finite moments.  Introduce the real random variable
\[
M \coloneqq \max\nolimits_i \lambda_{\max}(\mtx{W}_i - \Expect \mtx{W}_i).
\]
For a parameter $R \geq 2 \cdot \Expect M$, define the upper bound
statistic
\[
R_+ \coloneqq R + \max\nolimits_i \sigma_*(\mtx{W}_i).
\]
Form the sum and its Gaussian proxy:
\[
\mtx{Y} \coloneqq \sum_{i=1}^n \mtx{W}_i
\quad\text{and}\quad
\mtx{Z} \sim \normal(\Expect[\mtx{Y}], \Varo[\mtx{Y}]).
\]
For each $s \geq 0$, the maximum eigenvalue satisfies
the probability inequality
\begin{multline*}
\Probe\Big\{ \lambda_{\max}(\mtx{Y}) \geq
	\Expect \lambda_{\max}(\mtx{Z})  + \sqrt{2 \sigma_*^2(\mtx{Z})} \\ +
	\sqrt{\left(\tfrac{1}{3} R_+ \phi(\mtx{Z}) + \sigma_*^2(\mtx{Z})\right) \cdot  2s} + \tfrac{1}{3} R_+ s \Big\} %
	\leq \Prob{ M > R } + d \cdot \econst^{-s}.
\end{multline*}
The statistics $\phi(\mtx{Z})$ and $\sigma_*^2(\mtx{Z})$ and the weak variance $\sigma_*^2(\mtx{W}_i)$ are defined in Section~\ref{sec:norm-flux}.
\end{corollary}

The proof of this result depends on a truncation argument
that benefits from the monotonicity properties of Gaussian distributions.
See Proposition~\ref{prop:truncation} for the details.

\subsection{Comparison: Minimum eigenvalue of a self-adjoint matrix}

As an immediate corollary of Theorem~\ref{thm:maxeig-main}, we can derive
a comparison theorem for the \hilite{minimum eigenvalue}
of a random self-adjoint matrix.  This result is especially valuable
for treating a sum of independent random positive-semidefinite (\term{psd}) matrices,
such as a sample covariance matrix.

\begin{corollary}[Comparison: Minimum eigenvalue of an independent sum] \label{cor:mineig-main}
Consider an independent family $(\mtx{W}_1, \dots, \mtx{W}_n)$ of
random {self-adjoint} matrices with common dimension $d$
and with two finite moments.
Assume the random matrices satisfy the uniform \hilite{lower} bound
\[
- R_- \leq
\lambda_{\min}(\mtx{W}_i - \Expect \mtx{W}_i) %
\quad\text{for $i = 1, \dots, n$.}
\]
Form the sum and its Gaussian proxy:
\[
\mtx{Y} \coloneqq \sum_{i=1}^n \mtx{W}_i
\quad\text{and}\quad
\mtx{Z} \sim \normal(\Expect[\mtx{Y}], \Varo[\mtx{Y}]).
\]
Then the minimum eigenvalues admit the comparison
\[
\Expect \lambda_{\min}(\mtx{Y})
	\geq \Expect \lambda_{\min}(\mtx{Z})
		- \Big[ \sqrt{\left(\tfrac{1}{3} R_- \phi(\mtx{Z}) + \smash{\sigma_*^2}(\mtx{Z})\right) \cdot 2\log d}
		+ \tfrac{1}{3} R_- \log d \Big].
\]
Furthermore, for a parameter $s \geq 0$, the tail probability satisfies the bound
\begin{align*}
\Prob{ \lambda_{\min}(\mtx{Y}) \leq \Expect \lambda_{\min}(\mtx{Z}) - \Big[
	\sqrt{\left(\tfrac{1}{3} R_- \phi(\mtx{Z}) + \smash{\sigma_*^2}(\mtx{Z}) \right) \cdot 2s} + \tfrac{1}{3} R_- s \Big] }
	\leq d \cdot \econst^{-s}.
\end{align*}
The statistics $\phi(\mtx{Z})$ and $\sigma_*^2(\mtx{Z})$ are defined in Section~\ref{sec:norm-flux}.
\end{corollary}

\begin{proof}
Let $\mtx{Y}$ be the independent sum %
in Corollary~\ref{cor:mineig-main}.
Note the relation $\lambda_{\min}(\mtx{Y}) = - \lambda_{\max}(-\mtx{Y})$.
An application of Theorem~\ref{thm:maxeig-main} delivers probabilistic
bounds for the maximum eigenvalue of $-\mtx{Y}$.
\end{proof}

\subsubsection{Example: Rademacher covariance matrix}

This example highlights the benefit of %
matrix concentration inequalities that only require one-sided
control on the summands.
Consider iid Rademacher random vectors
$\vct{w}_i \sim \uniform\{\pm 1\}^d$ for $i = 1, \dots, n$,
and form the sample covariance matrix
\begin{equation} \label{eqn:scov-rad}
\widehat{\mtx{K}}_n \coloneqq \frac{1}{n} \sum_{i=1}^n \vct{w}_i \vct{w}_i^\transp
	\eqqcolon \sum_{i=1}^n \mtx{W}_i \in \Sym_d(\R).
\end{equation}
The random matrix is isotropic: $\Expect \widehat{\mtx{K}}_n = \Id_d$.
We seek probabilistic bounds on the minimum and maximum eigenvalues
of $\widehat{\mtx{K}}_n$ as a function of the number $n$ of samples.
For simplicity, assume that $n \geq d$, and define the reciprocal
of the oversampling ratio $\varrho \coloneqq d / n \in (0, 1]$.

The centered summands admit uniform lower and upper bounds:
\[
- n^{-1} = \lambda_{\min}(\mtx{W}_i - \Expect \mtx{W}_i )
	\leq \lambda_{\max}(\mtx{W}_i - \Expect \mtx{W}_i ) = n^{-1} (d-1) < \varrho.
\]
The lower bound statistic $R_- \coloneqq n^{-1}$ is quite small,
while the upper bound statistic $R_+ \coloneqq \varrho$ is much larger.
This distinction leads to divergent bounds for the minimum
and maximum eigenvalue of $\widehat{\mtx{K}}_n$.

In either case, the Gaussian proxy for $\widehat{\mtx{K}}_n$ is
the matrix
\[
\mtx{Z} \coloneqq \Id_d + \frac{1}{\sqrt{n}}
	\sum_{1 \leq j < k \leq d} \gamma_{jk} \cdot (\mathbf{E}_{jk} + \mathbf{E}_{kj}).
\]
Aside from the scaling and shift of the mean, this is the same as
the comparison model~\eqref{eqn:gauss-wigner} for the Wigner matrix.
Using Slepian's lemma again, we determine that the statistics satisfy
\[
\Expect \lambda_{\min}(\mtx{Z}) \geq 1 - 2 \sqrt{\varrho}
\quad\text{and}\quad
\phi(\mtx{Z}) \leq 2 \sqrt{\varrho}
\quad\text{and}\quad
\sigma_*^2(\mtx{Z}) < 2/n.
\]
The comparison theorem for the minimum eigenvalue (Corollary~\ref{cor:mineig-main})
yields
\begin{equation} \label{eqn:scov-min}
\begin{aligned}
\Expect \lambda_{\min}(\widehat{\mtx{K}}_n)
	&\geq \Expect \lambda_{\min}(\mtx{Z}) - \sqrt{ \left(\tfrac{1}{3} R_- \phi(\mtx{Z}) + \sigma_*^2(\mtx{Z})\right) \cdot 2 \log d} - \tfrac{1}{3} R_- \log d \\
	&\geq 1 - 2\sqrt{\varrho} - \mathcal{O}\left(\sqrt{n^{-1} \log d} + n^{-1} \log d \right)
	\to 1 - 2\sqrt{\varrho}.
\end{aligned}
\end{equation}
The limit is taken as $d, n \to \infty$ in fixed proportion $\varrho = d/n$.
Meanwhile, the classic Bai--Yin law~\cite[Thms.~1,2]{BY93:Limit-Smallest}
gives the sharp asymptotic:
\[
\lambda_{\min}(\widehat{\mtx{K}}_n)
	\to 1 - 2\sqrt{\varrho} + \varrho
	\quad\text{almost surely.}
\]
Thus, %
the minimum eigenvalue comparison
is correct to first order when $n \gg d$.

For the \hilite{maximum} eigenvalue of the sample covariance matrix~\eqref{eqn:scov-rad},
the comparison theorem (Theorem~\ref{thm:maxeig-main}) provides a less satisfactory estimate:
\begin{equation} \label{eqn:scov-max}
\begin{aligned}
\Expect \lambda_{\max}(\widehat{\mtx{K}}_n)
	&\leq \Expect \lambda_{\max}(\mtx{Z}) + \sqrt{ \left(\tfrac{1}{3} R_+ \phi(\mtx{Z}) + \sigma_*^2(\mtx{Z})\right) \cdot 2 \log d} + \tfrac{1}{3} R_+ \log d \\
	&\leq 1 + 2\sqrt{\varrho} + \mathcal{O}\left(\varrho^{3/4}\sqrt{\log d} + \varrho \log d \right).
\end{aligned}
\end{equation}
This bound is sharp when $n \gg d \log d$, but it overestimates the
maximum eigenvalue of $\widehat{\mtx{K}}_n$ in the proportional regime $n \asymp d$.
Section~\ref{sec:norm-rad} derives a better estimate for the maximum eigenvalue
by exploiting the internal structure of the summands $\mtx{W}_i$
composing $\widehat{\mtx{K}}_n$.

This analysis parallels the ``S-universality''
statement for covariance matrices established by
Brailovskaya \& van Handel~\cite[Thm.~3.21]{BvH24:Universality-Sharp}.
Their approach leads to a maximum eigenvalue bound similar with~\eqref{eqn:scov-max},
but their argument cannot reproduce the minimum eigenvalue bound~\eqref{eqn:scov-min}
because it requires uniform control on the \hilite{spectral norm}  of the summands.

\subsection{Comparison: Spectral norm of a rectangular matrix}

As a second corollary, we can extend Theorem~\ref{thm:maxeig-main} to obtain
a comparison theorem for the \hilite{spectral norm} of a random
\hilite{rectangular} matrix.  This is one of the most common use
cases for matrix concentration tools.

\begin{corollary}[Comparison: Spectral norm of an independent sum] \label{cor:norm-main}
Consider an independent family $(\mtx{W}_1, \dots, \mtx{W}_n)$ of random %
matrices with common dimension $d_1 \times d_2$.  Assume that the random matrices satisfy the uniform upper bound
\begin{equation*} %
\norm{ \mtx{W}_i - \Expect \mtx{W}_i } \leq R_{\pm}
\quad\text{for $i = 1, \dots, n$.}
\end{equation*}
Form the sum and its Gaussian proxy:
\[
\mtx{Y} \coloneqq \sum_{i=1}^n \mtx{W}_i
\quad\text{and}\quad
\mtx{Z} \sim \normal(\Expect[\mtx{Y}], \Varo[\mtx{Y}]).
\]
Then the spectral norms admit the comparison
\[
\Expect \norm{ \mtx{Y} }
	\leq \Expect \norm{\mtx{Z}}
		+ \sqrt{\left(\tfrac{1}{3} R_{\pm} \phi_{\pm}(\mtx{Z}) + \sigma_*^2(\mtx{Z})\right) \cdot 2\log(d_1 + d_2)}
		+ \tfrac{1}{3} R_{\pm} \log(d_1 + d_2).
\]
Furthermore, for a parameter $s \geq 0$, the tail probability satisfies the bound
\begin{align*}
\Prob{ \norm{\mtx{Y}} \geq \Expect \norm{\mtx{Z}}  +
	\sqrt{\left(\tfrac{1}{3} R_{\pm} \phi_{\pm}(\mtx{Z}) + \sigma_*^2(\mtx{Z})\right) \cdot 2s} + \tfrac{1}{3} R_{\pm} s }
	\leq (d_1 + d_2) \cdot \econst^{-s}.
\end{align*}
For a rectangular matrix, the statistics are defined as
\[
\phi_{\pm}(\mtx{Z}) \coloneqq \Expect \norm{ \mtx{Z} - \Expect \mtx{Z} }
\quad\text{and}\quad
\sigma_*^2(\mtx{Z}) \coloneqq \sup\nolimits_{\norm{\vct{u}_1} = \norm{\vct{u}_2} = 1}
\Var[ \Re (\vct{u}_1^* \mtx{Z} \vct{u}_2) ].
\]
\end{corollary}

\begin{proof}
The \term{self-adjoint dilation}~\cite[Sec.~2.1.17]{Tro15:Introduction-Matrix}
is a real-linear operator defined on rectangular matrices:
\begin{equation} \label{eqn:sa-dilation}
\coll{H} : \F^{d_1 \times d_2} \to \Sym_{d_1+d_2}(\F)
\quad\text{where}\quad
\coll{H}(\mtx{A}) \coloneqq \begin{bmatrix} \mtx{0} & \mtx{A} \\ \mtx{A}^* & \mtx{0} \end{bmatrix}.
\end{equation}
Note that $\lambda_{\max}( \coll{H}(\mtx{A}) ) = - \lambda_{\min}( \coll{H}(\mtx{A})) = \norm{ \mtx{A} }$.
We can control the maximum eigenvalue of the random self-adjoint matrix $\coll{H}(\mtx{Y})$ by an application of
Theorem~\ref{thm:maxeig-main}.  This corollary is the result; see~Section~\ref{sec:gauss-rect} for more exposition.
\end{proof}

\begin{remark}[Universality]
Under the hypotheses of Corollary~\ref{cor:norm-main}, it is possible to prove an
exponential tail bound for %
$\abs{ \norm{\mtx{Y}} - \Expect \norm{\mtx{Z}} }$
by adapting the methods in this paper.
We omit this development because the result contains
additional error terms, and the argument is more involved.
\end{remark}

\subsubsection{Example: Spectral norm of a Rademacher matrix}
\label{sec:norm-rad}

For $n \geq d$,
consider the $d \times n$ rectangular Rademacher matrix
\begin{equation} \label{eqn:rad-rect}
\mtx{Y} \coloneqq \frac{1}{\sqrt{n}} \sum_{j=1}^{d} \sum_{k=1}^{n} \eps_{jk} \cdot \mathbf{E}_{jk} \in \R^{d \times n}
\quad\text{where $\eps_{jk} \sim \uniform\{\pm 1\}$ iid.}
\end{equation}
In the formula~\eqref{eqn:rad-rect}, each summand is centered,
and its spectral norm is bounded by $R_{\pm} \coloneqq n^{-1/2}$.
The Gaussian proxy for the random matrix $\mtx{Y}$ is
\[
\mtx{Z} \coloneqq \frac{1}{\sqrt{n}} \sum_{j=1}^{d} \sum_{k=1}^{n} \gamma_{jk} \cdot \mathbf{E}_{jk}  \in \R^{d \times n}
\quad\text{where $\gamma_{jk} \sim \normal_{\R}(0,1)$ iid.}
\]
By Chevet's theorem~\cite[Thm.~7.3.1]{Ver25:High-Dimensional-Probability},
the statistics satisfy
\[
\Expect \norm{\mtx{Z}} = \phi_{\pm}(\mtx{Z}) \leq 1 + \sqrt{\varrho}
\quad\text{and}\quad
\sigma_*^2(\mtx{Z}) = n^{-1}.
\]
As before, we have abbreviated $\varrho \coloneqq d/n \in (0, 1]$.
The spectral norm comparison theorem (Corollary~\ref{cor:norm-main}) yields the bound
\begin{align*}
\Expect \norm{\mtx{Y}}
	&\leq \Expect \norm{\mtx{Z}}
	+ \sqrt{\left(\tfrac{1}{3} R_{\pm} \phi_{\pm}(\mtx{Z}) + \sigma_*^2(\mtx{Z})\right) \cdot 2\log(d+n)}
	+ \tfrac{1}{3} R_{\pm} \log(d+n) \\
	&\leq 1 + \sqrt{\varrho} + \mathcal{O}\left( \sqrt{{n^{-1/2}} \log n} + n^{-1/2} \log n \right)
	\to 1 + \sqrt{\varrho}. 
\end{align*}
We take the limit as $d, n \to \infty$ in fixed proportion $\varrho = d/n$.

The same calculation also sheds light on the Rademacher sample covariance
matrix $\widehat{\mtx{K}}_n$, defined in~\eqref{eqn:scov-rad}.
Since $\widehat{\mtx{K}}_n = \mtx{YY}^*$, the last display suggests that
\[
\Expect \lambda_{\max}(\widehat{\mtx{K}}_n) = \Expect \norm{\mtx{Y}}^2
	\to (1 + \sqrt{\varrho})^2 = 1 + 2 \sqrt{\varrho} + \varrho.
\]
This heuristic can be fully justified using
the probability bound from Corollary~\ref{cor:norm-main}.
The Bai--Yin law~\cite[Thms.~1, 2]{BY93:Limit-Smallest} ensures
that this analysis is sharp.

A similar estimate follows from the ``Y-universality'' statement
of Brailovskaya \& van Handel~\cite[Thm.~3.24]{BvH24:Universality-Sharp}.
Their arguments also provide bounds for the \hilite{minimum}
singular value of the rectangular random matrix $\mtx{Y}$, which
imply the Bai--Yin law for the minimum eigenvalue
of $\widehat{\mtx{K}}_n$.  In contrast, our approach does not provide data
on the minimum singular value of a random rectangular matrix.

\subsection{Related work}
\label{sec:related}

This paper ties together several strands of research to achieve
new results for the independent sum model.

\subsubsection{Matrix concentration}
\label{sec:mtx-bernstein}

Most of the RMT literature focuses on highly symmetric random matrix models
(e.g., Wigner matrices and sample covariance matrices), and researchers
have achieved an exquisite understanding of these examples.
While these models are fascinating, they are insufficient for
some contemporary applications.

In recent years, the independent sum model~\eqref{eqn:indep-sum-main} has become
a standard template for describing random matrices because it negotiates
a balance between modeling power and probabilistic structure~\cite{Tro15:Introduction-Matrix}.
Matrix concentration inequalities provide good bounds for the extreme
eigenvalues of an independent sum in terms of accessible statistics
of the summands.
While this framework is powerful, researchers have continued searching
for refinements that can produce more precise bounds.

To elaborate, we present one of the fundamental matrix concentration inequalities.
Define the \term{matrix variance} statistic: %
\[
\sigma^2(\mtx{W}) \coloneqq \lnorm{ \Expect\big[ (\mtx{W} - \Expect \mtx{W})^2 \big] }
\quad\text{for random self-adjoint $\mtx{W}$.}
\]
The matrix variance can often be calculated by recourse to elementary
probability and linear algebra.
Under the same assumptions as Theorem~\ref{thm:maxeig-main},
the \term{matrix Bernstein inequality} yields an explicit bound
for the maximum eigenvalue of an independent sum $\mtx{Y}$:
\begin{equation} \label{eqn:mtx-bennett}
\Expect \lambda_{\max}(\mtx{Y})
	\leq \lambda_{\max}(\Expect \mtx{Y}) + \sqrt{2 \sigma^2(\mtx{Y}) \log d} %
	+ \tfrac{1}{3} R_+ \log d.
\end{equation}
This statement combines arguments from~\cite[Thm.~6.1]{Tro12:User-Friendly}
and~\cite[Thm.~6.1.1]{Tro15:Introduction-Matrix}.
Each of the terms in the matrix Bernstein inequality~\eqref{eqn:mtx-bennett}
is required to address worst-case examples~\cite[Sec.~6.1.2]{Tro15:Introduction-Matrix}.
On the other hand, the logarithmic factor attached to the matrix variance
$\sigma^2(\mtx{Y})$ can lead to pessimistic bounds for other instances.

The comparison theorem (Theorem~\ref{thm:maxeig-main}) is always sharper
than the matrix Bernstein inequality~\eqref{eqn:mtx-bennett},
modulo small constant factors.
This claim follows from the matrix Khinchin inequality~\eqref{eqn:mki-eig}
and the property~\eqref{eqn:weak-var-mtx-var} that the matrix variance dominates
the weak variance.

Combining Theorem~\ref{thm:maxeig-main} with the matrix Khinchin inequalities~\eqref{eqn:mki-eig} and~\eqref{eqn:mki-2nd}, we can also derive a \term{second-order matrix Bernstein inequality} that involves the interaction energy statistic $w(\mtx{Y})$,
defined in~\eqref{eqn:interaction-energy}:
\begin{multline} \label{eqn:bernstein-2nd}
\Expect \lambda_{\max}(\mtx{Y})
	\leq \lambda_{\max}(\Expect \mtx{Y})
	+ 2 \sigma(\mtx{Y}) \\
	 + R_+ \log d
	+ \mathrm{Const} \cdot \left[ \sigma^2(\mtx{Y}) (R_+^2 + w(\mtx{Y})) \log^3 d \right]^{1/4}.
\end{multline}
The inequality~\eqref{eqn:bernstein-2nd} frequently allows us to reach
optimal or near-optimal estimates because it relocates a logarithmic
factor in~\eqref{eqn:mtx-bennett} from $\sigma(\mtx{Y})$ to an error term.
This bound provides a variant of a result from
Brailovskaya \& van Handel~\cite[Cor.~2.17]{BvH24:Universality-Sharp}.

\subsubsection{Universality for independent sums}

There is a vast body of work showing that the spectral statistics of a random matrix
depend primarily on low-order moments of its entries, rather than 
their precise distribution.  This phenomenon is called \term{universality}.
One %
strategy for proving universality results is to compare
the original random matrix model with a matching Gaussian distribution,
often by means of Lindeberg's method (Section~\ref{sec:lindeberg}).
Most research on universality treats highly symmetric random matrix models.
See~\cite{Tao19:Least-Singular} for an overview.

Recently, Brailovskaya \& van Handel~\cite{BvH24:Universality-Sharp} developed
a suite of universality theorems for the independent sum model~\eqref{eqn:indep-sum-main}.
Their main results concern an independent sum of bounded random matrices.
They argue that the spectral statistics of the sum are comparable
with the spectral statistics of the Gaussian matrix
whose first- and second-order moments match the sum,
as in~\eqref{eqn:gauss-proxy-main}.
Their results address both the mean spectral distribution
and the support of the spectrum.
The proof involves Stein's method, cumulant expansions, M{\"o}bius inversion,
matrix inequalities, and concentration tools.
See~\cite{Tro26:Universality-Laws} for an alternative approach.

We can make a direct comparison between Theorem~\ref{thm:maxeig-main}
and one of the results from~\cite{BvH24:Universality-Sharp}.
Consider an independent sum of bounded random self-adjoint matrices
with dimension $d$:
\[
\mtx{Y} \coloneqq \sum_{i=1}^n \mtx{W}_i
\quad\text{where}\quad
\norm{ \mtx{W}_i - \Expect \mtx{W}_i } \leq R_{\pm}.
\]
Let $\mtx{Z} \sim \normal(\Expect[\mtx{Y}], \Varo[\mtx{Y}])$ be the Gaussian proxy.
Brailovskaya \& van Handel~\cite[Cor.~2.7]{BvH24:Universality-Sharp} obtain a
two-sided comparison for the maximum eigenvalues:
\begin{equation} \label{eqn:maxeig-bvh}
\labs{ \Expect \lambda_{\max}(\mtx{Y}) - \Expect \lambda_{\max}(\mtx{Z}) }
	\lesssim \left[ R_{\pm} \sigma^2(\mtx{Z}) \log^2 d \right]^{1/3} + R_{\pm} \log d + \sqrt{ \sigma_*^2(\mtx{Z}) \log d}. 
\end{equation}
The relation $\lesssim$ suppresses a universal constant.

The new bound~\eqref{eqn:maxeig-expect} from~Theorem~\ref{thm:maxeig-main} is always sharper than the
\hilite{upper} bound on $\Expect \lambda_{\max}(\mtx{Y})$ from~\eqref{eqn:maxeig-bvh}.
Indeed, our result only requires one-sided control on the summands; cf.~\eqref{eqn:maxeig-R+},
and it removes some of the logarithmic factors. %
Section~\ref{sec:applications} develops some applications
that benefit from these improvements.

It is unlikely that the arguments from~\cite{BvH24:Universality-Sharp}
can deliver tail bounds under one-sided control of the summands.
In addition, our proof of~\eqref{eqn:maxeig-expect} is
conceptually simpler and far shorter than the proof of~\eqref{eqn:maxeig-bvh}.
On the other hand, our approach struggles to reach
two-sided bounds like~\eqref{eqn:maxeig-bvh}, and it does not provide any information
about the mean spectral distribution or the support
of the spectrum.

\subsubsection{Lindeberg's method}
\label{sec:lindeberg}

To prove universality theorems about the spectral statistics of a random matrix,
one of the standard strategies is
the modern revision of Lindeberg's exchange method~\cite{Cha06:Generalization-Lindeberg,Tao19:Least-Singular}.
Lindeberg's method allows us to compare smooth functions
of two sequences $(W_1, \dots, W_n)$
and $(X_1, \dots, X_n)$ of independent random variables
that have matching first- and second-order moments.
The argument interpolates between the two sequences by replacing
one random variable at a time, controlling the discrepancy with a
Taylor expansion.
Papers that employ Lindeberg's method for RMT
include~\cite{KM11:Applications-Lindeberg,TV10:Random-Matrices,TV11:Random-Matrices,OT17:Universality-Laws}.

In particular, Chen et al.~\cite{CDB+24:Sparse-Random} applied Lindeberg's method
to obtain a suite of universality theorems for a sum of independent random matrices.
Their results are incomparable with the bounds from~\cite{BvH24:Universality-Sharp},
so the two strands of research are complementary.
The proof of our main result (Theorem~\ref{thm:maxeig-main}) also depends on
Lindeberg's method, but our implementation hinges on a novel application
of Stahl's theorem~\cite{Sta13:Proof-BMV}.
As compared with existing work,
our analysis produces stronger one-sided bounds for the extreme eigenvalues
of an independent sum.

\subsubsection{Stahl's theorem and comparison for psd matrices}

Stahl's theorem~\cite{Sta13:Proof-BMV} asserts that the trace exponential function
on a line can be represented as the Laplace transform of a positive measure.
More precisely, for fixed self-adjoint matrices $\mtx{A}$ and $\mtx{H}$,
\[
\trace \econst^{\mtx{A} + t \mtx{H}}
	= \int_{\R} \econst^{\lambda t} \idiff{\nu}(\lambda)
	\quad\text{for $t \in \R$,}
\]
where $\nu$ is a positive Borel measure on the real line that depends on the matrices
$\mtx{A}$ and $\mtx{H}$.  This paper develops a new method for exploiting
Stahl's theorem to study random matrices.

To the best of our knowledge, the companion paper~\cite{Tro26:Comparison-Theorems-Minimum}
contains the only existing application of Stahl's theorem in RMT.
The main result of the companion paper concerns the minimum eigenvalue
of a sum of \hilite{iid} random \hilite{psd} matrices.

\begin{fact}[Comparison: Sum of iid psd matrices] \label{fact:mineig-psd}
Let $\mtx{W}$ be a random \hilite{psd} matrix with dimension $d$
and with two finite moments.
Draw a \hilite{Gaussian self-adjoint} matrix $\mtx{X}$ with dimension $d$
whose first- and second-order moments satisfy
\begin{equation} \label{eqn:iid-intro-match}
\Expect \mtx{X} = \Expect \mtx{W}
\quad\text{and}\quad
\Var[ \ip{\mtx{X}}{\mtx{M}} ] = \Expect[ \ip{\mtx{W}}{\mtx{M}}^2 ]
\quad\text{for all self-adjoint $\mtx{M}$.}
\end{equation}
For a natural number $n \in \N$, define the random matrices
\begin{equation} \label{eqn:iid-comp}
\begin{aligned}
\mtx{Y} &\coloneqq %
\sum_{i=1}^n \mtx{W}_i
&&\text{where $\mtx{W}_i \sim \mtx{W}$ iid;} \\
\mtx{Z} &\coloneqq %
\sum_{i=1}^n \mtx{X}_i
&&\text{where $\mtx{Z}_i \sim \mtx{X}$ iid.}%
\end{aligned}
\end{equation}
Then the minimum eigenvalues admit the comparison
\begin{align} \label{eqn:intro-thm-iid-expect}
\Expect \lambda_{\min}(\mtx{Y})
	&\geq %
	\Expect \lambda_{\min}( \mtx{Z} ) - \sqrt{{\sigma_*^2}(\mtx{Z}) \cdot 2\log(2d)}. %
\end{align}
Furthermore, for all $s \geq 0$, the tail probability satisfies the bound
\begin{align} \label{eqn:intro-thm-iid-tail}
\Prob{ \lambda_{\min}(\mtx{Y}) \leq \Expect \lambda_{\min}( \mtx{Z} ) - \sqrt{ {\sigma_*^2}(\mtx{Z}) \cdot 2s} }
	&\leq 2d \cdot \econst^{- s}. %
\end{align}
The weak variance $\sigma_*^2(\mtx{Z})$ is defined in~\eqref{eqn:weak-var-gauss}.
\end{fact}

The proof of Fact~\ref{fact:mineig-psd} involves continuous interpolation,
exchangeable pairs, Stahl's theorem, and Poissonization techniques.
That argument is distant from the proof of Theorem~\ref{thm:maxeig-main}.

Let us emphasize that the Gaussian comparison model~\eqref{eqn:iid-comp}
differs from the model~\eqref{eqn:gauss-proxy-main} we are using in this paper.
The variability of~\eqref{eqn:iid-comp} is larger because
it reflects the raw second moment of the summand rather than its variance.
In compensation, the bounds~\eqref{eqn:intro-thm-iid-expect} and~\eqref{eqn:intro-thm-iid-tail}
do not suffer extra terms involving the lower bound statistic
$R_-$ of the centered summands.

As tools for studying the minimum eigenvalue,
Fact~\ref{fact:mineig-psd} and Corollary~\ref{cor:mineig-main}
are incomparable with each other.
There are some examples where the two theorems
lead to similar outcomes, such as the minimum
eigenvalue bound~\eqref{eqn:scov-min} for a
Rademacher sample covariance matrix.
The older result (Fact~\ref{fact:mineig-psd}) gives
better results in certain cases, such as
sampling from a complex projective 2-design~\cite[Thm.~4.2]{Tro26:Comparison-Theorems-Minimum}.
The new result (Corollary~\ref{cor:mineig-main}) gives
better results for other problems, such as
the application to the sparse randomized
dimension reduction (Section~\ref{sec:sparsestack-appl}).

\subsection{Notation}
\label{sec:notation}

For the most part, we use standard conventions from linear algebra and probability,
and the notation is similar with the companion paper~\cite{Tro26:Comparison-Theorems-Minimum}.
Nonlinear functions bind before the trace and the expectation.
We may omit brackets enclosing an argument when they are unnecessary.
The symbols $\vee$ and $\wedge$ refer to the infix maximum and minimum.
 
Our results hold in both the real ($\F = \R$) and complex ($\F = \C$) setting.
For $d \in \N$, the set $\M_d(\F)$ is the linear
space of $d \times d$ square matrices with entries in the field $\F$.
The set $\Sym_d(\F)$ denotes the real-linear space of \term{self-adjoint} %
matrices in $\M_d(\F)$.
The function $\trace[\cdot]$ returns the (unnormalized) trace of a square matrix.
We write $\ip{ \cdot }{ \cdot }$ for the standard trace inner product.

The symbol $\Id_d$ refers to the identity matrix with dimension $d$;
we sometimes omit the dimensional subscript.
The symbol $\mathbf{E}_{ij}$
denotes the standard basis matrix with a one in
the $(i, j)$ position and zeros elsewhere.
These matrices' dimensions are determined by context.

The symbol ${}^\transp$ denotes the transpose of a vector or matrix,
while ${}^*$ denotes the conjugate transpose. %
The norm $\norm{\cdot}$ refers the $\ell_2$ norm on vectors
or the associated $\ell_2$ operator norm on matrices,
which we call the \term{spectral norm}.
The Frobenius norm $\fnorm{\cdot}$ also appears regularly.

The functions $\lambda_{\max}$ and $\lambda_{\min}$ compute the largest and
smallest eigenvalues of a self-adjoint matrix.  The function
$\lambda_j$ returns the $j$th largest eigenvalue.
The psd partial order $\mtx{A} \psdge \mtx{H}$ on self-adjoint matrices
means that $\mtx{A} - \mtx{H}$ is psd.
We also use the symbol $\psdge$ for the partial order on variance functions
defined in~\eqref{eqn:varfn-order}.

The operator $\Probe(\cdot)$ returns the probability of an event,
while the operator $\Expect[\cdot]$ computes the expectation of
a random variable.  The symbol $\sim$ means ``has the distribution''.
Small capitals denote named distributions.  In particular,
$\normal_{\R}(m, c)$ is the Gaussian distribution on the real line
with expectation $m$ and variance $c$.
The abbreviation \term{iid} means ``independent
and identically distributed''.

The notation $\mathcal{O}(\cdot)$ refers to a function
growing no faster than its argument, and $o(\cdot)$
refers to a function tending to zero with its argument.
In heuristic discussions, we may employ orders $\lesssim$ and $\gtrsim$
and $\asymp$ that
suppress universal constants, while %
$\ll$ means ``much less than''.

\section{Gaussian random matrices}
\label{sec:gaussian}

This section briefly introduces the theory
of Gaussian random matrices.  The presentation is adapted
from the companion paper~\cite[Sec.~3]{Tro26:Comparison-Theorems-Minimum}.

\subsection{First- and second-order moments}
\label{sec:moments}

We work primarily in the real linear space $\Sym_d(\F)$ of self-adjoint
matrices with dimension $d$, whose entries lie in the field $\F = \R$ or $\F = \C$.
This linear space is equipped with the real-linear trace inner product
\[
\ip{ \mtx{A} }{ \mtx{B} } \coloneqq \trace[ \mtx{AB} ]
\quad\text{for $\mtx{A}, \mtx{B} \in \Sym_d(\F)$.}
\]
This inner product reports real values, even for self-adjoint matrices
with complex entries.

Consider a random matrix $\mtx{W}$ taking values in $\Sym_d(\F)$,
and assume that $\mtx{W}$ is square-integrable: $\Expect \norm{\mtx{W}}^2 < + \infty$.
The most basic statistic is its first moment: $\Expect[ \mtx{W} ] \in \Sym_d(\F)$.
The expectation is computed entrywise, and it produces a deterministic matrix.

To describe the fluctuations of the random matrix $\mtx{W}$, we introduce
the \term{variance function}:
\begin{align*}
\Varo[\mtx{W}] : \Sym_d(\F) \to \R_+
\quad\text{where}\quad
\Varo[\mtx{W}] : \mtx{M} \mapsto \Var[ \ip{\mtx{W}}{\mtx{M}} ].
\end{align*}
Equivalently, for each $\mtx{M} \in \Sym_d(\F)$,
\[
\Varo[\mtx{W}](\mtx{M}) = \Expect{} \ip{\mtx{W} - \Expect \mtx{W}}{\mtx{M}}^2
	= \Expect{} \ip{\mtx{W}}{\mtx{M}}^2
	- \ip{\Expect \mtx{W}}{\mtx{M}}^2.
\]
Recall that the square binds before the expectation.
This function collects the second-order central moments
of the one-dimensional linear marginals of the random matrix.

The variance function $\Varo[\mtx{W}]$ %
is a psd quadratic form on $\Sym_d(\F)$.  %
Variance functions
stand in one-to-one correspondence with psd quadratic forms
(i.e., covariance operators of self-adjoint random matrices).
For two variance functions $\set{V}, \set{V}'$, we define the
partial order
\begin{equation} \label{eqn:varfn-order}
\set{V} \psdle \set{V}'
\quad\text{if and only if}\quad
\set{V}(\mtx{M}) \leq \set{V}'(\mtx{M})
\quad\text{for all $\mtx{M} \in \Sym_d(\F)$.}
\end{equation}
If we regard $\set{V}, \set{V}'$ as psd quadratic forms, %
then $\psdle$ agrees with the semidefinite order.

First- and second-order moments cooperate with independent sums.  Consider an independent
family $(\mtx{W}_1, \dots, \mtx{W}_n)$ of square-integrable random matrices
in $\Sym_d(\F)$.  Then
\begin{equation*}
\Expect\left[ \sum_{i=1}^n \mtx{W}_i \right] = \sum_{i=1}^n \Expect[ \mtx{W}_i ]
\quad\text{and}\quad
\Varo\left[ \sum_{i=1}^n \mtx{W}_i \right] %
	= \sum_{i=1}^n \Varo[ \mtx{W}_i ]. %
\end{equation*}
That is to say, we can express the first- and second-order moments of an independent sum
directly in terms of the moments of the summands.

The second-order moments of a random self-adjoint matrix can
also be extracted from the second-order moment tensor.

\begin{fact}[Covariance tensor] \label{fact:moment-tensor}
Suppose that $\mtx{W}, \mtx{X} \in \Sym_d(\F)$ are random \hilite{self-adjoint} matrices
with the same expectation: $\Expect \mtx{W} = \Expect\mtx{X}$. %
Then
\[
\Varo[\mtx{W}] = \Varo[\mtx{X}]
\quad\text{if and only if}\quad
\Expect[ \mtx{W} \otimes \mtx{W} ] = \Expect[ \mtx{X} \otimes \mtx{X} ].
\]
We employ the Kronecker product as a concrete model for the tensor product operator.
\end{fact}

In the complex setting, Fact~\ref{fact:moment-tensor} depends on the assumption
that the random matrices are self-adjoint.  For the most part, we find it
more transparent to work with variance functions rather than tensor products.

\subsection{Gaussian random matrices}

Consider a random self-adjoint matrix $\mtx{X}$ taking values in $\Sym_d(\F)$.
We say that $\mtx{X}$ is \term{Gaussian} when the one-dimensional linear marginal
$\ip{\mtx{X}}{\mtx{M}}$ follows a real-valued Gaussian distribution
\hilite{for every} matrix $\mtx{M} \in \Sym_d(\F)$.
A Gaussian random matrix is fully characterized by its expectation and its variance function.
Suppose that $\mtx{\Delta} \in \Sym_d(\F)$ is an arbitrary deterministic matrix, and
let $\mathsf{V} : \Sym_d(\F) \to \R_+$ be any psd quadratic form.
We write $\normal( \mtx{\Delta}, \mathsf{V} )$ for the (unique) Gaussian distribution
on $\Sym_d(\F)$ with expectation $\mtx{\Delta}$ and variance function $\mathsf{V}$.

Two random matrices $\mtx{X}, \mtx{Z} \in \Sym_d(\F)$ are \term{jointly Gaussian}
when their direct sum $\mtx{X} \oplus \mtx{Z} \in \Sym_{2d}(\F)$ is also a Gaussian matrix. %
Gaussian matrices enjoy a fundamental \term{stability property}.
If $\mtx{X}$ and $\mtx{Z}$ are jointly Gaussian and $\alpha, \beta \in \R$,
then the linear combination $\alpha \mtx{X} + \beta \mtx{Z}$ remains a Gaussian random matrix.
Jointly Gaussian matrices $\mtx{X}$ and $\mtx{Z}$ also have the special feature that uncorrelation
is equivalent to statistical independence:
\[
\Varo[ \mtx{X} + \mtx{Z} ] = \Varo[\mtx{X}] + \Varo[\mtx{Z}]
	\quad\text{if and only if}\quad
	\mtx{X} \indep \mtx{Z}.
\]
The symbol $\indep$ means \term{statistically independent}.
To be clear, the entries \hilite{within} each matrix $\mtx{X}$ or $\mtx{Z}$ can be statistically dependent,
but the two matrices are statistically independent from each other.
These statements extend from two Gaussian matrices to several Gaussian matrices
in the obvious way.

\subsection{Monotonicity}

Gaussian matrices exhibit a beautiful monotonicity property.
We will employ this result to argue that many statistics of
a Gaussian matrix increase with the variance function.
See~\cite[Prop.~3.2]{Tro26:Comparison-Theorems-Minimum} for a short proof.

\begin{fact}[Gaussian monotonicity] \label{fact:gauss-mono}
Consider two \hilite{Gaussian} self-adjoint matrices $\mtx{X} \sim \normal(\mtx{\Delta}, \mathsf{V})$
and $\mtx{X}' \sim \normal(\mtx{\Delta}, \mathsf{V}')$ that share the same expectation $\mtx{\Delta}$.
Assume that the variance functions admit the
comparison $\set{V} \psdle \set{V}'$.
Then, for each convex function $f : \Sym_d(\F) \to \R$,
\[
\Expect f(\mtx{X}) \leq \Expect f(\mtx{X}').
\]
\end{fact}

In particular, the matrix fluctuation $\phi(\mtx{X}) \coloneqq \Expect \lambda_{\max}(\mtx{X} - \Expect \mtx{X})$
and norm fluctuation $\phi_{\pm}(\mtx{X}) \coloneqq \Expect \norm{\mtx{X} - \Expect \mtx{X}}$
increase with the variance function of the Gaussian matrix $\mtx{X}$.

\subsection{Variance statistics}

We can capture information about the spectral features of a Gaussian matrix
by means of summary statistics that are derived from the variance function.
In this section, both $\mtx{X}, \mtx{X}'$ are Gaussian matrices
taking values in $\Sym_d(\F)$.

\subsubsection{The matrix variance}

Define the \term{matrix variance}:
\[
\sigma^2(\mtx{X}) \coloneqq \lnorm{ \Expect\big[ (\mtx{X} - \Expect \mtx{X})^2 \big] }.
\]
The statistic $\sigma^2(\mtx{X})$ increases with the variance function $\Varo[\mtx{X}]$. %
That is,
\[
\Varo[\mtx{X}] \psdle \Varo[\mtx{X}']
\quad\text{implies}\quad
\sigma^2(\mtx{X}) \leq \sigma^2(\mtx{X}').
\]
For a Gaussian self-adjoint matrix $\mtx{X}$,
the matrix variance provides some control on the matrix fluctuation
$\phi(\mtx{X})$:
\begin{equation} \label{eqn:mki-eig}
0 \leq \phi(\mtx{X}) \coloneqq \Expect \lambda_{\max}(\mtx{X} - \Expect \mtx{X})
	= - \Expect \lambda_{\min}(\mtx{X} - \Expect \mtx{X})
	\leq \sqrt{2 \sigma^2(\mtx{X}) \log d}.
\end{equation}
The lower bound is Jensen's inequality, while the upper bound appears in~\cite[Thm.~4.6.1]{Tro15:Introduction-Matrix}.
For the norm fluctuation $\phi_{\pm}(\mtx{X})$, we can achieve a two-sided comparison:
\begin{equation}  \label{eqn:mki-norm}
\sqrt{(2/\pi) \sigma^2(\mtx{X})} \leq \phi_{\pm}(\mtx{X}) \coloneqq \Expect \norm{\mtx{X} - \Expect \mtx{X}} \leq \sqrt{2 \sigma^2(\mtx{X}) \log(2d)}.
\end{equation}
The lower bound follows from~\cite[Cor.~3]{LO99:Gaussian-Measures},
and the upper bound is drawn from~\cite[Thm.~4.1.1]{Tro15:Introduction-Matrix}.
These results are closely related to the \term{matrix Khinchin inequality}.
Each of the inequalities in~\eqref{eqn:mki-eig} and~\eqref{eqn:mki-norm}
is saturated by particular Gaussian matrices as $d \to \infty$.

\subsubsection{The weak variance}

Next, define the \term{weak variance}:
\begin{equation} \label{eqn:weak-var}
\sigma_*^2(\mtx{X}) \coloneqq \sup\nolimits_{\norm{\vct{u}} = 1} \Var[ \vct{u}^* \mtx{X} \vct{u} ]
	= \sup\nolimits_{\norm{\mtx{M}}_* = 1} \Var[ \mtx{X} ] (\mtx{M}).
\end{equation}
We have written $\norm{\cdot}_*$ for the Schatten 1-norm.  From the latter representation,
it is obvious that the relation
\[
\Varo[\mtx{X}] \psdle \Varo[\mtx{X}']
\quad\text{implies}\quad
\sigma_*^2(\mtx{X}) \leq \sigma_*^2(\mtx{X}').
\]
The weak variance and the matrix variance are comparable
up to a dimensional factor: %
\begin{equation} \label{eqn:weak-var-mtx-var}
\sigma_*^2(\mtx{X}) \leq \sigma^2(\mtx{X}) \leq d \cdot \sigma_*^2(\mtx{X}).
\end{equation}
Both inequalities in~\eqref{eqn:weak-var-mtx-var} are attained by particular examples.
The weak variance controls the concentration of the eigenvalues
of a Gaussian random matrix; see~\cite[Fact 3.4]{Tro26:Comparison-Theorems-Minimum}
for a proof. %

\begin{fact}[Gaussian concentration] \label{fact:gauss-conc}
Let $f : \Sym_d(\F) \to \R$ be a function that is 1-Lipschitz
with respect to the spectral norm:
\[
\abs{ f(\mtx{A}) - f(\mtx{B}) } \leq \norm{ \mtx{A} - \mtx{B} }
\quad\text{for all $\mtx{A}, \mtx{B} \in \Sym_d(\F)$.}
\]
For a Gaussian self-adjoint matrix $\mtx{X}$ taking values in $\Sym_d(\F)$,
\[
\Expect \econst^{ f(\mtx{X}) - \Expect f(\mtx{X}) }
	\leq \econst^{ \sigma_*^2(\mtx{X}) / 2}.
\]
In particular, this result applies to $f = \lambda_{\max}$
and $f = \lambda_{\min}$ and $f = \norm{\cdot}$.
\end{fact}

\subsubsection{The interaction energy}

Last, we define the \term{interaction energy}:
\begin{equation} \label{eqn:interaction-energy}
w(\mtx{X}) \coloneqq \sup\nolimits_{\fnorm{\mtx{M}} = 1} \Varo[\mtx{X}](\mtx{M}).
\end{equation}
Equivalently, the interaction energy is the maximum eigenvalue
of the variance function, interpreted as a psd quadratic form (i.e., covariance operator).
Roughly speaking, the interaction energy $w(\mtx{X})$ is \hilite{small} when
two independent copies of the Gaussian matrix $\mtx{X}$
are ``highly noncommutative'' in an appropriate sense~\cite[Sec.~1.4.1]{BBvH23:Matrix-Concentration}.

From the formulation~\eqref{eqn:interaction-energy},
we recognize that the interaction energy is monotone
with respect to the variance function:
\[
\Varo[\mtx{X}] \psdle \Varo[\mtx{X}']
	\quad\text{implies}\quad w(\mtx{X}) \leq w(\mtx{X}').
\]
It is also clear that the interaction energy compares with
the weak variance, up to a dimensional factor:
\[
\sigma_*^2(\mtx{X}) \leq w(\mtx{X}) \leq d \cdot \sigma_*^2(\mtx{X}).
\]
While the interaction energy $w(\mtx{X})$
is often smaller than the matrix variance $\sigma^2(\mtx{X})$,
these statistics are not comparable.

The interaction energy arises in the theory of intrinsic freeness,
and it sometimes allows for improvements over the matrix Khinchin
inequality~\eqref{eqn:mki-norm}.  Indeed, for a Gaussian matrix,
\begin{equation} \label{eqn:mki-2nd}
\phi_{\pm}(\mtx{X}) \coloneqq \Expect \norm{ \mtx{X} - \Expect \mtx{X} }
	\leq 2 \sqrt{\sigma^2(\mtx{X})} + \mathrm{Const} \cdot \left[ \sigma^2(\mtx{X}) w(\mtx{X}) \log^3 d \right]^{1/4}.
\end{equation}
This result is extracted from~\cite[Cor.~2.2 and Lem.~2.5]{BBvH23:Matrix-Concentration}.
The same bound~\eqref{eqn:mki-2nd} evidently holds for the matrix fluctuation $\phi(\mtx{X})$.
The result~\eqref{eqn:mki-2nd} is one of the primary tools that
Brailovskaya \& van Handel~\cite{BvH24:Universality-Sharp} exploit to
pass from universality theorems to matrix concentration inequalities.
This paper only treats simple examples that obviate
the need for~\eqref{eqn:mki-2nd}.

\subsection{Examples: GOE and GUE}

In this section, we introduce the two most fundamental Gaussian random matrices,
and we provide bounds for their statistics.

\subsubsection{Gaussian orthogonal ensemble}
\label{sec:goe}

For each dimension $d$, we can construct a member of 
the Gaussian orthogonal ensemble (GOE):
\[
\mtx{X}_{\goe} \coloneqq \frac{1}{\sqrt{2}} \sum\nolimits_{j, k = 1}^d \gamma_{jk} \cdot (\mathbf{E}_{jk} + \mathbf{E}_{kj})
\in \Sym_d(\R)
\quad\text{where $\gamma_{jk} \sim \normal_{\R}(0,1)$ iid.}
\]
This is a real symmetric Gaussian matrix.  It has iid $\normal_{\R}(0,1)$ entries
above the diagonal, while its diagonal entries are iid $\normal_{\R}(0,2)$.
Its expectation and variance function are
\[
\Expect \mtx{X}_{\goe} = \mtx{0}
\quad\text{and}\quad
\Varo[\mtx{X}_{\goe}](\mtx{M}) = 2 \fnormsq{\mtx{M}}
\quad\text{for $\mtx{M} \in \Sym_d(\R)$.}
\]
The GOE distribution is invariant under orthogonal conjugation:
\[
\mtx{X}_{\goe} \sim \mtx{Q}^\transp \mtx{X}_{\goe} \mtx{Q}
\quad\text{for all orthogonal $\mtx{Q} \in \M_d(\R)$.}
\]
As a consequence, we can define a GOE matrix acting on any subspace of $\R^d$.
It follows from~\cite[Cor.~6.38]{AS17:Alice-Bob} that
the matrix fluctuation statistic of the GOE matrix satisfies the bounds
\[
\phi(\mtx{X}_{\goe}) \leq \phi_{\pm}(\mtx{X}_{\goe}) \leq 2 \sqrt{d}.
\]
The matrix variance, weak variance, and interaction energy admit the formulas
\[
\sigma^2(\mtx{X}_{\goe}) = d + 1
\quad\text{and}\quad
\sigma_*^2(\mtx{X}_{\goe}) = 2
\quad\text{and}\quad
w(\mtx{X}_{\goe}) = 2.
\]
These expressions follow from direct calculation.

\subsubsection{Gaussian unitary ensemble}
\label{sec:gue}

For each dimension $d$, we can also construct a member of
the Gaussian unitary ensemble (GUE):
\begin{equation*}
\mtx{X}_{\gue} \coloneqq \frac{1}{2} \sum\nolimits_{j,k = 1}^d \big[\gamma_{jk} \cdot (\mathbf{E}_{jk} + \mathbf{E}_{kj})
	+ \iunit \gamma_{jk}' \cdot (\mathbf{E}_{jk} - \mathbf{E}_{kj}) \big]
	\in \Sym_d(\C),
\end{equation*}
where $\gamma_{jk}, \gamma_{jk}' \sim \normal_{\R}(0,1)$ iid
and $\iunit$ denotes the imaginary unit.
This is a complex Hermitian Gaussian matrix.  It has iid complex standard normal entries
above its diagonal and iid real standard normal entries on the diagonal.
The expectation and variance function are
\[
\Expect \mtx{X}_{\gue}  = \mtx{0}
\quad\text{and}\quad
\Varo[\mtx{X}_{\gue}](\mtx{M}) = \fnormsq{\mtx{M}}
\quad\text{for $\mtx{M} \in \Sym_d(\C)$.}
\]
The GUE distribution is invariant under unitary conjugation:
\[
\mtx{X}_{\gue} \sim \mtx{Q}^* \mtx{X}_{\gue} \mtx{Q}
\quad\text{for all unitary $\mtx{Q} \in \M_d(\C)$.}
\]
According to~\cite[Prop.~6.24]{AS17:Alice-Bob},
the matrix fluctuation statistic of the GUE matrix satisfies %
\[
\phi(\mtx{X}_{\gue}) \leq \phi_{\pm}(\mtx{X}_{\gue}) \leq 2 \sqrt{d}.
\]
The matrix variance, weak variance, and interaction energy admit the formulas
\[
\sigma^2(\mtx{X}_{\gue}) = d
\quad\text{and}\quad
\sigma_*^2(\mtx{X}_{\gue}) = 1
\quad\text{and}\quad
w(\mtx{X}_{\gue}) = 1.
\]
These expressions follow from direct calculation.

\subsection{Rectangular Gaussian matrices}
\label{sec:gauss-rect}

Corollary~\ref{cor:norm-main} invokes the self-adjoint dilation operator~\eqref{eqn:sa-dilation}
as a formal device to extend the comparison theorem (Theorem~\ref{thm:maxeig-main}) %
to rectangular matrices.
This section explains how to use the self-adjoint dilation to
adapt the definitions of a Gaussian matrix and its statistics to the rectangular case.

\subsubsection{First- and second-order moments}

In this section, we work in the matrix space $\F^{d_1 \times d_2}$,
treated as a \hilite{real} linear space, equipped with the real trace inner product
\[
\ip{\mtx{A}}{\mtx{B}}_{\Re} \coloneqq \Re \trace[\mtx{A}^* \mtx{B}]
\quad\text{for all $\mtx{A},\mtx{B} \in \F^{d_1 \times d_2}$.}
\]
For a rectangular random matrix $\mtx{W} \in \F^{d_1 \times d_2}$, we compute
its expectation entrywise: $\Expect[\mtx{W}] \in \F^{d_1 \times d_2}$.
The variance function of a rectangular matrix is defined as
\[
\Varo[\mtx{W}](\mtx{M})
	 \coloneqq \Var[ \ip{\mtx{W}}{\mtx{M}}_{\Re} ]
	= \tfrac{1}{4} \cdot \Varo[ \coll{H}(\mtx{W})](\coll{H}(\mtx{M}))
\quad\text{for all $\mtx{M} \in \F^{d_1 \times d_2}$.}
\]
Up to a constant, this is simply the variance function of the self-adjoint dilation~\eqref{eqn:sa-dilation}.
We have chosen the scaling so that the definitions of the two trace inner products are consistent.

A rectangular matrix $\mtx{X} \in \F^{d_1 \times d_2}$ is \term{Gaussian}
if and only if the one-dimensional linear marginal $\ip{\mtx{X}}{\mtx{M}}_{\Re}$ follows
a real-valued Gaussian distribution for every $\mtx{M} \in \F^{d_1 \times d_2}$.
A rectangular Gaussian matrix is also characterized by its expectation
and variance function.  For a fixed matrix $\mtx{\Delta} \in \F^{d_1 \times d_2}$
and a psd quadratic form $\mathsf{V} : \F^{d_1 \times d_2} \to \R_+$,
we write $\normal(\mtx{\Delta}, \mathsf{V})$ for the (unique) Gaussian distribution
on $\F^{d_1 \times d_2}$ with these statistics.

\subsubsection{Variance statistics}

Likewise, we define the other statistics of a rectangular matrix using the self-adjoint dilation.
Each of these statistics can also be expressed directly in terms of the original matrix.
Let $\mtx{X} \in \F^{d_1 \times d_2}$ be a rectangular Gaussian matrix.
The \term{rectangular matrix norm fluctuation} statistic is
\[
\phi_{\pm}(\mtx{X}) \coloneqq \phi(\coll{H}(\mtx{X}))
	= \Expect \norm{ \mtx{X} - \Expect \mtx{X} }.
\]
The \term{rectangular matrix variance} is
\[
\sigma^2(\mtx{X}) \coloneqq \sigma^2(\coll{H}(\mtx{X}))
	= \lnorm{ \Expect \big[ ( \mtx{X} - \Expect \mtx{X})( \mtx{X} - \Expect \mtx{X})^* \big] }
	\vee \lnorm{ \Expect\big[ ( \mtx{X} - \Expect \mtx{X})^*( \mtx{X} - \Expect \mtx{X}) \big] }.
\]
The \term{rectangular weak variance} is
\[
\sigma_*^2(\mtx{X}) \coloneqq \sigma_*^2(\coll{H}(\mtx{X}))
	= \sup\nolimits_{\norm{\vct{u}_1} = \norm{\vct{u}_2} = 1} \Var[ \Re(\vct{u}_1^* \mtx{X} \vct{u}_2) ]
	= \sup\nolimits_{\norm{\mtx{M}}_* = 1} \Varo[\mtx{X}](\mtx{M}).
\]
The \term{rectangular interaction energy} is
\[
w(\mtx{X}) \coloneqq w(\coll{H}(\mtx{X}))
	= 2 \sup\nolimits_{\fnorm{\mtx{M}} = 1} \Varo[ \mtx{X} ](\mtx{M}).
\]
All in all, the rectangular statistics are similar with the self-adjoint statistics,
but we must pay attention to constant factors.

\subsubsection{Example: A Gaussian matrix with iid entries}

As a short example, let us consider a real-valued $d_1 \times d_2$ Gaussian matrix with iid
standardized entries:
\[
\mtx{X}_{\rect} \coloneqq \sum_{j=1}^{d_1} \sum_{k=1}^{d_2} \gamma_{jk} \cdot \mathbf{E}_{jk}
	\in \R^{d_1 \times d_2}
	\quad\text{where $\gamma_{jk} \sim \normal_{\R}(0,1)$ iid.}
\]
The expectation and variance function are
\[
\Expect \mtx{X}_{\rect} = \mtx{0}
\quad\text{and}\quad
\Varo[\mtx{X}_{\rect}](\mtx{M}) = \fnormsq{\mtx{M}} 
\quad\text{for all $\mtx{M} \in \R^{d_1 \times d_2}$.}
\]
A classic bound due to Chevet~\cite[Thm.~7.3.1]{Ver25:High-Dimensional-Probability} %
controls the matrix norm fluctuation:
\[
\phi_{\pm}(\mtx{X}_{\rect}) \coloneqq \Expect \norm{ \mtx{X}_{\rect} - \Expect \mtx{X}_{\rect} }
	\leq \sqrt{d_1} + \sqrt{d_2}.
\]
Meanwhile, the variance statistics satisfy
\[
\sigma^2(\mtx{X}_{\rect}) = d_1 \vee d_2
\quad\text{and}\quad
\sigma_*^2(\mtx{X}_{\rect}) = 1
\quad\text{and}\quad
w(\mtx{X}_{\rect}) = 2.
\]
These expression follow from direct calculations.

\section{Applications}
\label{sec:applications}

The comparison theorems from the introduction provide tools
for analyzing various random matrices that arise in applications.
To showcase the maximum eigenvalue comparison (Theorem~\ref{thm:maxeig-main}),
we develop a simple example in spectral graph theory (Section~\ref{sec:rdm-graph-appl}).
Afterward, we use the spectral norm comparison (Corollary~\ref{cor:norm-main})
to obtain new guarantees for a problem in quantum information theory
(Section~\ref{sec:quantum-appl}) that involves exponentially large random matrices.
Last, we employ the minimum eigenvalue comparison
(Corollary~\ref{cor:mineig-main}) to treat problems in statistics (Section~\ref{sec:scov-appl})
and numerical linear algebra (Section~\ref{sec:sparsestack-appl});
these results rely %
on the fact that the comparison
only requires one-sided control on the summands.

\subsection{Random regular graph: Second eigenvalue bound}
\label{sec:rdm-graph-appl}

As a warmup, we treat a basic problem from spectral graph
theory involving the expansion properties of random regular graphs.
In this setting, the benefits of the comparison theorem (Theorem~\ref{thm:maxeig-main})
are limited, but similar arguments apply to more challenging problems
involving random Cayley graphs; see~\cite[Sec.~3.2.3]{BvH24:Universality-Sharp}.

Heuristically, an \term{expander} is a graph where each
set of vertices is adjacent to a much larger set of vertices.
Beginning from any vertex, we can reach
any other vertex after a modest number of steps.
Expanders are the subject
of a rich literature in group theory, combinatorics,
theoretical computer science, and other fields.
See~\cite{HLW06:Expander-Graphs,Kow19:Introduction-Expander}
for an overview.

One of the key examples of an expander is a random regular graph.
To be concrete, we introduce the \term{permutation model}
for a random regular graph on $n$ vertices with \hilite{even} degree $d = 2k$,
where $k \in \N$.
Let $\mtx{\Pi} \in \M_n(\R)$ be a uniformly random permutation matrix
acting on $\R^n$.  Construct the random self-adjoint matrix
\begin{equation} \label{eqn:rdm-reg-graph}
\mtx{Y} \coloneqq \sum_{i=1}^k \mtx{W}_i \in \Sym_n(\R)
\quad\text{where $\mtx{W}_i \sim \mtx{\Pi} + \mtx{\Pi}^\transp$ iid.}
\end{equation}
We interpret $\mtx{Y}$ %
as the adjacency matrix of the random graph.
Note that this random graph model %
is likely to have both
self-loops and duplicated edges.
Regardless, the graph is $d$-regular, and the vector of ones $\vct{1} \in \R^n$
is always an eigenvector of $\mtx{Y}$ with eigenvalue $d$.

We are interested in the \hilite{second-largest} eigenvalue
of the random matrix $\mtx{Y}$,
which describes the expansion properties of the graph~\cite[Sec.~2.3]{HLW06:Expander-Graphs}.
When the degree $d$ is fixed and the number $n$
of vertices increases, classic bounds for the
second eigenvalue $\lambda_2(\mtx{Y})$ assert that
\[
2\sqrt{d - 1} - o_n(1) \leq \lambda_2(\mtx{Y}) \leq 
2\sqrt{d - 1} + o_n(1)
\quad\text{with probability $1 - o_n(1)$.}
\]
The left-hand inequality is a relatively easy result due to
Alon \& Boppana~\cite[Sec.~5.2]{HLW06:Expander-Graphs},
while the right-hand inequality is a celebrated
and difficult theorem due to Friedman~\cite{Fri08:Proof-Alons}.

There has been ongoing interest in establishing quantitative
variants of Friedman's upper bound that address the case where both $d,n \to \infty$.
For recent advances, see the papers~\cite{BvH24:Universality-Sharp,CGV+26:New-Approach}.
The comparison theorem for the maximum eigenvalue (Theorem~\ref{thm:maxeig-main}) supports
a short, direct argument that yields sharper bounds in parts of the parameter regime.

\begin{theorem}[Random regular graph: Second eigenvalue] \label{thm:rdm-graph-2nd}
Consider the random regular graph model~\eqref{eqn:rdm-reg-graph}
with \hilite{even} degree $d$ and with $n \geq 2$ vertices.
In the parameter regime $\log^2 n \leq d \leq n$,
the random adjacency matrix $\mtx{Y}$ satisfies
\begin{equation} \label{eqn:rdm-graph-expect}
\Expect \lambda_2(\mtx{Y}) \leq 2 \sqrt{d-1} \cdot \left[ 1 + 
4 \cdot \left( \log^2( n) / d \right)^{1/4} \right].
\end{equation}
Furthermore, for $\log n \leq s \leq \sqrt{d}$,
\[
\Prob{ \lambda_2(\mtx{Y}) \geq 2\sqrt{d-1} \cdot \left[ 1 + 4 \cdot \left( s^2/d \right)^{1/4} \right] }
	\leq n \cdot \econst^{-s}.
\]
\end{theorem}

\noindent
The proof of Theorem~\ref{thm:rdm-graph-2nd} appears in Section~\ref{sec:rdm-graph-pf}.

For contrast, the maximum eigenvalue bound~\eqref{eqn:maxeig-bvh}
of Brailovskaya \& van Handel~\cite[Thm.~3.8]{BvH24:Universality-Sharp}
produces estimates of the form
\[
\Expect \lambda_2(\mtx{Y}) \leq 2 \sqrt{d-1} \cdot \left[ 1 + \mathrm{Const} \cdot \left(\log^4(n) / d \right)^{1/6} \right]
\quad\text{for $d \gg \log^4(n)$.}
\]
Theorem~\ref{thm:rdm-graph-2nd} improves uniformly on their bound while extending the range of legal parameters.  The paper~\cite{CGV+26:New-Approach} contains stronger bounds
in the regime $d \ll n^{1/6}$, but that approach requires more
detailed combinatorial information about the random permutation model.

\subsubsection{Proof of Theorem~\ref{thm:rdm-graph-2nd}}
\label{sec:rdm-graph-pf}

To access the second eigenvalue of the random adjacency matrix $\mtx{Y}$, %
we compress %
to the orthogonal complement $\vct{1}^\perp$ %
of the leading eigenvector.  Define
\[
\mtx{Y}^\perp \coloneqq \sum_{i=1}^k \mtx{W}_i^\perp
\quad\text{where}\quad
\mtx{W}_i^\perp \coloneqq \mtx{W}_i \, \vert_{\vct{1}^\perp}.
\]
The maximum eigenvalue of $\mtx{Y}^\perp$ agrees with the
second-largest eigenvalue of $\mtx{Y}$.
Each summand $\mtx{W}^\perp$ satisfies
$\lambda_{\max}(\mtx{W}^\perp) \leq \lambda_{\max}(\mtx{\Pi} + \mtx{\Pi}^\transp) \leq 2$,
so we may take
the upper bound statistic $R_+ \coloneqq 2$.

To activate the comparison theorem, we must calculate
the Gaussian proxy for the adjacency matrix.
It is easiest to start with the expectation and variance
function of the compression $\mtx{\Pi}^\perp$ of a single
random permutation $\mtx{\Pi} \in \M_n(\R)$.
As for the expectation, note that
\[
\Expect \mtx{\Pi}^\perp
	= (\Expect \mtx{\Pi})^{\perp}
	= ( n^{-1} \vct{11}^\transp )^{\perp} = \mtx{0}.
\]
To develop the variance function, fix a self-adjoint matrix
$\mtx{M}^\perp \coloneqq [m_{ij}]$ with centered rows and columns.
Writing $\pi(i)$ for the nonzero index in row $i$ of the permutation $\mtx{\Pi}$,
we find that
\begin{align*}
\Var[ \ip{ \mtx{\Pi}^\perp }{ \mtx{M}^\perp } ]
	&= \Expect[ \ip{ \mtx{\Pi} }{ \mtx{M}^\perp }^2 ]
	= \Expect \left( \sum_{i=1}^n m_{i \pi(i)} \right)^2
	= \sum_{i,j = 1}^n \Expect[ m_{i\pi(i)} m_{j \pi(j)} ] \\
	&= \sum_{i=1}^n \Expect[ m_{i\pi(i)}^2 ]
	+ \sum_{i \neq j} \Expect[ m_{i\pi(i)} \Expect[ m_{j\pi(j)} \condbar \pi(i) ] ] \\
	&= \frac{1}{n} \sum_{i,j=1}^n m_{ij}^2
	+ \frac{1}{n(n-1)} \sum_{i,j=1}^n m_{ij}^2
	= \frac{1}{n-1} \sum_{i,j = 1}^n m_{ij}^2.
\end{align*}
The first relation depends on the facts that $\Expect \mtx{\Pi}^\perp = \mtx{0}$
and that the compression operator is an orthogonal projector on $\M_n(\R)$.
For the conditional expectation, we exploited the fact
\[
\sum_{j \neq i, \ell \neq \pi(i)} m_{j\ell}
	= \sum_{j \neq i} - m_{j \pi(i)} = m_{i \pi(i)}.
\]
This point follows because each row and column of $\mtx{M}^\perp$ sums to zero.

By elaborating on the latter calculation, we quickly determine the
expectation and variance function of each self-adjoint summand
$\mtx{W}^\perp = (\mtx{\Pi} + \mtx{\Pi}^\transp)^\perp$.
Indeed,
\[
\Expect \mtx{W}^\perp = \mtx{0}
\quad\text{and}\quad
\Var[ \ip{ \mtx{W}^\perp }{ \mtx{M}^\perp } ]
	= \frac{2}{n-1} \cdot 2 \fnorm{ \mtx{M}^\perp }^2.
\]
By inspection, we recognize that the Gaussian proxy for each summand $\mtx{W}$ is the scaled
GOE matrix $\sqrt{2/(n-1)} \mtx{X}_{\goe}$ acting on the $(n-1)$-dimensional subspace $\vct{1}^\perp$.

Since the random matrix $\mtx{Y}^\perp$ is the sum of $k = d/2$ iid copies of $\mtx{W}^\perp$,
these calculations imply that its Gaussian proxy is
\[
\mtx{Z}^{\perp} \coloneqq %
	\sqrt{\frac{d}{n-1}} \cdot \mtx{X}_{\goe}.
\]
The statistics of the Gaussian proxy are a matter of record (Section~\ref{sec:goe}):
\[
\phi(\mtx{Z}^{\perp}) = \Expect \lambda_{\max}(\mtx{Z}^{\perp}) \leq 2 \sqrt{d}
\quad\text{and}\quad
\sigma_*^2(\mtx{Z}^{\perp}) \leq \frac{2d}{n-1}.
\]
With this information at hand, we can invoke Theorem~\ref{thm:maxeig-main} with $R_+ = 2$
to
arrive at the bound
\[
\Expect \lambda_{\max}(\mtx{Y}^\perp) \leq 2\sqrt{d} + \sqrt{\left(\tfrac{4}{3} \sqrt{d} + 2d/(n-1)\right) \cdot 2\log(n-1)}
	+ \tfrac{2}{3} \log(n-1).
\]
Furthermore, for $s \geq \log(n-1)$,
\[
\Prob{ \lambda_{\max}(\mtx{Y}^\perp) \geq 2 \sqrt{d} + \sqrt{\left(\tfrac{4}{3} \sqrt{d} + 2d/(n-1)\right) \cdot 2s} + \tfrac{2}{3} s }
	\leq (n-1) \cdot \econst^{-s}.
\]
Theorem~\ref{thm:rdm-graph-2nd} follows after some simplifications.
\hfill\qed

\subsection{The random Pauli model: Spectral edges}
\label{sec:quantum-appl}

This section employs the spectral-norm comparison theorem (Corollary~\ref{cor:norm-main})
to treat a problem in quantum information theory that involves
an exponentially large random matrix.  Since our results have
a weaker dependence on the dimension of the random matrix
than previous work, we reap significant benefits. %

Can we design random matrices whose spectral statistics mimic
the GUE model, yet have more structure and less randomness?
This question arises %
on the quest for a possible example of quantum supremacy~\cite{CDB+24:Sparse-Random}.
We study the random Pauli model, proposed by Chen et al.~\cite{CDB+24:Sparse-Random},
and we %
prove that
the edges of its spectral distribution compare with those of the GUE.

First, recall the definition of the \term{Pauli matrices}:
\begin{equation} \label{eqn:pauli}
\mtx{H}_0 \coloneqq \begin{bmatrix} 1 & \phantom{+}0 \\ 0 & \phantom{+}1 \end{bmatrix}; \quad
\mtx{H}_1 \coloneqq \begin{bmatrix} 0 & \phantom{+}1 \\ 1 & \phantom{+}0 \end{bmatrix}; \quad
\mtx{H}_2 \coloneqq \begin{bmatrix} 0 & -\iunit \\ \iunit & \phantom{+}0 \end{bmatrix}; \quad
\mtx{H}_3 \coloneqq \begin{bmatrix} 1 & \phantom{+}0 \\ 0 & -1 \end{bmatrix}.
\end{equation}
Each of these matrices is self-adjoint and unitary, and the collection forms an
orthogonal basis for the real-linear space $\Sym_2(\C)$ equipped with the trace
inner product.
For a parameter $n \in \N$, we can construct a self-adjoint, unitary matrix with
dimension $N \coloneqq 2^n$ by taking a Kronecker product of $n$ Pauli matrices:
\[
\mtx{H}_{I} \coloneqq \mtx{H}_{i_1} \otimes \cdots \otimes \mtx{H}_{i_n} \in \Sym_N(\C)
\quad\text{where}\quad I \coloneqq (i_1, \dots, i_n) \in \{0,1,2,3\}^n.
\]
Each matrix $\mtx{H}_I$ is self-adjoint and unitary.
Its spectral norm $\norm{\mtx{H}_I} = 1$, while its Frobenius norm $\fnorm{\mtx{H}_I} = \sqrt{N}$.
Together, these $4^n$ matrices compose an orthogonal basis for $\Sym_N(\C)$.

For a parameter $k \in \N$, the \term{random Pauli model} is a random matrix of the form
\begin{equation} \label{eqn:rdm-pauli}
\mtx{Y} \coloneqq \frac{1}{\sqrt{k}} \sum_{i=1}^k \eps_i \mtx{H}_{J_i} \in \Sym_N(\C)
\quad\text{where}\quad
\begin{aligned}
&\text{$\eps_i \sim \uniform\{\pm 1\}$ iid};\\
&\text{$J_i \sim \uniform\{0,1,2,3\}^n$ iid}.
\end{aligned}
\end{equation}
The Rademacher variables are independent from the random indices,
and each summand has spectral norm equal to one.
Observe that we can construct this random matrix using
only $(2n + 1) k$ random bits, even though its dimension
$N = 2^n$ is exponential in $n$.

We can easily compute the expectation and variance function of the random
Pauli model.  It is straightforward to check (see below) that
\[
\Expect \mtx{Y} = \mtx{0}
\quad\text{and}\quad
\Varo[\mtx{Y}](\mtx{M}) = N^{-1}\fnormsq{\mtx{M}}
\quad\text{for each $\mtx{M} \in \Sym_N(\C)$.}
\]
Therefore, we anticipate that the spectral statistics of $\mtx{Y}$
should be comparable with those of the normalized GUE matrix
$N^{-1/2} \mtx{X}_{\gue} \in \Sym_N(\C)$.
The question is how many summands $k$ are sufficient to
guarantee this affinity.

\begin{theorem}[Random Pauli model] \label{thm:rdm-pauli}
Choose a parameter $n \in \N$, and set $N \coloneqq 2^n$.
Draw a random matrix $\mtx{Y} \in \Sym_N(\C)$ from
the random Pauli model~\eqref{eqn:rdm-pauli} with $k$ summands,
and let $\mtx{Z} \coloneqq N^{-1/2} \mtx{X}_{\gue} \in \Sym_N(\C)$
be a normalized GUE matrix. %
For parameters $\alpha, p \in (0, 1)$, suppose that the number $k$ of summands
satisfies
\[
\left[ \frac{ (n+1) + \log(1/p) }{\alpha^2} \right]^2 \leq k \leq \tfrac{1}{9} N^2.
\]
Then the spectral edges of the two random matrices compare in the sense that
\[
\Prob{ \norm{ \mtx{Y} } \geq (1 + \alpha) \cdot \Expect \norm{\mtx{Z}} }
	\leq p \quad\text{and}\quad %
\Expect \norm{ \mtx{Y} } \leq (1 + \alpha) \cdot \Expect \norm{\mtx{Z}}.
\]
Since $\mtx{Y}$ and $\mtx{Z}$ have symmetric distributions,
the same probability inequalities hold
for both the minimum and maximum eigenvalue.
\end{theorem}

\noindent
The proof of Theorem~\ref{thm:rdm-pauli} appears in the upcoming subsection.

Our result %
improves
over the bounds that follow from existing techniques,
but it is not clear whether it reaches the optimal scaling. %
Roughly, we have demonstrated that $k \gtrsim n^2 \alpha^{-4}$ summands are enough
to obtain a one-sided relative error less than $\alpha$.
The original result for the random Pauli model,
due to Chen et al.~\cite[Thm.~III.1]{CDB+24:Sparse-Random},
requires $k \gtrsim n^3 \alpha^{-4}$ to achieve the same
level of control.
Meanwhile, if we invoke the maximum eigenvalue estimate~\eqref{eqn:maxeig-bvh}
of Brailovskaya \& van Handel~\cite[Cor.~2.7]{BvH24:Universality-Sharp},
we arrive at the bound $k \gtrsim n^4 \alpha^{-6}$.

Under similar assumptions, the mean spectral distributions of
the two random matrices $\mtx{Y}$ and $\mtx{Z}$ are also comparable;
for example, one may apply the universality result~\cite[Thm.~2.10]{BvH24:Universality-Sharp}.
Thus, the spectral norm of $\mtx{Y}$ is unlikely
to be much smaller than the spectral norm of $\mtx{Z}$.
These results are outside the %
scope of this paper,
so we terminate the discussion here.

\subsubsection{Proof of Theorem~\ref{thm:rdm-pauli}} %

This result is a consequence of Corollary~\ref{cor:norm-main}.
To begin, let us confirm the calculation of the second-order statistics of the
random Pauli matrix $\mtx{Y}$, defined in~\eqref{eqn:rdm-pauli}.
Each summand takes the form
\[
\mtx{W} \coloneqq \eps \cdot \mtx{H}_{J}
\quad\text{where}\quad
\begin{aligned}
&\text{$\eps \sim \uniform\{\pm 1\}$;} \\
&\text{$J \sim \uniform\{0,1,2,3\}^n$.}
\end{aligned}
\]
The Rademacher random variables are symmetric, so $\Expect \mtx{W} = \mtx{0}$.
For a fixed $\mtx{M} \in \Sym_N(\C)$,
\begin{align*}
\Varo[\mtx{W}](\mtx{M}) = \Expect[ \ip{ \mtx{W} }{ \mtx{M} }^2 ]
	= \sum_{I \in \{0,1,2,3\}^n} \ip{ \mtx{H}_{I} }{ \mtx{M} }^2 \cdot N^{-2}
	= N^{-1} \fnormsq{\mtx{M}}.
\end{align*}
Indeed, there are $4^{n} = N^2$ equally likely choices for the random index $J$, and the family
$( N^{-1/2} \mtx{H}_I : I \in \{0,1,2,3\}^n )$ composes an orthonormal basis for $\Sym_N(\C)$.

Since the random Pauli model $\mtx{Y}$ is defined by~\eqref{eqn:rdm-pauli}, its statistics satisfy
\[
\Expect \mtx{Y} = \mtx{0}
\quad\text{and}\quad
\Varo[\mtx{Y}](\mtx{M}) = N^{-1} \fnormsq{\mtx{M}}.
\]
The Gaussian matrix $\mtx{Z} \coloneqq N^{-1/2} \mtx{X}_{\gue} \in \Sym_N(\C)$ has
the same second-order statistics as $\mtx{Y}$.  From the discussion in Section~\ref{sec:gue},
we find that
\[
\phi_{\pm}(\mtx{Z}) = \Expect \norm{\mtx{Z}}
	\leq 2
\quad\text{and}\quad
\sigma_*^2(\mtx{Z}) = N^{-1}.
\]
Moreover, each of the $k$ summands in $\mtx{Y}$ has spectral norm bounded
above by $R_{\pm} = k^{-1/2}$.
The comparison theorem (Corollary~\ref{cor:norm-main}) yields the bound
\begin{gather*}
\Prob{ \norm{\mtx{Y}} \geq \Expect \norm{\mtx{Z}} + \sqrt{ \left(\tfrac{2}{3}k^{-1/2} + N^{-1}\right) \cdot 2s } + \tfrac{1}{3} k^{-1/2} s  } \leq 2N \cdot \econst^{-s}.
\end{gather*}
To achieve failure probability $p \in (0,1)$, set the tail level $s = \log(2N/p)$.
For an error tolerance $\alpha \in (0,1)$,
we can submerge the total of the two error terms
below the level $\alpha \cdot \Expect \norm{\mtx{Z}} \approx 2 \alpha$
by taking $k \geq \alpha^{-4} \log^2(2N/p)$ and requiring that $3\sqrt{k} \leq N$.
The expectation bound holds with the same choice of the number $k$.
\hfill\qed

\subsection{Sample covariance matrices: Fourth-moment theorem}
\label{sec:scov-appl}

In this section, we use the comparison theorem for the minimum
eigenvalue (Corollary~\ref{cor:mineig-main}) to analyze the sample
covariance matrix of a random vector whose marginals have uniformly
bounded fourth moments.  This result depends on %
the fact that
we only need one-sided control on the summands.

Consider a centered random vector $\vct{w} \in \R^d$. %
For simplicity, assume that its population covariance $\mtx{K} \coloneqq \Expect[\vct{ww}^\transp]$ has full rank.  Given $n$ iid samples of the random vector, the sample covariance matrix
is defined as
\begin{equation} \label{eqn:scov-vec}
\widehat{\mtx{K}}_n \coloneqq \frac{1}{n} \sum_{i=1}^n \vct{w}_i \vct{w}_i^\transp
\quad\text{where $\vct{w}_i \sim \vct{w}$ iid.}
\end{equation}
Our goal is to find conditions under which the sample covariance
reliably detects a constant proportion of the variance in each
direction.  That is, for a parameter $\eps \in (0, 1)$,
\[
\vct{a}^\transp \widehat{\mtx{K}}_n \vct{a}
	\geq (1 - \eps) \cdot \vct{a}^\transp \mtx{K} \vct{a}
	\quad\text{for all $\vct{a} \in \R^d$.}
\]
We will establish the following theorem.

\begin{theorem}[Sample covariance: Fourth-moment theorem] \label{thm:scov-var}
Assume that $d \geq 2$.  Let $\vct{w} \in \R^d$ be a centered random vector with
covariance matrix $\mtx{K}$ and with four finite moments.
For a constant $\beta \geq 1$, assume that
\begin{equation} \label{eqn:scov-mom}
\Var[ \ip{\vct{w}}{\vct{a}}^2 ]
	\leq \beta^2 \cdot ( \Var[ \ip{\vct{w}}{\vct{a}} ] )^2
	\quad\text{for all $\vct{a} \in \R^d$.}
\end{equation}
For parameters $\eps, p \in (0,1)$, suppose that the
number $n$ of samples satisfies
\begin{equation} \label{eqn:scov-samples}
n \geq \frac{26 \beta^2 (d \vee \log(d/p))}{\eps^2}.
\end{equation}
Then, with probability at least $1 - p$, the sample
covariance matrix $\widehat{\mtx{K}}_n$ defined in~\eqref{eqn:scov-vec}
satisfies
\begin{equation} \label{eqn:scov-bd}
\vct{a}^\transp \widehat{\mtx{K}}_n \vct{a}
	\geq (1 - \eps) \cdot \vct{a}^\transp \mtx{K} \vct{a}
	\quad\text{for all $\vct{a} \in \R^d$.}
\end{equation}
\end{theorem}

\noindent
The proof of Theorem~\ref{thm:scov-var} appears in the next subsection.

This statement implies a fundamental fact from high-dimensional
statistics due to Oliveira~\cite[Thm.~1]{Oli16:Lower-Tail}.
The companion paper~\cite[Thm.~5.2]{Tro26:Comparison-Theorems-Minimum}
contains another proof based on Fact~\ref{fact:mineig-psd}.
Our purpose here is to demonstrate that the new comparison theorem
(Corollary~\ref{cor:mineig-main}) is also powerful enough to describe
the same phenomenon.
We remark that prior %
results place a slightly stronger hypothesis
on the random vector:
\[
\Expect[ \ip{\vct{w}}{\vct{a}}^4 ]
	\leq \alpha^2 \cdot ( \Var[ \ip{\vct{w}}{\vct{a}} ] )^2
	\quad\text{for all $\vct{a} \in \R^d$.}
\]
Although the assumption~\eqref{eqn:scov-mom} is formally weaker,
the two results are comparable because $\beta^2 \geq 1/3$
when $d \geq 2$; see \cite[Prob.~1.23]{Tro20:Randomized-Algorithms-LN}. %

\subsubsection{Proof of Theorem~\ref{thm:scov-var}}

Without loss of generality,
we may assume that the random vector $\vct{w}$
is isotropic: $\Expect[\vct{ww}^\transp] = \Id_d$.
Indeed, the moment condition~\eqref{eqn:scov-mom}
and the conclusion~\eqref{eqn:scov-bd}
are both invariant under invertible linear
transformations.  

The sample covariance matrix $\widehat{\mtx{K}}_n$ is
defined by~\eqref{eqn:scov-vec}.
As the dimension $d \geq 2$, the lower bound statistic satisfies %
\[
n^{-1} \lambda_{\min}(\vct{ww}^\transp - \Expect[\vct{ww}^\transp])
	= n^{-1} \lambda_{\min}( - \Id_d)
	= - n^{-1} \eqqcolon - R_-.
\]
Since the sample covariance matrix is isotropic,
the Gaussian proxy $\mtx{Z} \in \Sym_d(\R)$
has expectation $\Expect[\mtx{Z}] = \Expect[ \widehat{\mtx{K}}_n ] = \Id_d$.
For any \hilite{unit} vector $\vct{u} \in \R^d$,
the variance function of $\mtx{Z}$ satisfies %
\begin{align*}
\Varo[\mtx{Z}](\vct{uu}^\transp)
	&= \Varo[\widehat{\mtx{K}}_n](\vct{uu}^\transp)
	= n^{-1} \Varo[ \vct{ww}^\transp ]( \vct{uu}^\transp ) \\
	&= n^{-1} \Var[ \ip{\vct{w}}{\vct{u}}^2 ]
	\leq (\beta^2 / n) \cdot \Var[ \ip{\vct{w}}{\vct{u}} ]
	= \beta^2 / n.
\end{align*}
By a direct application of an easy net argument~\cite[Lem.~5.3]{Tro26:Comparison-Theorems-Minimum},
the statistics of $\mtx{Z}$ satisfy
\[
\phi(\mtx{Z}) \leq \Expect \norm{ \mtx{Z} - \Expect \mtx{Z} }
	\leq \sqrt{12 \beta^2 d/n}
\quad\text{and}\quad
\sigma_*^2(\mtx{Z}) \leq \beta^2 / n.
\]
In particular,
\[
\Expect \lambda_{\min}(\mtx{Z}) \geq \lambda_{\min}(\Expect \mtx{Z}) - \Expect \lambda_{\max}(\mtx{Z} - \Expect \mtx{Z})
	\geq 1 - \sqrt{12 \beta^2 d/n}.
\]
The comparison theorem for the minimum eigenvalue (Corollary~\ref{cor:mineig-main})
delivers %
\[
\Prob{ \lambda_{\min}(\widehat{\mtx{K}}_n) \geq 1 - \sqrt{12 \beta^2 d/n}
	- \sqrt{\left(\tfrac{1}{3} n^{-1}\sqrt{12 \beta^2 d/n} + \beta^2/n \right) \cdot 2s}
	- \tfrac{1}{3} n^{-1} s }
	\leq d \cdot \econst^{-s}.
\]
To achieve failure probability $p$, set $s \coloneqq \log(d/p)$.
When the sample complexity $n$ is chosen via~\eqref{eqn:scov-samples},
the event controls the minimum eigenvalue at the level
$\lambda_{\min}(\widehat{\mtx{K}}_n) \geq 1 - \eps$.
\hfill\qed

\subsection{Sparse dimension reduction: Injectivity}
\label{sec:sparsestack-appl}

Our last application develops a substantial new result on the injectivity 
properties of sparse randomized dimension reduction maps,
confirming part of a conjecture posed by
Nelson \& Nguyen~\cite{NN13:OSNAP-Faster,NN14:Lower-Bounds}.
Our analysis depends on the comparison theorem for the
minimum eigenvalue (Corollary~\ref{cor:mineig-main}).
Its success stems from %
the one-sided hypothesis on the
summands and the %
form of the Gaussian comparison model.

\subsubsection{Subspace embeddings}

Many contemporary algorithms for matrix computations depend
on a primitive called a {subspace embedding}, a linear map that
reduces the dimension of the ambient space while preserving the
geometry of a fixed subspace~\cite{MT20:Randomized-Numerical,KT23:Randomized-Matrix}.

\begin{definition}[Subspace embedding]
Choose a matrix $\mtx{U} \in \C^{n \times d}$ with orthonormal
columns that span a subspace with dimension $d$.
A \term{subspace embedding} for $\range(\mtx{U})$
with \term{embedding dimension} $k$ is a linear map $\mtx{\Phi} : \C^n \to \C^k$
that satisfies the bounds
\begin{equation} \label{eqn:subspace-embed}
0 < 1 - \alpha \leq \sigma_{\min}^2(\mtx{\Phi} \mtx{U})
\quad\text{and}\quad
\sigma_{\max}^2(\mtx{\Phi} \mtx{U}) \leq 1 + \beta.
\end{equation}
In this context, $\sigma_{\min}$ and $\sigma_{\max}$ denote
the smallest and largest singular values of a linear map.
The parameter $\alpha \in [0,1)$ is called the \term{lower distortion}
of the subspace embedding, while the parameter $\beta \geq 0$ is called
the \term{upper distortion}.
\end{definition}

The left-hand inequality in~\eqref{eqn:subspace-embed} is the critical one because it ensures
that the matrix $\mtx{\Phi}$ is an injection---it does not annihilate any vector in the subspace.
To obtain an injection, it is necessary that the embedding dimension exceeds the subspace dimension ($k \geq d$).
In many situations, the target subspace $\range(\mtx{U})$ is not
known in advance, but we can still achieve the subspace embedding
property with high probability for any fixed subspace of dimension $d$
by choosing %
$\mtx{\Phi}$ \hilite{at random}.
In this case, we say that the subspace embedding is \term{oblivious}.

Constructions of (oblivious) subspace embeddings attempt to minimize the
embedding dimension, while controlling the distortion parameters.
We also seek subspace embeddings that have favorable computational properties.
In particular, the cost of forming a matrix--vector product %
$\vct{u} \mapsto \mtx{\Phi} \vct{u}$ should be as small as possible.

\subsubsection{SparseStack and CountSketch}

To achieve the computational goals, we can design random
subspace embeddings that are very {sparse}.
One attractive example is the SparseStack, a sparse dimension reduction map
first proposed by Kane \& Nelson~\cite[Fig.~1(c)]{KN14:Sparser-Johnson-Lindenstrauss}.
There is extensive empirical evidence that SparseStack matrices are effective tools for
matrix computations; for example, see~\cite{CEMT25:Faster-Linear}.
Nevertheless, the analysis has presented an ongoing challenge.

The SparseStack matrix is constructed from a more basic dimension reduction map,
called {CountSketch}~\cite{CCF04:Finding-Frequent}.
For an embedding dimension $b \in \N$, the (complex-valued) \term{CountSketch matrix} $\mtx{S} : \C^n \to \C^b$
takes the form
\begin{equation} \label{eqn:countsketch}
\mtx{S} \coloneqq \sum_{j=1}^n \eps_j \mathbf{E}_{I_j j}
	\in \C^{b \times n}
\quad\text{where}\quad
\begin{aligned}
&\text{$\eps_j \sim \uniform\{ z \in \C : \abs{z} = 1\}$ iid;} \\
&\text{$I_j \sim \uniform\{1, \dots,b\}$ iid.}
\end{aligned}
\end{equation}
As usual, $\mathbf{E}_{ij} \in \C^{b \times n}$ is a standard basis matrix, %
and the Steinhaus variables $\eps_j$ are independent from the random indices $I_j$.
By itself, the CountSketch matrix is not an effective subspace embedding.
Indeed, to obtain an injection with high probability, it is necessary that
the embedding dimension $b \gtrsim d^2$, which is far from the ideal ($b \approx d$).

The \term{SparseStack matrix} combines several CountSketch matrices
to achieve superior dimension reduction properties.
Choose a block size $b \in \N$ and a column sparsity parameter $\zeta \in \N$,
and set the embedding dimension $k \coloneqq \zeta b$.
The SparseStack dimension reduction matrix $\mtx{\Phi} : \C^n \to \C^k$
is obtained by stacking $\zeta$ iid copies of a scaled CountSketch matrix~\eqref{eqn:countsketch}:
\begin{equation} \label{eqn:sparsestack}
\mtx{\Phi} \coloneqq \frac{1}{\sqrt{\zeta}} \begin{bmatrix} & \mtx{S}_1 & \\ & \vdots & \\ & \mtx{S}_\zeta & \end{bmatrix}
	\in \C^{k \times n}
\quad\text{where $\mtx{S}_i \sim \mtx{S} \in \C^{b \times n}$ iid.}
\end{equation}
When the column sparsity $\zeta \ll k$,
the cost of a matrix--vector product
$\vct{u} \mapsto \mtx{\Phi} \vct{u}$
is at most $\mathcal{O}(\zeta n)$ arithmetic operations,
in contrast to the $\mathcal{O}(kn)$ cost when the matrix
is dense and unstructured.

\subsubsection{The Nelson--Nguyen conjecture}

For the SparseStack matrix to serve as an oblivious subspace embedding
with specified distortion parameters, %
what are the minimal column sparsity $\zeta$ and embedding dimension $k$ possible?
In 2013, Nelson \& Nguyen~\cite{NN13:OSNAP-Faster,NN14:Lower-Bounds}
broadcast a conjecture that speaks to this question.

\begin{conjecture}[Nelson--Nguyen] \label{conj:nn}
Fix a subspace dimension $d$ and distortion parameters $\alpha = \beta \in (0,1)$.
There exists a construction of a random sparse matrix $\mtx{\Phi} : \R^n \to \R^k$
with at most $\zeta$ nonzero entries per column and with embedding dimension
$k$ that satisfy
\[
\zeta \leq \mathrm{Const} \cdot \alpha^{-1} \log d
\quad\text{and}\quad
k \leq \mathrm{Const} \cdot \alpha^{-2} d.
\]
Moreover, for any fixed subspace of dimension $d$, with probability at least $1 - d^{-1}$,
this random sparse matrix $\mtx{\Phi}$ serves as a subspace embedding
with distortion parameters $\alpha, \beta$.
\end{conjecture}

So far, Conjecture~\ref{conj:nn} has resisted all efforts.
The next theorem provides the first %
proof
that a SparseStack matrix with the conjectured parameters
satisfies the \hilite{lower} distortion bound.
This result does not address the upper distortion, %
which is less important in practice~\cite{CEMT25:Faster-Linear}.

\begin{theorem}[SparseStack: Injectivity] \label{thm:sparsestack}
Choose an arbitrary orthonormal matrix $\mtx{U} \in \C^{n \times d}$,
a lower distortion parameter $\alpha \in (0,1)$,
and a failure probability $p \in (0,1]$.
Construct the SparseStack matrix $\mtx{\Phi} : \C^n \to \C^k$,
defined in~\eqref{eqn:countsketch} and~\eqref{eqn:sparsestack},
with parameters
\[
\zeta \geq 6 \cdot \alpha^{-1} \log(d/p)
\quad\text{and}\quad
k \geq 16 \cdot \alpha^{-2} (d \vee \log(d/p)).
\]
Then the matrix $\mtx{\Phi}$ acts as an injection on $\range(\mtx{U})$
with high probability and on average:
\[
\Prob{ \sigma^2_{\min}(\mtx{\Phi} \mtx{U}) > 1 - \alpha }
	\geq 1 - p
\quad\text{and}\quad
\Expect \sigma^2_{\min}(\mtx{\Phi} \mtx{U}) \geq 1 - \alpha.
\]
\end{theorem}

\noindent
The proof of Theorem~\ref{thm:sparsestack} appears in the next subsection.

The author's recent papers~\cite{Tro26:Comparison-Theorems-Minimum,CEMT25:Faster-Linear}
use the psd comparison theorem (Fact~\ref{fact:mineig-psd})
to show that the SparseStack matrix achieves
the \hilite{lower} distortion bound under the stricter assumptions that
$\zeta \gtrsim \alpha^{-2} \log d$ and $k \gtrsim \alpha^{-2} d$.
We can surpass the earlier results because Corollary~\ref{cor:mineig-main}
employs a more refined comparison model.
The lengthy analysis in~\cite{CDD26:Optimal-Subspace} obtains bounds
on both the \hilite{lower and upper} distortion,
but that approach falls short of the desired parameter scalings
(by sub-logarithmic factors).

\subsubsection{Proof of Theorem~\ref{thm:sparsestack}}

Introduce the random self-adjoint matrix
\[
\mtx{Y} \coloneqq \mtx{U}^* \mtx{\Phi}^* \mtx{\Phi} \mtx{U}
	= \frac{1}{\zeta} \sum_{i = 1}^\zeta \mtx{U}^* \mtx{S}_i^* \mtx{S}_i \mtx{U}
	\eqqcolon \frac{1}{\zeta} \sum_{i=1}^\zeta \mtx{W}_i
	\in \Sym_d(\C).
\]
The summands $\mtx{W}_i$ are iid copies of the random matrix $\mtx{W} \coloneqq \mtx{U}^*\mtx{S}^* \mtx{S} \mtx{U}$,
where $\mtx{S} \in \C^{b \times n}$ is a complex CountSketch matrix~\eqref{eqn:countsketch} with block size $b$.
The minimum eigenvalue of $\mtx{Y}$ controls the lower distortion of the embedding:
\[
\lambda_{\min}(\mtx{Y}) = \sigma_{\min}^2(\mtx{\Phi} \mtx{U}).
\]
Thus, we can prove the theorem by developing lower bounds
for $\lambda_{\min}(\mtx{Y})$.

Let us begin with the expectation of an individual summand $\mtx{W}$.
A short argument confirms that the CountSketch matrix is isotropic:
\begin{equation} \label{eqn:countsketch-iso}
\Expect[ \mtx{S}^* \mtx{S} ] = \sum_{i,j=1}^n \Expect[ \eps_i^* \eps_j ] \cdot \Expect[ \delta_{I_i I_j} ] \cdot \mathbf{E}_{ij}
	= \Id_n.
\end{equation}
Indeed, only paired indices ($i = j$) survive in the expectation.
The symbol ${}^*$ denotes the complex conjugate, while $\delta_{\cdot\cdot}$ is the Kronecker delta.
Using~\eqref{eqn:countsketch-iso},
it is straightforward to check that each summand is isotropic:
\[
\Expect \mtx{W} = \mtx{U}^* \cdot \Expect[ \mtx{S}^* \mtx{S} ] \cdot \mtx{U}
	= \mtx{U}^* \mtx{U} = \Id_d.
\]
Since each summand $\zeta^{-1} \mtx{W}$ is also psd,
it satisfies the bound %
\[
\zeta^{-1} \lambda_{\min}(\mtx{W} - \Expect \mtx{W})
	\geq \zeta^{-1} \lambda_{\min}( - \Id_d)
	= - \zeta^{-1}.
\]
Therefore, we can set the lower bound parameter $R_- \coloneqq \zeta^{-1}$.

To obtain the variance function of the summand $\mtx{W}$,
fix a self-adjoint matrix $\mtx{M} \in \Sym_d(\C)$,
and define the transformed matrix
$\widetilde{\mtx{M}} \coloneqq \mtx{U} \mtx{M} \mtx{U}^* \in \Sym_n(\C)$.
Then
\begin{align*}
\Var[ \mtx{W} ](\mtx{M})
	&= \Expect \labs{ \ip{ \mtx{W} - \Id_d }{ \mtx{M} } }^2
	= \Expect \labs{ \ip{ \mtx{S}^*\mtx{S} - \Id_n }{ \smash{\widetilde{\mtx{M}}} } }^2
	= \Expect \labs{ \sum_{i \neq j} \eps_i^* \eps_j \delta_{I_i I_j} \widetilde{m}_{ij} }^2 \\
	&= \sum_{i \neq j, i' \neq j'} \Expect[ \eps_i^* \eps_j \eps_{i'} \eps_{j'}^* ] \cdot \Expect[ \delta_{I_i I_j} \delta_{I_{i'} I_{j'}}] \cdot \widetilde{m}_{ij} \widetilde{m}_{i'j'}^* \\
	&= b^{-1} \sum_{i \neq j } \abssq{\widetilde{m}_{ij}}
	\leq b^{-1} \fnormsq{ \widetilde{\mtx{M}} }
	= b^{-1} \fnormsq{\mtx{M}}.
\end{align*}
Indeed, the only summands with nonzero expectation have paired indices ($i' = i$ and $j' = j$).
We recognize that the Gaussian matrix
\[
\mtx{X} \coloneqq \Id_d + b^{-1/2} \mtx{X}_{\gue} \in \Sym_d(\C) %
\]
has the same expectation and greater variance function than the Gaussian
distribution $\normal(\Expect[\mtx{W}], \Varo[\mtx{W}])$;
cf.~\eqref{eqn:gauss-proxy-dom}.
We are exploiting the fact that Gaussian
matrix statistics increase with the variance function.

Altogether, for the random matrix $\mtx{Y} = \zeta^{-1} \sum_{i=1}^\zeta \mtx{W}_i$,
the Gaussian proxy described in~Corollary~\ref{cor:mineig-main} is dominated by the Gaussian matrix
\[
\mtx{Z} \coloneqq \frac{1}{\zeta} \sum_{i=1}^\zeta \mtx{X}_i \sim \Id_d + k^{-1/2} \mtx{X}_{\gue} %
\quad\text{where $\mtx{X}_i \sim \mtx{X}$ iid.}
\]
Recall that the embedding dimension $k = \zeta b$.
The statistics of the Gaussian matrix $\mtx{Z}$ are familiar to us (Section~\ref{sec:gue}):
\[
\Expect \lambda_{\min}(\mtx{Z}) \geq 1 - 2 \sqrt{d/k};
\quad %
\phi(\mtx{Z}) = \Expect \lambda_{\max}(\mtx{Z} - \Expect \mtx{Z}) \leq 2 \sqrt{d/k};
\quad %
\sigma_*^2(\mtx{Z}) \leq k^{-1}.
\]
We may now activate the comparison theorem for the minimum eigenvalue (Corollary~\ref{cor:mineig-main}): %
\[
\Prob{ \lambda_{\min}(\mtx{Y}) \leq 1 - \left[ 2 \sqrt{d/k} + \sqrt{ \left( \tfrac{2}{3} \zeta^{-1} \sqrt{d/k} + k^{-1}\right) \cdot 2s} + \tfrac{1}{3} \zeta^{-1} s \right] } \leq d \cdot \econst^{-s}.
\]
To achieve failure probability $p \in (0,1)$, set the tail parameter $s = \log(d / p)$.
The lower distortion is at most $\alpha$ when the column sparsity
$\zeta \geq 6 \alpha^{-1} \log(d/p)$ and the embedding dimension
$k \geq 16 \alpha^{-2} (d \vee \log(d/p))$.
The same parameters serve for the expectation bound.
\hfill\qed

\begin{remark}[The psd comparison theorem]
To analyze the SparseStack matrix, the papers~\cite{Tro26:Comparison-Theorems-Minimum,CEMT25:Faster-Linear} employ
Fact~\ref{fact:mineig-psd} from the companion paper~\cite[Thm.~2]{Tro26:Comparison-Theorems-Minimum}.
For this problem, the new comparison theorem (Corollary~\ref{cor:mineig-main}) provides superior bounds.
Nevertheless, the two comparison theorems are incomparable.
\end{remark}

\section{Comparison for the maximum eigenvalue: Proof}
\label{sec:maxeig-proof}

This section contains the proof of the main result (Theorem~\ref{thm:maxeig-main}),
along with the background required for the argument.
To set the stage, we restate %
the result in a form that aligns with the proof strategy.

\begin{theorem}[Comparison: Maximum eigenvalue, version 2] \label{thm:maxeig-compare}
Consider two independent sums of random self-adjoint matrices, taking values in $\Sym_d(\F)$ and with two finite moments:
\begin{align*}
\mtx{Y} &\coloneqq %
\sum_{i=1}^n \mtx{W}_i
\quad\text{where $\Expect[ \mtx{W}_i ] = \mtx{0}$ and $\lambda_{\max}(\mtx{W}_i) \leq R$;} \\
\mtx{Z} &\coloneqq %
\sum_{i=1}^n \mtx{X}_i
\quad\text{where $\mtx{X}_i \sim \normal(\mtx{0}, \Varo[\mtx{W}_i])$.}
\end{align*}
Choose any deterministic self-adjoint matrix $\mtx{\Delta} \in \Sym_d(\F)$,
and define the statistics
\[
\mu \coloneqq \Expect \lambda_{\max}(\mtx{Z} + \mtx{\Delta})
\quad\text{and}\quad
\phi \coloneqq \Expect \lambda_{\max}(\mtx{Z})
\quad\text{and}\quad
\sigma_*^2 \coloneqq \sup\nolimits_{\norm{\vct{u}} = 1} \Var[ \vct{u}^* \mtx{Z} \vct{u} ].
\]
Then the expectation of the maximum eigenvalue satisfies the comparison
\[
\Expect \lambda_{\max}(\mtx{Y} + \mtx{\Delta})
	\leq %
		\mu + \sqrt{\left(\tfrac{1}{3} R \phi + \sigma_*^2\right) \cdot 2\log d}
		+ \tfrac{1}{3} R \log d.
\]
Furthermore, for each $s \geq 0$, the tail probability satisfies the bound
\begin{align*}
\Prob{ \lambda_{\max}(\mtx{Y} + \mtx{\Delta}) \geq \mu + \sqrt{\left(\tfrac{1}{3}R \phi + \sigma_*^2\right) \cdot 2s} + \tfrac{1}{3} Rs }
	\leq d \cdot \econst^{-s}.
\end{align*}
\end{theorem}

Theorem~\ref{thm:maxeig-compare} modifies the statement of Theorem~\ref{thm:maxeig-main}
in two places.
First, we require the summands $\mtx{W}_i$ to be centered,
and the matrix $\mtx{\Delta}$ explicitly models the expectation of the sum. %
Second, we express the Gaussian proxy $\mtx{Z}$ as an independent sum of
Gaussian matrices $\mtx{X}_i$, where the statistics of $\mtx{X}_i$
match the statistics of $\mtx{W}_i$.  By the stability of the Gaussian distribution,
this approach leads to the same Gaussian proxy as in Theorem~\ref{thm:maxeig-main}.

\subsection{Overview of the proof}

The argument depends on some
strategies from the
theory of matrix concentration~\cite{Tro15:Introduction-Matrix}.
We can control the maximum eigenvalue of the sum $\mtx{Y}$ by passing
to its trace mgf. %
To compare the trace mgf of the sum $\mtx{Y}$ with the trace mgf of the Gaussian proxy $\mtx{Z}$,
we use Lindeberg's method to exchange one pair of summands $(\mtx{W}_i, \mtx{X}_i)$ at each step.
To bound the magnitude of this change, we require detailed information
about the derivatives of the trace mgf,
which we extract from Stahl's theorem~\cite{Sta13:Proof-BMV}.
Afterward, we use Gaussian concentration to express
the results in terms of statistics of the Gaussian proxy.

Lindeberg's method has several applications in random matrix
theory~\cite{Cha06:Generalization-Lindeberg,TV10:Random-Matrices,TV11:Random-Matrices,
KM11:Applications-Lindeberg,Tao19:Least-Singular,OT17:Universality-Laws,CDB+24:Sparse-Random}.
The distinctive feature of our approach is the novel appeal to Stahl's theorem.

\subsection{Extensions}

Let us mention several potential directions for extending
Theorem~\ref{thm:maxeig-compare}.  These avenues may inspire
further research.

\begin{itemize}
\item	We might sharpen the theorem for an independent
sum $\mtx{Y}$ whose summands have symmetric distributions
or whose moments match the Gaussian distribution to higher order.

\item	In fact, the Gaussian distribution of the comparison model
plays only a modest role in the argument, and we can contemplate
other types of comparison models.

\item	Classic matrix concentration results have
variants that depend on the \term{intrinsic dimension}
of the random matrix model, rather than the nominal
dimension of the matrices~\cite{Min17:Some-Extensions,Tro15:Introduction-Matrix,MR26:Intrinsic-Dimension}.
It would be valuable to explore analogous improvements to
Theorem~\ref{thm:maxeig-compare}.

\item	The proof of Theorem~\ref{thm:maxeig-compare}
depends on a comparison of exponential moments.
It may also be possible to compare polynomial moments
of random matrix models by exploiting Hein{\"a}vaara's
extension~\cite[Cor.~1.2]{Hei25:Tracial-Joint} of Stahl's theorem.

\item	Section~\ref{sec:freedman} contains an extension
of Theorem~\ref{thm:maxeig-compare} to matrix martingales.
It may be productive to develop additional tools for
studying matrix martingales.
\end{itemize}

\subsection{Background: Matrix Laplace transform}

We employ basic results from the matrix Laplace transform method~\cite[Sec.~3.2]{Tro15:Introduction-Matrix}.
This approach allows us to obtain probability bounds for the
maximum eigenvalue of a random self-adjoint matrix by passing to
the expectation of the trace exponential of the random matrix,
often called the \term{trace mgf}.

\begin{fact}[Matrix Laplace transform method] \label{fact:matrix-lt}
Let $\mtx{Y}$ be a random self-adjoint matrix whose maximum eigenvalue
is integrable.  Then the maximum eigenvalue satisfies %
\begin{align*}
\Expect  \lambda_{\max}(\mtx{Y})  &\leq \inf\nolimits_{\theta > 0}
\frac{1}{\theta} \log \Expect \trace \econst^{\theta \mtx{Y}}; \\
\Prob{ \lambda_{\max}(\mtx{Y}) \geq t }
	&\leq \inf\nolimits_{\theta > 0} \exp\left( -\theta t + \log \Expect \trace \econst^{\theta \mtx{Y} } \right)
	\quad\text{for $t \in \R$.}
\end{align*}
The matrix exponential function is defined via its Taylor series.
These results are nontrivial whenever the expectation on the
right-hand side is finite at some $\theta > 0$.
\end{fact}

\subsection{Background: The trace exponential and its derivatives}

The proof hinges on deep properties of the trace exponential function.
Stahl's theorem~\cite{Sta13:Proof-BMV}, formerly the BMV conjecture~\cite{BMV75:Monotonic-Converging},
asserts that the trace exponential along a line coincides with the
Laplace transform of a positive measure.

\begin{fact}[Stahl's theorem] \label{fact:stahl}
Fix self-adjoint matrices $\mtx{A}, \mtx{H}$ with the same dimension. %
Then
\[
f(t) \coloneqq \trace \econst^{\mtx{A} + t \mtx{H}} = \int_{\lambda_{\min}(\mtx{H})}^{\lambda_{\max}(\mtx{H})} \econst^{t \lambda} \idiff{\nu}(\lambda)
	\quad\text{for $t \in \R$.}
\]
The finite, positive Borel measure $\nu \coloneqq \nu_{\mtx{A},\mtx{H}}$
is supported on the interval $[\lambda_{\min}(\mtx{H}), \lambda_{\max}(\mtx{H})]$.
This measure is uniquely determined by the two matrices.
\end{fact}

\noindent
This formulation of Stahl's theorem %
appears in Er{\"e}menko's account of Stahl's proof~\cite{Ere15:Herbert-Stahls}.
Clivaz's paper~\cite{Cli16:Stahls-Theorem} contains another exposition of Stahl's argument.
Hein{\"a}vaara's article~\cite{Hei25:Tracial-Joint} develops a creative new approach to the result that yields many other dividends.

The Laplace transform representation of the trace exponential
furnishes detailed information about its derivatives.

\begin{proposition}[Trace exponential: Differential properties] \label{prop:trexp-diff}
Fix two self-adjoint matrices $\mtx{A}, \mtx{H}$ with the same dimension,
and define the function $f(t) \coloneqq \trace \econst^{\mtx{A} + t \mtx{H}}$ for $t \in \R$.
The derivatives with respect to the scalar variable $t$ enjoy the representation
\begin{equation} \label{eqn:trexp-Dk}
f^{(k)}(t) = %
\int_{\lambda_{\min}(\mtx{H})}^{\lambda_{\max}(\mtx{H})} \lambda^{k} \econst^{t \lambda} \idiff{\nu}(\lambda)
\quad\text{for $t \in \R$ and $k \in \Z_+$.}
\end{equation}
In particular,
\begin{enumerate}
\item	\label{item:trexp-even-pos} %
The even-order derivatives are positive:
$f^{(2k)}(t) \geq 0$ for all $t \in \R$.
Therefore, %
the odd-order derivatives $f^{(2k+1)}$ are increasing.

\item	\label{item:trexp-even-bounds} %
The even-order derivatives admit the comparisons
\[
\econst^{t \lambda_{\min}(\mtx{H})} f^{(2k)}(0)
	\leq f^{(2k)}(t) \leq
	\econst^{t \lambda_{\max}(\mtx{H})} f^{(2k)}(0)
	\quad\text{for all $t \geq 0$.} %
\]
\end{enumerate}
\end{proposition}

\begin{proof}
The formula~\eqref{eqn:trexp-Dk} follows from Fact~\ref{fact:stahl} and the
bounded convergence theorem.  The other consequences ensue immediately.
When $t \geq 0$, we have the upper bound
\[
f^{(2k)}(t) = \int_{\lambda_{\min}(\mtx{H})}^{\lambda_{\max}(\mtx{H})} \lambda^{2k} \econst^{t \lambda} \idiff{\nu}(\lambda)
	\leq \int_{\lambda_{\min}(\mtx{H})}^{\lambda_{\max}(\mtx{H})} \lambda^{2k} \econst^{t \lambda_{\max}(\mtx{H})} \idiff{\nu}(\lambda)
	= \econst^{t \lambda_{\max}(\mtx{H})}
	f^{(2k)}(0).
\]
The lower bound follows from the parallel argument.
\end{proof}

\subsection{Background: Derivatives of smooth functions}

We also require some facts about the derivatives
of smooth functions defined on matrices.  %
Although we are only concerned with the trace exponential,
it is more transparent to frame the results in a general setting.

Let $F : \Sym_d(\F) \to \R$ be a smooth, real-valued function on self-adjoint matrices,
say $K$-times Fr{\'e}chet differentiable.
As a particular example, one may consider a scalar function $\psi : \R \to \R$,
and extend it to a trace function $F : \Sym_d(\F) \to \R$
via the rule $F(\mtx{H}) \coloneqq \trace \psi(\mtx{H})$,
where $\psi(\mtx{H})$ is defined via functional calculus.
If the scalar function $\psi$ is smooth, %
then the trace function $F$ is also smooth.
For more background, see~\cite[Sec.~X.4]{Bha97:Matrix-Analysis}
or~\cite[Ch.~3]{Hig08:Functions-Matrices}.

Fix a self-adjoint matrix $\mtx{A} \in \Sym_d(\F)$.
For another self-adjoint matrix $\mtx{H} \in \Sym_d(\F)$,
define the restriction of the function $F$ to the line $t \mapsto \mtx{A} + t\mtx{H}$:
\begin{equation} \label{eqn:line-fn}
f : \R \times \Sym_d(\F) \to \R
\quad\text{where}\quad
f(t; \mtx{H}) \coloneqq F(\mtx{A} + t \mtx{H}).
\end{equation}
We can differentiate $t \mapsto f(t; \mtx{H})$ repeatedly with respect
to the scalar variable $t$.  For $k \leq K$,
the function $F$ has a Fr{\'e}chet derivative $\Diff^k F(\mtx{A})$ of order $k$,
which is a $k$-multilinear function on $(\Sym_d(\F))^k$.
Therefore, the $k$th derivative of the restriction $f$
at $t = 0$ takes the form
\begin{equation} \label{eqn:kth-gateaux}
f^{(k)}(0; \mtx{H})
	= \Diff^k F(\mtx{A})[\mtx{H}, \dots, \mtx{H}]
	= \ell_{\mtx{A}}( \otimes^{k} \, \mtx{H} ),
\end{equation}
where $\ell_{\mtx{A}} : (\otimes^k \, \Sym_d(\F)) \to \R$ is a linear functional that depends
only on the fixed matrix $\mtx{A}$ and on the function $F$.
The operator $\otimes^k$ denotes the $k$-fold tensor product,
which may be realized concretely using the Kronecker product.
In particular, we have the homogeneity property
\begin{equation} \label{eqn:kth-homogeneous}
f^{(k)}(0; \alpha \mtx{H}) = \alpha^k f^{(k)}(0; \mtx{H})
	\quad\text{for all $\alpha \in \R$.}
\end{equation}
Let us emphasize that the last two displayed formulae hold only at $t = 0$.

We are interested in the properties of derivatives at $t = 0$
when the direction $\mtx{H}$ is chosen at random.
Here and elsewhere, we assume that all expectations are finite.

\begin{proposition}[Derivatives: Statistical properties] \label{prop:diff-stat}
Consider the function $f$ defined in~\eqref{eqn:line-fn}.
Let $\mtx{W}, \mtx{X} \in \Sym_d(\F)$ be random self-adjoint matrices
that are statistically independent from $\mtx{A}$ and that
have $K$ finite moments.
For each natural number $k \leq K$,
\begin{equation} \label{eqn:moment-match}
\Expect[ \otimes^k\, \mtx{W} ] = \Expect[ \otimes^k \,\mtx{X} ]
\quad\text{implies}\quad
\Expect f^{(k)}(0; \mtx{W}) = \Expect f^{(k)}(0; \mtx{X}).
\end{equation}
In particular,
\begin{align}
\Expect \mtx{W}  = \mtx{0}
&\quad\text{implies}\quad
\Expect  f'(0; \mtx{W})   = 0;  \label{eqn:w-center} \\
\mtx{W} \sim - \mtx{W}
&\quad\text{implies}\quad
\Expect  f^{(2k+1)}(0; \mtx{W})  = 0
\quad\text{when $2k+1 \leq K$.} \label{eqn:w-sym}
\end{align}
In other words, if $\mtx{W}$ is centered, then the first derivative of $f$ at zero is centered.
If $\mtx{W}$ has a symmetric distribution, the each odd-order derivative of $f$ at zero
is centered.
\end{proposition}

\begin{proof}
Under the moment matching condition~\eqref{eqn:moment-match},
the representation~\eqref{eqn:kth-gateaux} ensures that
\[
\Expect f^{(k)}(0; \mtx{W}) 
	= \ell_{\mtx{A}}\big( \Expect[ \otimes^k\, \mtx{W} ] \big)
	= \ell_{\mtx{A}}\big( \Expect[ \otimes^k\, \mtx{X} ] \big)
	= \Expect f^{(k)}(0; \mtx{X}) .
\]
The remaining two points follow when we take $\mtx{X} = \mtx{0}$.
For example, when the distribution of $\mtx{W}$ is symmetric,
then $\mtx{W} \sim \eps \mtx{W}$.
The Rademacher variable $\eps \sim \uniform\{\pm 1\}$ is %
independent from $\mtx{W}$.
Thus,
\[
\Expect[ \otimes^{2k+1} \, \mtx{W} ]
	= \Expect[ \eps^{2k+1} \cdot (\otimes^{2k+1} \,\mtx{W}) ]
	= \Expect[ \eps^{2k+1} ] \cdot \Expect[\otimes^{2k+1} \, \mtx{W} ]
	= \mtx{0}.
\]
This observation completes the argument.
\end{proof}

\subsection{Step 1: Comparison for the trace mgf of one matrix}

The main technical challenge is to show that the trace mgf increases
when we replace a single summand %
with a \hilite{scaled} copy of the matching Gaussian summand. %
The first result isolates this computation.

\begin{proposition}[Comparison: Trace mgf of one matrix] \label{prop:compare-one}
Fix a deterministic self-adjoint matrix $\mtx{A}$.
Consider a random self-adjoint matrix $\mtx{W}$ with two finite moments %
that satisfies $\Expect[ \mtx{W} ]  = \mtx{0}$ and $\lambda_{\max}(\mtx{W}) \leq R$.
Construct a Gaussian random matrix $\mtx{X} \sim \normal(\mtx{0}, \Varo[\mtx{W}])$
with the same first- and second-order moments as $\mtx{W}$.
Then
\begin{equation} \label{eqn:wx-comp}
\Expect \trace \econst^{\mtx{A} + \mtx{W}}
	\leq \Expect \trace \econst^{\mtx{A} + g \mtx{X}}
\quad\text{where}\quad
g \coloneqq g_R \coloneqq \left[ \frac{\econst^{R} - R - 1}{R^2/2} \right]^{1/2}.
\end{equation}
\end{proposition}

As usual in Lindeberg's method, we exploit the smoothness
of the trace exponential by means of a Taylor expansion.
Stahl's theorem (Fact~\ref{fact:stahl}) provides fine control on terms in these Taylor series.

\begin{proof}
For a random self-adjoint matrix $\mtx{H}$, introduce the function
\[
f(t; \mtx{H}) \coloneqq \Expect \trace \econst^{\mtx{A} + t \mtx{H}}
\quad\text{for $t \in \R$.}
\]
We have incorporated the expectation into the function
$f$ to lighten notation.  The distribution of the random matrix
$\mtx{H}$ will always have enough regularity that we can
exchange integrals and derivatives with the expectation;
we proceed without further comment on this matter.

Let us begin with the left-hand side of the desired inequality~\eqref{eqn:wx-comp}.
The first-order Taylor expansion with integral remainder yields
\begin{align*}
\onecirc &\coloneqq \Expect \trace \econst^{\mtx{A} + \mtx{W}}
	= f(1; \mtx{W})
	= f(0; \mtx{W}) + f'(0; \mtx{W}) + \int_{0}^1 f''(t; \mtx{W}) (1-t) \idiff{t}.
\intertext{The constant term $f(0; \mtx{W}) = \trace \econst^{\mtx{A}}$.
Equation~\eqref{eqn:w-center} of Proposition~\ref{prop:diff-stat} 
ensures that the first derivative $f'(0; \mtx{W}) = 0$ because $\mtx{W}$ is centered.
Meanwhile, since $\lambda_{\max}(\mtx{W}) \leq R$,
we can employ Proposition~\ref{prop:trexp-diff} (\ref{item:trexp-even-bounds})
to bound the second derivative:}
\onecirc
	&\leq \trace \econst^{\mtx{A}} + f''(0; \mtx{W}) \int_0^1 \econst^{Rt} (1-t) \idiff{t}
	= \trace \econst^{\mtx{A}} + \frac{1}{2} g^2 f''(0; \mtx{W}).
\end{align*}
In the last step, the integral equals $(\econst^{R} - R - 1) / R^2 \eqqcolon g^2 / 2$.

Next, we turn to the right-hand side of our object~\eqref{eqn:wx-comp}.
The second-order Taylor expansion with integral remainder produces
the expression
\begin{align*}
\twocirc &\coloneqq \Expect \trace \econst^{\mtx{A} + g \mtx{X}}
	= f(1; g\mtx{X}) \\
	&= f(0; g\mtx{X}) + f'(0; g\mtx{X}) + \frac{1}{2} f''(0; g \mtx{X})
	+ \frac{1}{2} \int_0^1 f'''(t; g\mtx{X}) (1-t)^2 \idiff{t}.
\intertext{%
The constant term $f(0; g\mtx{X}) = \trace \econst^{\mtx{A}}$,
and the first-order term $f'(0; g\mtx{X}) = 0$
because the random matrix $\mtx{X}$ is centered.
In the second derivative term, formula~\eqref{eqn:kth-homogeneous} allows us
to draw out the scalar factor: $f''(0; g\mtx{X}) = g^2 f''(0; \mtx{X})$.
As for the third derivative,
Proposition~\ref{prop:trexp-diff} (\ref{item:trexp-even-pos}) states that $t \mapsto f'''(t; g\mtx{X})$
is increasing, so we obtain a \hilite{lower} bound by reducing
$t$ to zero:}
\twocirc &\geq \trace \econst^{\mtx{A}}
	+ \frac{1}{2} g^2 f''(0;\mtx{X})
	+ \frac{1}{2} \int_0^1 f'''(0; g \mtx{X}) (1-t)^2 \idiff{t} \\
	&= \trace \econst^{\mtx{A}} + \frac{1}{2} g^2 f''(0; \mtx{X}).
\end{align*}
In the last step, we invoked equation~\eqref{eqn:w-sym} of Proposition~\ref{prop:diff-stat} 
to determine that $f'''(0; g \mtx{X}) = 0$ because the centered Gaussian random
matrix $\mtx{X}$ has a symmetric distribution.

Altogether,
\[
\Expect \trace \econst^{\mtx{A} + \mtx{W}}
	- \Expect \trace \econst^{\mtx{A} + g\mtx{X}}
	= \onecirc - \twocirc
	\leq \frac{1}{2} g^2 \big[ f''(0; \mtx{W}) - f''(0; \mtx{X}) \big]
	= 0.
\]
Indeed, by construction and by Fact~\ref{fact:moment-tensor},
the two random self-adjoint matrices share the same
second-order statistics: $\Expect[\mtx{W} \otimes \mtx{W}] = \Expect[\mtx{X} \otimes \mtx{X}]$.
Therefore, formula~\eqref{eqn:moment-match} in Proposition~\ref{prop:diff-stat} ensures that
$f''(0; \mtx{W}) = f''(0; \mtx{X})$.  This observation completes the argument.
\end{proof}

\begin{remark}[Bennett's inequality]
This argument is inspired by Bennett's inequality~\cite[Sec.~2.7]{BLM13:Concentration-Inequalities},
which depends on the fact that the function
\[
h(t) \coloneqq \frac{\econst^{t} - t - 1}{t^2/2}
\quad\text{is increasing for $t \in \R$.}
\]
In particular, $h(\lambda t) \leq h(Rt)$ when $t \geq 0$ and $\lambda \leq R$.
We reach the same upper bound for $\onecirc$ by introducing
this inequality into the Laplace transform representation (Fact~\ref{fact:stahl})
of $\trace \econst^{\mtx{A} + \mtx{W}}$.
\end{remark}

\subsection{Step 2: Comparison of trace mgfs for a sum}

To compare the trace mgfs of two sums of independent random matrices,
we can exchange one summand at a time.  This is the essence of Lindeberg's strategy.

\begin{proposition}[Comparison: Trace mgf of a sum] \label{prop:exchange}
As in Theorem~\ref{thm:maxeig-compare},
consider the independent sum $\mtx{Y}$ %
and the Gaussian proxy $\mtx{Z}$. %
Choose a deterministic self-adjoint matrix $\mtx{\Delta}$.
For a parameter $\theta \geq 0$,
\[
\Expect \trace \econst^{\theta ( \mtx{Y} + \mtx{\Delta})}
	\leq \Expect \trace \econst^{g_R(\theta) \mtx{Z} + \theta \mtx{\Delta}}
	\quad\text{where}\quad
	g_R(\theta) \coloneqq \left[ \frac{\econst^{\theta R} - \theta R - 1}{R^2/2} \right]^{1/2}.
\]
Recall that $\lambda_{\max}(\mtx{W}_i) \leq R$ for each summand in $\mtx{Y}$.
\end{proposition}

\begin{proof}
Without loss, assume that the family $(\mtx{W}_1, \dots, \mtx{W}_n)$
is statistically independent from the family $(\mtx{X}_1, \dots, \mtx{X}_n)$.
By a scaling argument, we restrict our attention to $\theta = 1$.

Define matrices
\[
\mtx{A}_k \coloneqq \mtx{\Delta} + \sum_{i=1}^{k-1} \mtx{W}_i + \sum_{i=k+1}^n g_R(1) \mtx{X}_i  
\quad\text{for $k = 1, 2, 3, \dots, n$.}
\]
These matrices interpolate between the two sums $\mtx{Y}$ and $\mtx{Z}$ in the sense that
\[
\mtx{A}_n + \mtx{W}_n = \mtx{Y} + \mtx{\Delta } \quad\text{and}\quad
\mtx{A}_{1} + g_R(1) \mtx{X}_1 = g_R(1) \mtx{Z} + \mtx{\Delta}.
\]
At intermediate steps in the process, we have the identity
\[
\mtx{A}_{k} + \mtx{W}_{k} = \mtx{A}_{k+1} + g_R(1) \mtx{X}_{k+1}
\quad\text{for $k = 1, 2, 3, \dots, n-1$.}
\]
In addition, $\mtx{A}_k$ is statistically independent from $(\mtx{W}_k, \mtx{X}_k)$.
Now, we can write the difference between the trace exponentials as a telescoping sum:
\[
\Expect\trace \econst^{\mtx{Y} + \mtx{\Delta}} - \Expect \trace \econst^{g_R(1) \mtx{Z} + \mtx{\Delta}}
	= \sum_{k=1}^n \left[ \Expect \trace \econst^{\mtx{A}_k + \mtx{W}_k}
		- \Expect \trace \econst^{\mtx{A}_k + g_R(1) \mtx{X}_k} \right]
	\leq 0.
\]
Indeed, for each $k$, the bracket is bounded above by zero.
This point follows from Proposition~\ref{prop:compare-one}, applied with conditioning
on $\mtx{A}_k$, averaging over the randomness in $(\mtx{W}_k, \mtx{X}_k)$.
\end{proof}

\subsection{Step 3: Gaussian concentration}

We have compared the trace mgf of the original sum
with the trace mgf of a Gaussian matrix.
The next step uses Gaussian concentration to obtain
bounds that depend more explicitly on the statistics
of the Gaussian random matrix.

\begin{proposition}[Gaussian concentration] \label{prop:gauss-bound}
As in Theorem~\ref{thm:maxeig-compare}, 
consider the independent sum $\mtx{Y}$ %
and the Gaussian proxy $\mtx{Z}$. %
Choose a deterministic self-adjoint matrix $\mtx{\Delta}$.
For each parameter $\theta \geq 0$,
the trace mgf of the shifted sum %
satisfies the bound
\[
\log \Expect \trace \econst^{\theta (\mtx{Y} + \mtx{\Delta})}
	\leq \log d + \Expect \lambda_{\max}(g_R(\theta) \mtx{Z} + \theta \mtx{\Delta})
	+ \tfrac{1}{2} g_R(\theta)^2 \sigma_*^2(\mtx{Z}).
\]
The function $g_R$ is defined in Proposition~\ref{prop:exchange},
and  $\sigma_*^2(\mtx{Z})$ denotes the weak variance~\eqref{eqn:weak-var}.
\end{proposition}

\begin{proof}
It is convenient to abbreviate $\mtx{Z}_{\theta} \coloneqq g_R(\theta) \mtx{Z} + \theta \mtx{\Delta}$.
Beginning from Proposition~\ref{prop:exchange},
\[
\Expect \trace \econst^{\theta (\mtx{Y} + \mtx{\Delta})}
	\leq \Expect \trace \econst^{\mtx{Z}_{\theta}}
	\leq d \cdot \Expect \econst^{\lambda_{\max}(\mtx{Z}_{\theta})}.
\]
Add and subtract $\Expect \lambda_{\max}(\mtx{Z}_{\theta})$ in the exponent:
\[
\Expect \trace \econst^{\theta (\mtx{Y} + \mtx{\Delta})}
	\leq d \cdot \econst^{ \Expect \lambda_{\max}(\mtx{Z}_{\theta}) }
	 \cdot \Expect \econst^{ \lambda_{\max}(\mtx{Z}_{\theta})
	 - \Expect\lambda_{\max}(\mtx{Z}_{\theta})}.
\]
A standard application of Gaussian concentration (Fact~\ref{fact:gauss-conc}) %
yields the bound %
\[
\Expect \econst^{\lambda_{\max}(\mtx{Z}_{\theta}) - \Expect \lambda_{\max}(\mtx{Z}_{\theta}) }
	\leq \econst^{\sigma_*^2( \mtx{Z}_{\theta} ) / 2}
	= \econst^{ g_R(\theta)^2 \sigma_*^2(\mtx{Z}) / 2 }.
\]
Indeed, the weak variance statistic $\sigma_*^2(\mtx{Z}_{\theta}) = \sigma_*^2(g_R(\theta) \mtx{Z} + \theta \mtx{\Delta})$
does not depend on the deterministic shift $\theta \mtx{\Delta}$,
and it is a 2-homogeneous function of the centered random matrix
$g_R(\theta) \mtx{Z}$.
Combine the last two displays, and take the logarithm to reach the stated result.
\end{proof}

\subsection{Step 4: Bounds for the maximum eigenvalue}

Given bounds for the trace mgf, we can employ standard methods
from the theory of matrix concentration (Fact~\ref{fact:matrix-lt})
to obtain probability bounds for the maximum eigenvalue.
This result completes the proof of Theorem~\ref{thm:maxeig-compare},
supplementing the previous statement with additional estimates.

\begin{proposition}[Maximum eigenvalue: Probability bounds] \label{prop:maxeig-prob}
As in Theorem~\ref{thm:maxeig-compare},
consider the centered independent sum $\mtx{Y}$ %
and the Gaussian proxy $\mtx{Z}$.
Choose a deterministic self-adjoint matrix $\mtx{\Delta}$, and
introduce the statistics $R$, $\mu$, $\phi$, and $\sigma_*^2$.
The expectation of the maximum eigenvalue satisfies
\begin{align*}
\Expect \lambda_{\max}(\mtx{Y} + \mtx{\Delta}) \leq \mu %
	+ \sqrt{\left( \tfrac{1}{3} R \phi  + \sigma_*^2 \right) \cdot 2 \log d} + \tfrac{1}{3} R \log d.
\end{align*}
At a level $t \geq 0$, the maximum eigenvalue admits a Bennett-type tail bound
\begin{align*}
&\Prob{ \lambda_{\max}(\mtx{Y} + \mtx{\Delta}) \geq \mu %
	+ \big(\phi/3 + \sigma_*^2/R \big) \cdot t }
	\leq d \cdot \left( \frac{\econst^t}{(1+t)^{1+t}} \right)^{\phi/(3R) + \sigma_*^2/R^2}.
\intertext{The analogous Bernstein-type inequality reads}
&\Prob{ \lambda_{\max}(\mtx{Y} + \mtx{\Delta}) \geq \mu +  t} %
	\leq d \cdot \exp \left( \frac{-t^2/2}{R \phi / 3 + \sigma_*^2 + Rt / 3} \right).
\intertext{To control the failure probability at a level $s \geq 0$,}
&\Prob{ \lambda_{\max}(\mtx{Y} + \mtx{\Delta}) \geq \mu %
	+ \sqrt{\left(\tfrac{1}{3} R \phi + \sigma_*^2 \right) \cdot 2s} + \tfrac{1}{3} Rs }
	\leq d \cdot \econst^{-s}.
\end{align*}
\end{proposition}

\begin{proof}
By a scaling argument, we may take the upper bound $R = 1$.
Proposition~\ref{prop:gauss-bound} provides a bound for the trace mgf:
\[
\log \Expect \trace\econst^{\theta (\mtx{Y} + \mtx{\Delta})}
	\leq \log d + \Expect \lambda_{\max}(g_1(\theta) \mtx{Z} + \theta \mtx{\Delta})
	+ \tfrac{1}{2} g_1(\theta)^2 \sigma_*^2.
\]
To simplify the expectation at the right, first observe that
\[
0 \leq g_1(\theta) - \theta
	= ( 2 (\econst^{\theta} - \theta - 1))^{1/2} - \theta
	\leq \tfrac{1}{3} (\econst^{\theta} - \theta - 1)
	\quad\text{for $\theta \geq 0$.}
\]
Add and subtract $\theta \mtx{Z}$, and invoke Weyl's inequality:
\begin{align*}
\Expect \lambda_{\max}(g_1(\theta) \mtx{Z} + \theta \mtx{\Delta})
	&\leq (g_1(\theta) - \theta) \cdot \Expect \lambda_{\max}(\mtx{Z})
	+ \theta \cdot \Expect \lambda_{\max}(\mtx{Z} + \mtx{\Delta} ) \\
	&\leq %
	(\econst^{\theta} - \theta - 1) \cdot (\phi/3) +  \theta \cdot \mu,
\end{align*}
where we have recalled the definitions of the statistics $\phi$ and $\mu$.
It follows that
\begin{align*}
\log \Expect \trace \econst^{\theta ( \mtx{Y} + \mtx{\Delta})}
	&\leq \log d + \theta \mu + (\econst^\theta - \theta - 1)(\phi/3 + \sigma_*^2 ) \\
	&\leq \log d + \theta \mu + \frac{(\phi/3 + \sigma_*^2 )\theta^2 /2}{1 - \theta/3}.
\end{align*}
These standard mgf estimates arise within the proofs of Bennett's inequality
and Bernstein's inequality~\cite[Secs.~2.7 and 2.8]{BLM13:Concentration-Inequalities}.
The matrix Laplace transform method (Fact~\ref{fact:matrix-lt})
now furnishes expectation inequalities and tail bounds
for $\lambda_{\max}(\mtx{Y} + \mtx{\Delta})$.
The remaining details---optimization over $\theta$ and further bounds---are routine.
For example, see~\cite{BLM13:Concentration-Inequalities,Tro12:User-Friendly,MJCFT14:Matrix-Concentration,Tro15:Introduction-Matrix}.
\end{proof}

\subsection{Step 5: Unbounded summands via truncation}

Finally, we will remove the assumption that the maximum eigenvalues
of the summands are bounded above.
The next proposition restates Theorem~\ref{thm:maxeig-unbdd}
using the definitions from Theorem~\ref{thm:maxeig-compare}.

\begin{proposition}[Truncation: Maximum eigenvalue] \label{prop:truncation}
Instate the notation and assumptions from Theorem~\ref{thm:maxeig-compare}.
But this time, we allow the summands $\mtx{W}_i$ to be unbounded,
and we define %
\[
M \coloneqq \max\nolimits_i \lambda_{\max}(\mtx{W}_i).
\]
For a parameter $R \geq 2 \cdot \Expect M$,
define the upper bound statistic
\[
R_+ \coloneqq R + \max\nolimits_i \sigma_*(\mtx{W}_i).
\]
For each $s \geq 0$, the maximum eigenvalue
admits the tail probability bound
\begin{multline*}
\Probe \Big\{ \lambda_{\max}(\mtx{Y} + \mtx{\Delta}) \geq \Expect \lambda_{\max}(\mtx{Z} + \mtx{\Delta})
	+ \sqrt{2 \sigma_*^2(\mtx{Z})} \\
	+ \sqrt{\big(\tfrac{1}{3} R_+ \phi(\mtx{Z}) + \sigma_*^2(\mtx{Z}) \big) \cdot 2s} + \tfrac{1}{3} R_+ s \Big\}
	\leq \Prob{M > R} + d \cdot \econst^{-s}.
\end{multline*}
\end{proposition}

This result follows from a truncation argument;
for related examples, see~\cite[Ch.~6]{LT91:Probability-Banach}
or~\cite[Sec.~8]{BvH24:Universality-Sharp}.
In our setting, we can exploit monotonicity properties
of Gaussian random matrices to simplify some of the steps.

\begin{remark}[Expectation bounds]
We can also obtain an expectation bound
by integrating the probability bound from Proposition~\ref{prop:truncation},
as in~\cite[Sec.~8]{BvH24:Universality-Sharp}.
The resulting estimate is somewhat misleading
because it contains parasitic factors of $\log d$,
so we have chosen to omit the statement.
It seems to be delicate to obtain
an accurate bound for the expectation.
\end{remark}

\subsubsection{Truncation}

As in Theorem~\ref{thm:maxeig-compare},
let $\mtx{Y}$ be a centered independent sum,
where the summands $\mtx{W}_1, \dots, \mtx{W}_n$
may be unbounded.
Let $\mtx{Z} \sim \normal(\Expect[\mtx{Y}], \Varo[\mtx{Y}])$ be the Gaussian proxy
of the sum $\mtx{Y}$.
Set the minimum truncation level
\[
R_0 \coloneqq 2 \cdot \Expect \max\nolimits_i \lambda_{\max}(\mtx{W}_i).
\]
For a parameter $R \geq R_0$, define the events
\[
\set{B}_i \coloneqq \{ \lambda_{\max}(\mtx{W}_i) \leq R \}
\quad\text{and}\quad
\set{B} \coloneqq \bigcap_{i=1}^n \set{B}_i.
\]
Consider the truncated sum and its Gaussian counterpart:
\[
\mtx{Y}_R \coloneqq \sum_{i=1}^n \indicator_{\set{B}_i} \mtx{W}_i
\quad\text{and}\quad
\mtx{Z}_R \sim \normal( \Expect[\mtx{Y}_R], \Varo[\mtx{Y}_R] ).
\]
On the event $\set{B}$, the original sum agrees with the truncated sum:
$ \indicator_{\set{B}}\mtx{Y} = \indicator_{\set{B}}\mtx{Y}_R $.
We plan to apply the main comparison theorem (Theorem~\ref{thm:maxeig-main}) to the truncated sum $\mtx{Y}_R$.
To do so, we need to understand the statistics of the matching Gaussian distribution $\mtx{Z}_R$.
Note that truncation changes the upper bound parameter for the 
summands, as well as the expectation and variance of the sum,
so we must take care.

\subsubsection{Minimum truncation level}

To begin, we outline the reason for choosing the minimum truncation level.
For each index $i$, define the real random variable
\[
M_{-i} \coloneqq \max\nolimits_{j \neq i} \lambda_{\max}(\mtx{W}_j)
\quad\text{and}\quad
M \coloneqq \max\nolimits_j \lambda_{\max}(\mtx{W}_j).
\]
Writing $\Expect_i$ for the expectation with respect to $\mtx{W}_i$ only,
\[
\Expect M %
	\geq \Expect \max\{ \lambda_{\max}(\Expect_{i} \mtx{W}_i), \Expect_i M_{-i} \}
	= \Expect \max\{ 0, M_{-i} \}
	= \Expect[ (M_{-i})_{+} ].
\]
We have employed independence and Jensen's inequality.
As usual, $(\cdot)_+$ returns the positive part of a real number.
The truncation parameter $R \geq R_0 \coloneqq 2 \cdot \Expect M$, so
Markov's inequality yields
\begin{equation} \label{eqn:bd-M-i}
\Prob{ M_{-i} > R }
	\leq R^{-1} \Expect[ (M_{-i})_+ ]
	\leq R^{-1} \Expect[ M ]
	\leq \tfrac{1}{2}.
\end{equation}
In other words, $M_{-i}$ is not likely to
exceed the truncation level $R$.

\subsubsection{Expectation shift}

Next, we must control how much the truncation deforms the expectation
of the random matrix model.  This argument is adapted from~\cite[Lem.~8.4]{BvH24:Universality-Sharp}.

For each index $i$, introduce the scale factor
\[
\alpha_i \coloneqq \Prob{ M_{-i} \leq R }
	= \Probe( {\textstyle\bigcap_{j \neq i} \set{B}_j} )
	\geq \Probe( \set{B} ).
\]
By the relation~\eqref{eqn:bd-M-i} from the previous section, each factor $\alpha_i \geq \tfrac{1}{2}$.
Furthermore,
\begin{equation} \label{eqn:scale-factor-indicator}
\alpha_i \cdot \Expect[ \indicator_{\set{B}_i} \condbar \mtx{W}_i ]
	= \Expect[ \indicator_{\cap_{j\neq i}\set{B}_j} \indicator_{\set{B}_i}  \condbar \mtx{W}_i ]
	= \Expect[ \indicator_{\set{B}} \condbar \mtx{W}_i ].
\end{equation}
We have exploited the independence of the family $(\mtx{W}_j)$ and the associated
events $(\set{B}_j)$.

The truncated sum $\mtx{Y}_R$ and its Gaussian counterpart $\mtx{Z}_R$
share the same expectation, while the original sum $\mtx{Y}$
induces a centered Gaussian model $\mtx{Z}$.  In view of~\eqref{eqn:scale-factor-indicator},
\[
\Expect \mtx{Z}_R - \Expect \mtx{Z}
	= \Expect \mtx{Y}_R
	=  \Expect\left[ \sum_{i=1}^n \indicator_{\set{B}_i}\mtx{W}_i \right ]
	=  \Expect\left[ \indicator_{\set{B}} \sum_{i=1}^n \alpha_i^{-1} \mtx{W}_i \right].
\]
By the Rayleigh--Ritz principle for the maximum eigenvalue~\cite[Cor.~III.1.2]{Bha97:Matrix-Analysis},
\begin{equation}  \label{eqn:trunc-mean-shift}
\begin{aligned}
\lambda_{\max}( \Expect \mtx{Z}_R - \Expect \mtx{Z} )
	&= \max_{\norm{\vct{u}} = 1}
	\Expect\left[ \indicator_{\set{B}} \sum_{i=1}^n \alpha_i^{-1} \vct{u}^* \mtx{W}_i \vct{u} \right] \\
	&\leq \max\nolimits_i \ [ \Probe(\set{B})^{1/2} \alpha_i^{-1} ] \cdot
	\max_{\norm{\vct{u}}=1}  \left[  \sum_{i=1}^n \Expect \left(\vct{u}^* \mtx{W}_i \vct{u} \right)^2 \right]^{1/2} \\
	&\leq \sqrt{2 \sigma_*^2(\mtx{Y})}
	= \sqrt{2 \sigma_*^2(\mtx{Z})}.
\end{aligned}
\end{equation}
The first inequality depends on Cauchy--Schwarz and H{\"o}lder.  This argument also exploits
the centering and independence of the summands.
Last, recall that $\Probe(\set{B}) \leq \alpha_i$ and $\alpha_i \geq 1/2$, and recognize the weak variance~\eqref{eqn:weak-var} of the independent sum $\mtx{Y}$. %

\subsubsection{Comparison of variance statistics}

Next, observe that truncation \hilite{decreases} the variance function.  For each self-adjoint
matrix $\mtx{M}$,
\begin{equation} \label{eqn:trunc-var}
\begin{aligned}
\Varo[ \mtx{Z}_R ](\mtx{M}) &= \Varo[ \mtx{Y}_R ](\mtx{M})
	= \sum_{i=1}^n \Var[ \ip{  \indicator_{\set{B}_i}\mtx{W}_i }{ \mtx{M} } ] \\
	&\leq \sum_{i=1}^n \Expect[ \ip{ \indicator_{\set{B}_i}\mtx{W}_i  }{ \mtx{M} }^2 ]
	\leq \sum_{i=1}^n \Expect[ \ip{ \mtx{W}_i }{ \mtx{M} }^2 ] \\
	&=  \sum_{i=1}^n \Var[ \ip{ \mtx{W}_i }{ \mtx{M} } ]
 	= \Varo[ \mtx{Y} ](\mtx{M})
	= \Varo[ \mtx{Z} ](\mtx{M}).
\end{aligned}
\end{equation}
As a consequence of~\eqref{eqn:trunc-var}, the weak variance statistics
satisfy $\sigma_*^2(\mtx{Z}_R) \leq \sigma_*^2(\mtx{Z})$.

These bounds also yield comparisons for the other statistics of the Gaussian models.
Since the maximum eigenvalue is convex, %
the inequality~\eqref{eqn:trunc-var} for the variance functions implies %
\[
\Expect \lambda_{\max}(\mtx{Z}_R - \Expect \mtx{Z}_R + \mtx{\Delta})
\leq \Expect \lambda_{\max}(\mtx{Z} - \Expect \mtx{Z} + \mtx{\Delta})
\quad\text{for each self-adjoint $\mtx{\Delta}$.}
\]
This point follows from monotonicity (Fact~\ref{fact:gauss-mono}) %
for centered Gaussian distributions. %
Taking $\mtx{\Delta} = \mtx{0}$ in the last expression, we confirm
that the matrix fluctuation statistics $\phi(\mtx{Z}_R) \leq \phi(\mtx{Z})$.
Furthermore,
\begin{equation} \label{eqn:trunc-eig-shift}
\begin{aligned} 
\Expect \lambda_{\max}(\mtx{Z}_R + \mtx{\Delta})
	&= \Expect \lambda_{\max}( (\mtx{Z}_R - \Expect \mtx{Z}_R) + \Expect \mtx{Z}_R + \mtx{\Delta} ) \\
	&\leq \Expect \lambda_{\max}( (\mtx{Z} - \Expect \mtx{Z}) + \Expect \mtx{Z}_R + \mtx{\Delta} ) \\
	&\leq \Expect \lambda_{\max}( \mtx{Z} + \mtx{\Delta})
	+ \lambda_{\max}(\Expect \mtx{Z}_R - \Expect \mtx{Z}) \\
	&\leq \Expect \lambda_{\max}( \mtx{Z} + \mtx{\Delta}) + \sqrt{2 \sigma_*^2(\mtx{Z})}.
\end{aligned}
\end{equation}
The penultimate bound follows from Weyl's inequality for the maximum
eigenvalue~\cite[Cor.~III.2.2]{Bha97:Matrix-Analysis}.
The last line depends on our estimate~\eqref{eqn:trunc-mean-shift}
for the expectation shift.

\subsubsection{Upper bound parameter}

Last, we develop a simple estimate for the upper bound parameter of the
truncated summands.  Note that
\begin{equation} \label{eqn:trunc-upper-bound}
\begin{aligned}
\lambda_{\max}(\indicator_{\set{B}_i} \mtx{W}_i - \Expect[ \indicator_{\set{B}_i} \mtx{W}_i ] )
	&\leq \lambda_{\max}(\indicator_{\set{B}_i} \mtx{W}_i )
	+ \lambda_{\max}( - \Expect[ \indicator_{\set{B}_i} \mtx{W}_i ] ) \\
	&\leq R + \max\nolimits_i \sigma_*(\mtx{W}_i)
	\eqqcolon R_+.
\end{aligned}
\end{equation}
The detailed argument is similar with the proof of~\eqref{eqn:trunc-mean-shift},
where we employ the trivial bound $\Probe(\set{B}_i) \leq 1$.

\subsubsection{Proof of Proposition~\ref{prop:truncation}}

We are now prepared to establish the tail bound stated in Proposition~\ref{prop:truncation}.
For a parameter $s \geq 0$, define levels
\begin{equation} \label{eqn:u-uR}
\begin{aligned}
u \coloneqq u(s) &\coloneqq \sqrt{\left(\tfrac{1}{3} R_+ \phi(\mtx{Z}) + \sigma_*^2(\mtx{Z})\right)\cdot 2s} + \tfrac{1}{3}R_+ s + \sqrt{2 \sigma_*^2(\mtx{Z})} \\
	&\geq \sqrt{\left(\tfrac{1}{3} R_+ \phi(\mtx{Z}_R) + \sigma_*^2(\mtx{Z}_R) \right) \cdot 2s} + \tfrac{1}{3} R_+ s + \sqrt{2 \sigma_*^2(\mtx{Z})} \\
	&\eqqcolon u_R + \sqrt{2 \sigma_*^2(\mtx{Z})}.
\end{aligned}
\end{equation}
The inequality holds because of the comparisons for the statistics
that we established in the previous subsection.
Now, for any self-adjoint $\mtx{\Delta}$,
\begin{align*}
&\Prob{ \set{B} \ \text{and}\ \lambda_{\max}(\mtx{Y} + \mtx{\Delta}) \geq \Expect \lambda_{\max}(\mtx{Z} + \mtx{\Delta}) + u } \\
&\qquad	= \Prob{ \set{B} \ \text{and}\ \lambda_{\max}(\mtx{Y}_R + \mtx{\Delta}) \geq \Expect \lambda_{\max}(\mtx{Z} + \mtx{\Delta}) + u } \\
&\qquad \leq \Prob{ \set{B} \ \text{and}\ \lambda_{\max}(\mtx{Y}_R + \mtx{\Delta}) \geq \Expect \lambda_{\max}(\mtx{Z}_R + \mtx{\Delta}) + u_R } \\
&\qquad \leq \Prob{ \lambda_{\max}(\mtx{Y}_R + \mtx{\Delta}) \geq \Expect \lambda_{\max}(\mtx{Z}_R + \mtx{\Delta}) + u_R }
	\leq d \cdot \econst^{-s}.
\end{align*}
The first equality holds because $\mtx{Y} = \mtx{Y}_R$ on the event $\set{B}$.
The first inequality follows from the eigenvalue comparison
stated in~\eqref{eqn:trunc-eig-shift} and the relationship~\eqref{eqn:u-uR} between $u$ and $u_R$.
The second inequality follows when we remove the intersection with the event $\set{B}$.
The last statement depends on the main comparison theorem (Theorem~\ref{thm:maxeig-main}),
applied to the truncated sum $\mtx{Y}_R + \mtx{\Delta}$ and its Gaussian proxy $\mtx{Z}_R + \mtx{\Delta}$
with upper bound parameter $R_+$. %

To conclude, we remove the intersection with $\set{B}$ from the event on the left-hand side
of the last display, and we transfer the probability of $\set{B}^\comp$ to the right-hand side:
\[
\Prob{ \lambda_{\max}(\mtx{Y} + \mtx{\Delta}) \geq \Expect  \lambda_{\max}(\mtx{Z} + \mtx{\Delta}) + u(s) }
	\leq \Probe(\set{B}^\comp) + d \cdot \econst^{-s}.
\]
Last, note that the event $\set{B}^\comp = \{ M > R \}$.
This point
completes the proof.
\hfill\qed

\section{Comparison for matrix martingales}
\label{sec:freedman}

Matrix martingales generalize the independent sum model,
and they arise in a range of contemporary applications~\cite{Tro26:Applied-Random}.
This section develops new concentration inequalities for
matrix martingales by means of the Gaussian comparison technique.
To the best of our knowledge, the literature contains
no similar results.

\subsection{Matrix martingales}

We work in a Polish probability space $(\Omega, \coll{F}, \Probe)$
that is rich enough to support all random variables required.
A (finite) filtration is an increasing collection of sub-sigma-algebras:
\[
\{\emptyset, \Omega\} \coloneqq \coll{F}_0 \subseteq \coll{F}_1 \subseteq \coll{F}_2
	\subseteq \dots \subseteq \coll{F}_n \subseteq \coll{F}.
\]
A finite sequence $(\mtx{Y}_0, \dots, \mtx{Y}_n)$ of random self-adjoint matrices in $\Sym_d(\F)$
is an \term{$\set{L}_2$ matrix martingale} with respect to the filtration $(\coll{F}_0, \dots, \coll{F}_n)$ when
it satisfies three properties:  %
\begin{enumerate}
\item	\textbf{Adaptation:} $\mtx{Y}_i$ is $\coll{F}_i$-measurable for each $i = 0, \dots, n$.
\item	\textbf{Square-integrability:} $\Expect \norm{\mtx{Y}_i}^2 < + \infty$ for each $i = 0, \dots, n$.
\item	\textbf{Status quo:} $\Expect[ \mtx{Y}_{i+1} \condbar \coll{F}_i ] = \mtx{Y}_i$ for each $i = 0, \dots, n - 1$.
\end{enumerate}
We place the convention that the initial element of the martingale is deterministic,
and we often suppress the filtration when introducing a martingale.
Throughout this section, the index $i \in \{0,1,2,\dots,n\}$, unless specifically noted.

The \term{difference sequence} of the martingale comprises the random self-adjoint matrices
$\mtx{W}_k \coloneqq \mtx{Y}_k - \mtx{Y}_{k-1}$ indexed by $k = 1, \dots, n$.
The \term{increment} $\mtx{W}_k$ is measurable with respect to $\coll{F}_k$.
By the status quo property, the increment is conditionally centered:
$\Expect[ \mtx{W}_k \condbar \coll{F}_{k-1} ] = \mtx{0}$.

To an $\set{L}_2$ matrix martingale, we can associate a random sequence of
\term{predictable variance functions}.
By convention, $\set{V}_0 \coloneqq \set{0}$.  For each $k = 1, \dots, n$,
\[
\set{V}_k : \Sym_d(\F) \to \R_+
\quad\text{where}\quad
\set{V}_k : \mtx{M} \mapsto \sum_{i=1}^k \Var[ \ip{\mtx{W}_i}{\mtx{M}} \condbar \coll{F}_{i-1} ].
\]
In other words, the predictable variance accumulates our best guess for the
variance function of the next increment in the difference sequence.
Let us emphasize that $\set{V}_{k+1}$ is measurable with respect to $\coll{F}_{k}$,
so we can construct stopping times that depend on the next value
of the predictable variance.

\begin{remark}[Versions]
In general, the predictable variance functions may take arbitrary values
on negligible events.  The symbols $(\set{V}_0, \dots, \set{V}_n)$
denote a fixed version of the sequence, defined everywhere on the probability space.
We can---and will always---select a version of the sequence that increases:
$\set{V}_0 \psdle \set{V}_1 \psdle \set{V}_2 \psdle \dots \psdle \set{V}_n$ pointwise.
To be precise, for two random variance functions,
we say that $\set{V} \psdle \set{V}'$ pointwise when
$\set{V}(\mtx{M}; \omega) \leq \set{V}'(\mtx{M}; \omega)$
for all matrices $\mtx{M} \in \Sym_d(\F)$
and for all sample points $\omega \in \Omega$.
This caution is unnecessary for simple examples, hence some readers
may prefer to pass over these technicalities.
\end{remark}

\subsection{The matrix Freedman inequality via comparison}

Using the technology from this paper, we can compare
the entire trajectory of a matrix martingale
against a Gaussian reference model.
We interpret this result as a Gaussian comparison version
of the matrix Freedman inequality~\cite{Oli10:Spectrum-Random,Tro11:Freedmans-Inequality}.

\begin{theorem}[Comparison: Maximum eigenvalue of a matrix martingale] \label{thm:freedman}
Consider an $\set{L}_2$ matrix martingale
$(\mtx{Y}_0, \dots, \mtx{Y}_n)$ taking values in $\Sym_d(\F)$,
with increments $(\mtx{W}_1, \dots, \mtx{W}_n)$,
and with predictable variance functions $(\set{V}_0, \set{V}_1, \dots, \set{V}_n)$.
Assume that each increment satisfies the bound
$\lambda_{\max}(\mtx{W}_i) \leq R$ almost surely.

Fix a deterministic variance function $\set{V} : \Sym_d(\F) \to \R_+$,
and construct a Gaussian matrix $\mtx{Z} \sim \normal(\mtx{Y}_0, \set{V})$.
Define the statistics
\[
\mu \coloneqq \Expect \lambda_{\max}(\mtx{Z})
\quad\text{and}\quad
\phi \coloneqq \Expect \lambda_{\max}(\mtx{Z} - \Expect \mtx{Z})
\quad\text{and}\quad
\sigma_*^2 \coloneqq \sup\nolimits_{\norm{\vct{u}}=1} \Var[\vct{u}^*\mtx{Z} \vct{u}].
\]
For each $s \geq 0$, the maximum eigenvalue of the matrix martingale
satisfies the tail bound
\[
\Prob{ \exists i : \text{$\set{V}_i \psdle \set{V}$} \ \text{and} \ \lambda_{\max}(\mtx{Y}_i) \geq \mu 	+ \sqrt{\left(\tfrac{1}{3} R \phi + \sigma_*^2 \right) \cdot 2s} + \tfrac{1}{3} Rs }
	\leq d \cdot \econst^{-s}.
\]
\end{theorem}

The proof of Theorem~\ref{thm:freedman} extends over the rest of this section.
The argument is inspired by classic proofs of the martingale central
limit theorem; for example, see Lalley's notes~\cite{Lal14:Martingale-Central}.
Our comparison tools (Proposition~\ref{prop:compare-one}) allow us to
execute this strategy in the matrix setting.
The analysis culminates with Proposition~\ref{prop:freedman-bds},
which contains several additional tail bounds.

\begin{remark}[Matrix Freedman: Infinite martingale sequences]
Theorem~\ref{thm:freedman} remains valid for infinite martingale
sequences.  This extension follows as a corollary after an
application of monotone convergence. %
\end{remark}

\subsection{Remarks}

Theorem~\ref{thm:freedman} compares the entire trajectory of the maximum eigenvalue
of a matrix martingale against the maximum eigenvalue of a Gaussian reference model.
The Gaussian matrix shares the same expectation as the martingale sequence,
and the variance function $\set{V}$ of the Gaussian is fixed in advance.
While the predictable variance function $\set{V}_i$ of the martingale remains smaller
than the reference variance $\set{V}$, %
the maximum eigenvalue $\lambda_{\max}(\mtx{Y}_i)$ of the martingale
is unlikely to be much larger than $\Expect \lambda_{\max}(\mtx{Z})$,
the expectation of the maximum eigenvalue of the reference model.

When we apply Theorem~\ref{thm:freedman} to an independent sum
of random self-adjoint matrices, we arrive at a stronger version of
Theorem~\ref{thm:maxeig-main} that also controls the trajectory
of the partial sums.
Modulo constants, Theorem~\ref{thm:freedman} improves over
the matrix Freedman inequality~\cite[Thm.~1.2]{Tro11:Freedmans-Inequality},
much as Theorem~\ref{thm:maxeig-main} improves over the
matrix Bernstein inequality~\cite[Thm.~6.1]{Tro12:User-Friendly}.
To the best of our knowledge, Theorem~\ref{thm:freedman} appears
to be the first nonasymptotic comparison result for matrix martingales
that can avoid a suboptimal dependence on matrix dimension in the leading term.
There is no obvious strategy for establishing a similar result
using existing universality arguments from
Brailovskaya \& van Handel~\cite{BvH24:Universality-Sharp}
or Tropp~\cite{Tro26:Universality-Laws}.

\subsection{Background: A maximal inequality for matrix martingales}

In place of the Laplace transform method (Fact~\ref{fact:matrix-lt}),
we employ a maximal inequality that secures the entire trajectory
of the maximum eigenvalue of the matrix martingale.

\begin{fact}[Exponential maximal inequality: Matrix martingale] \label{fact:doob}
Let $(\mtx{S}_0, \dots, \mtx{S}_n)$ be a matrix martingale that
takes self-adjoint values.  For each $t \in \R$, %
\[
\Prob{ \exists i : \lambda_{\max}(\mtx{S}_i) \geq t }
	\leq \inf\nolimits_{\theta > 0}\ \econst^{-\theta t} \cdot \Expect \trace \econst^{\theta \mtx{S}_n}.
\]
\end{fact}

For completeness, we include a short proof.  The argument parallels
the strategy for establishing Doob's maximal
inequality~\cite[Sec.~12.6]{GS01:Probability-Random}.

\begin{proof}
Fix a parameter $\theta > 0$.
Since the trace exponential function is convex~\cite[Thm.~2.10]{Car10:Trace-Inequalities},
the scalar sequence $(\trace \econst^{\theta \mtx{S}_i} : i = 0, \dots, n)$
composes a positive submartingale.

Define a stopping time by selecting the first instant when
the maximum eigenvalue of the matrix martingale surmounts the level $t$:
\[
\kappa \coloneqq \inf\{ i : \lambda_{\max}(\mtx{S}_i) \geq t \}
\]
As usual, $\kappa = +\infty$ when the argument of the infimum is the empty set.
Introduce the event
\[
\set{B} \coloneqq \{ \kappa \leq n \}
	= \{ \exists i : \lambda_{\max}(\mtx{S}_i) \geq t \}.
\]
By the optional stopping theorem~\cite[Sec.~12.5]{GS01:Probability-Random},
we can compare the expectation of the terminal value
of the submartingale with the expectation of the stopped submartingale:
\[
\Expect \trace \econst^{\theta \mtx{S}_n}
	\geq \Expect \trace \econst^{\theta \mtx{S}_{n \wedge \kappa}}
	\geq \Expect[ \indicator_{\set{B}} \cdot \trace \econst^{\theta \mtx{S}_{n \wedge \kappa}} ]
	\geq \Expect[ \indicator_{\set{B}} \cdot \econst^{\theta \lambda_{\max}(\mtx{S}_{n \wedge \kappa})} ]
	\geq \Probe(\set{B}) \cdot \econst^{\theta t}.
\]
The indicator reduces the value of the expectation, and the
trace exponential is bounded below by the exponential of the maximum eigenvalue.
On the event $\set{B}$, we have a lower bound for the maximum eigenvalue.
This inequality implies the required estimate.
\end{proof}

\subsection{Step 1: Comparison for the trace mgf of a martingale}

In the proof of Theorem~\ref{thm:freedman},
the fundamental ingredient is a comparison
between the trace mgf of a matrix martingale
and a Gaussian reference model.
Our key result treats a special case where the total predictable
variance of the martingale is bounded by a fixed variance function.

\begin{proposition}[Comparison: Trace mgf of a martingale] \label{prop:martingale-mgf}
Consider an $\set{L}_2$ matrix martingale $(\mtx{S}_0, \dots, \mtx{S}_n)$,
taking self-adjoint values, with increments $(\mtx{W}_1, \dots, \mtx{W}_n)$,
and with predictable variance functions $(\set{V}_0, \set{V}_1, \dots, \set{V}_n)$.
Assume that each increment satisfies the bound $\lambda_{\max}(\mtx{W}_i) \leq R$
almost surely.

Fix a deterministic variance function $\set{V}$,
and assume that $\set{V}_n \psdle \set{V}$ pointwise. %
Construct a Gaussian matrix $\mtx{Z} \sim \normal(\mtx{S}_0, \set{V})$.
For a parameter $\theta \geq 0$,
\[
\Expect \trace \econst^{\theta \mtx{S}_n}
	\leq \Expect \trace \econst^{g_R(\theta) (\mtx{Z} - \mtx{S}_0) + \theta \mtx{S}_0}
	\quad\text{where}\quad
	g_R(\theta) \coloneqq \left[ \frac{\econst^{\theta R} - \theta R - 1}{R^2/2} \right]^{1/2}.
\] 
\end{proposition}

The proof relies on an elegant technical insight,
extracted from~\cite{Lal14:Martingale-Central}.
Under the hypotheses of the proposition, we can add
a terminal Gaussian increment $\mtx{T}_n$ to the martingale
to absorb the excess variance ($\set{V} - \set{V}_n$).
The final matrix $\mtx{S}_n + \mtx{T}_n$ in the extended martingale
has total predictable variance $\set{V}$, which is no longer random.
This construction allows us to replace each increment of the extended martingale
by a Gaussian proxy, working backward from the end of the trajectory.

\begin{proof}
Without loss of generality, we may assume that the parameter $\theta = 1$.
To prove the result, we must design a family
of Gaussian random matrices, tailored to
the observed increments.
Table~\ref{tab:martingale-comp} summarizes the construction. %

At each time $k \in \{1, \dots, n\}$, we have observed
the martingale values $(\mtx{S}_0, \dots, \mtx{S}_k)$ %
and the increments $(\mtx{W}_1, \dots, \mtx{W}_k)$. %
We can also compute the predictable variance functions $(\set{V}_0, \set{V}_1, \dots, \set{V}_{k+1})$.
Let us emphasize that the relations
$\set{V}_0 \psdle \set{V}_1 \psdle \dots \psdle \set{V}_n \psdle \set{V}$
hold pointwise on the probability space,
owing to our conventions and hypotheses.
Now, observe that the increments have conditional variance functions
\[
\Var[ \ip{ \mtx{W}_{k} }{ \mtx{M} } \condbar \coll{F}_{k-1} ]
	= (\set{V}_{k} - \set{V}_{k-1})(\mtx{M})
	\quad\text{for $k = 1, \dots, n$.}
\]
We can choose the version of the conditional variance where the latter
relation holds pointwise. %
In the table, the martingale values, the increments,
and the predictable variance functions are positioned below
the sigma-algebra in which they are measurable.

\begin{table}[t]
\begin{center}
\caption{Construction of a Gaussian model for a matrix martingale
$(\mtx{S}_0, \dots, \mtx{S}_n)$.
At each index $k$, we form a Gaussian proxy $\mtx{X}_k$
for the increment $\mtx{W}_k$ and a Gaussian proxy $\mtx{T}_{k}$
for the unobserved increments of the martingale.}
\label{tab:martingale-comp}
\[
\begin{array}{l|cc|cccccccc}
\text{Sigma-algebra} &&
\coll{F}_0 & \coll{F}_1 & \coll{F}_2 & \cdots & \coll{F}_k & \cdots & \coll{F}_{n-2} & \coll{F}_{n-1} & \coll{F}_n \\
\hline\hline
\text{Martingale value} &&
\mtx{S}_0 & \mtx{S}_1 & \mtx{S}_2 & \cdots & \mtx{S}_k & \cdots & \mtx{S}_{n-2} & \mtx{S}_{n-1} & \mtx{S}_n \\
\text{Martingale increments} &&
& \mtx{W}_1 & \mtx{W}_2 & \cdots & \mtx{W}_k & \cdots & \mtx{W}_{n-2} & \mtx{W}_{n-1} & \mtx{W}_n \\
\text{Predictable variance functions} && \set{V}, \set{V}_0, \set{V}_1 & \set{V}_2 & \set{V}_3 & \cdots & \set{V}_{k+1} &\cdots & \set{V}_{n-1} & \set{V}_{n} & \\
\hline\hline
\text{Gaussian increment } &&
\mtx{X}_1 & \mtx{X}_2 & \mtx{X}_3 & \cdots & \mtx{X}_{k+1} &\cdots& \mtx{X}_{n-1} & \mtx{X}_n \\
\text{Gaussian suffix} &&
\mtx{T}_0, \mtx{T}_1 & \mtx{T}_2 & \mtx{T}_3 & \cdots & \mtx{T}_{k+1} & \cdots & \mtx{T}_{n-1} & \mtx{T}_n
\end{array}
\]
\end{center}
\end{table}

At time $k = 0$, independent from everything, draw the Gaussian matrix
\[
\mtx{T}_0 \sim \normal(\mtx{0}, \set{V} - \set{V}_0).
\]
We regard this random matrix as a proxy for the sum of all increments
($\mtx{W}_1 + \dots + \mtx{W}_n$), plus a terminal increment
that absorbs ``leftover'' variance. %
Note that $\mtx{T}_0 + \mtx{S}_0 \sim \mtx{Z}$.

Conditional on $\coll{F}_n$, draw a mutually independent family
of Gaussian matrices:
\[
\mtx{X}_{k} \sim \normal(\mtx{0}, \set{V}_{k} - \set{V}_{k-1})
\quad\text{and}\quad
\mtx{T}_{k} \sim \normal(\mtx{0}, \set{V} - \set{V}_{k})
\quad\text{for $k = 1, \dots, n$.}
\]
At each step $k$, the Gaussian matrix $\mtx{X}_k$ is a proxy for the increment $\mtx{W}_k$.
The Gaussian matrix $\mtx{T}_k$ is a proxy for the total of the unobserved
increments of the extended martingale ($\mtx{W}_{k+1} + \dots + \mtx{W}_{n}$
plus the terminal increment $\mtx{T}_n$).
Since $\set{V}_{k}$ and $\set{V}_{k-1}$ and $\set{V}$ are all $\coll{F}_{k-1}$-measurable,
the construction ensures that
\[
(\mtx{X}_k, \mtx{T}_k \condbar \coll{F}_{k-1})
\sim (\mtx{X}_k, \mtx{T}_k \condbar \coll{F}_n)
\quad\text{for $k = 1, \dots, n$.}
\]
Conditional on $\coll{F}_{k-1}$, we also have the relation
$\mtx{X}_k + \mtx{T}_k \sim \mtx{T}_{k-1}$.
In the table, the entries in the two Gaussian rows are positioned
to show when the variance functions are determined.

We can compare the trace mgfs of the martingale $\mtx{S}_n$ and the Gaussian model $\mtx{Z}$
through a chain of relations.
We begin with the terminal increment $\mtx{T}_n$ of the martingale.
By Jensen's inequality, conditional on $\coll{F}_n$,
we have the comparison
\begin{equation*} \label{eqn:martingale-mgf-def}
\mathrm{mgf} \coloneqq \Expect \trace \econst^{\mtx{S}_n}
	\leq \Expect \trace \econst^{\mtx{S}_n + g_R(1) \mtx{T}_{n}}.
\end{equation*}
Indeed, the trace exponential function is convex~\cite[Thm.~2.10]{Car10:Trace-Inequalities},
the random matrix $\mtx{S}_n$ is $\coll{F}_n$-measurable, and
$\Expect[ \mtx{T}_{n} \condbar \coll{F}_n ] = \mtx{0}$.

To continue, we replace the martingale increments with conditionally Gaussian matrices,
working backward from the end of the trajectory.  Here is the
initial step in this process:
\begin{align*}
\mathrm{mgf} &\leq \Expect \Expect[ \trace \econst^{\mtx{S}_{n-1} + \mtx{W}_n + g_R(1) \mtx{T}_{n}}
	\condbar \coll{F}_{n-1}, \mtx{T}_{n} ] \\
	&\leq \Expect \Expect[ \trace \econst^{\mtx{S}_{n-1} + g_R(1) \mtx{X}_{n} + g_R(1) \mtx{T}_{n}}
	\condbar \coll{F}_{n-1} ]
	= \Expect %
	\trace \econst^{\mtx{S}_{n-1} + g_R(1) \mtx{T}_{n-1}}.
\end{align*}
The first relation follows from %
the decomposition $\mtx{S}_{n} = \mtx{S}_{n-1} + \mtx{W}_n$ and the tower law.
Conditional on $\coll{F}_{n-1}$ and $\mtx{T}_n$, the centered random matrices
$\mtx{W}_n$ and $\mtx{X}_n$ have the same second-order moments.
Thus, we can replace the increment $\mtx{W}_n$ by the scaled Gaussian
proxy $g_R(1) \mtx{X}_n$ through a conditional application
of Proposition~\ref{prop:compare-one}.
This step relies on the assumption that $\lambda_{\max}(\mtx{W}_n) \leq R$ almost surely,
given $\coll{F}_{n-1}$.
To reach the last expression, recall that $\mtx{X}_{n} + \mtx{T}_{n} \sim \mtx{T}_{n-1}$
conditional on $\coll{F}_{n-1}$.

We may iterate the argument in the last paragraph to arrive at the bound
\[
\mathrm{mgf} \leq \Expect %
	\trace \econst^{\mtx{S}_{0} + g_R(1) \mtx{T}_{0}}
	= \Expect \trace \econst^{\mtx{S}_0 + g_R(1) (\mtx{Z} - \mtx{S}_0)}.
\]
Indeed, $\mtx{T}_0 \sim \mtx{Z} - \mtx{S}_0$.
This is the desired estimate.
\end{proof}

\subsection{Step 2: Bounds for the maximum eigenvalue}

To obtain tail bounds for a general matrix martingale,
we use a stopping time argument %
to exploit the trace mgf comparison (Proposition~\ref{prop:martingale-mgf})
for a matrix martingale that has bounded predictable variance.

\begin{proposition}[Maximum eigenvalue: Uniform probability bounds] \label{prop:freedman-bds}
Introduce the notation and hypotheses of Theorem~\ref{thm:freedman}.
At a level $t \geq 0$, the maximum eigenvalue admits a Bennett-type tail bound
\[
\Prob{ \exists i : \set{V}_i \psdle \set{V} \ \text{and} \ \lambda_{\max}(\mtx{Y}_i) \geq \mu %
	+ \big(\phi/3 + \sigma_*^2/R \big) \cdot t }
	\leq d \cdot \left( \frac{\econst^t}{(1+t)^{1+t}} \right)^{\phi/(3R) + \sigma_*^2/R^2}.
\]
The analogous Bernstein-type inequality reads
\[
\Prob{ \exists i : \set{V}_i \psdle \set{V} \ \text{and} \ \lambda_{\max}(\mtx{Y}_i) \geq \mu + t }
	\leq d \cdot \exp \left( \frac{-t^2/2}{R \phi / 3 + \sigma_*^2 + Rt / 3} \right).
\]
To control the failure probability at a level $s \geq 0$,
\[
\Prob{ \exists i : \set{V}_i \psdle \set{V} \ \text{and} \ \lambda_{\max}(\mtx{Y}_i) \geq \mu 	+ \sqrt{\left(\tfrac{1}{3} R \phi + \sigma_*^2 \right) \cdot 2s} + \tfrac{1}{3} Rs }
	\leq d \cdot \econst^{-s}.
\]
\end{proposition}

\begin{proof}
Fix a version of the predictable variance sequence where
$\set{V}_0 \psdle \set{V}_1 \psdle \cdots \psdle \set{V}_n$ pointwise.
Define the stopping time
\[
\kappa \coloneqq \inf\{ i : \set{V}_{i + 1} \not\psdle \set{V} \}.
\]
In full detail, the stopping time is the random variable
\[
\kappa(\omega) = \inf\{ i : \set{V}_{i + 1}(\mtx{M}; \omega) > \set{V}(\mtx{M}; \omega)\  \text{for some $\mtx{M} \in \Sym_d(\F)$} \}
\quad\text{for $\omega \in \Omega$.}
\]
In other words, $\kappa$ marks the first instant when we expect that the
next increment of the martingale will accumulate too much variance.
As usual, $\kappa = + \infty$ when the argument of the infimum is the empty set.
Note that $\kappa$ is a stopping time because $\set{V}_{i+1}$ is measurable
with respect to $\coll{F}_i$.
For $i = 0, \dots, n$, introduce the events
\[
\set{B}_i \coloneqq
 \{ \set{V}_i \psdle \set{V} \} =
 \{ \omega \in \Omega : \set{V}_i(\mtx{M}; \omega) \leq \set{V}(\mtx{M}; \omega)
\ \text{for all $\mtx{M} \in \Sym_d(\F)$} \}.
\]
On the event $\set{B}_i$, the stopping time $\kappa \geq i$.

Define the stopped martingale
$\mtx{S}_i \coloneqq \mtx{Y}_{i \wedge \kappa}$ for $i = 0, \dots, n$.
In other words, we freeze the value of the martingale just before
its predictable variance function can surmount the level $\set{V}$.
The sequence $(\mtx{S}_0, \dots, \mtx{S}_n)$ remains an $\set{L}_2$
matrix martingale, and each of its increments inherits the
maximum eigenvalue bound from the original martingale:
$\lambda_{\max}(\mtx{S}_{i} - \mtx{S}_{i-1}) \leq R$
for $i = 1, \dots, n$.
The construction of the stopped martingale guarantees that its
predictable variance functions are all bounded above
by $\set{V}$ pointwise,
so we may apply Proposition~\ref{prop:martingale-mgf}
to control its trace mgf.

Proceeding from the maximal inequality (Fact~\ref{fact:doob}),
we can develop a tail bound for the stopped martingale.
For a parameter $t \in \R$,
\begin{align*}
\Prob{ \exists i : \lambda_{\max}(\mtx{S}_i) \geq t }
	&\leq \inf\nolimits_{\theta > 0}\ \econst^{-\theta t} \cdot \Expect \trace \econst^{\theta \mtx{S}_n} \\
	&\leq \inf\nolimits_{\theta > 0}\ \econst^{-\theta t} \cdot \Expect \trace \econst^{g_R(\theta) (\mtx{Z} - \mtx{S}_0) + \theta \mtx{S}_0} \\
	&= \inf\nolimits_{\theta > 0}\ \econst^{-\theta t} \cdot \Expect \trace \econst^{g_R(\theta) (\mtx{Z} - \mtx{Y}_0) + \theta \mtx{Y}_0}.
\end{align*}
The second inequality is Proposition~\ref{prop:martingale-mgf},
and we have noted that $\mtx{S}_0 = \mtx{Y}_0$.
Meanwhile, we can bound the probability below by taking
intersections:
\begin{align*}
\Prob{ \exists i : \lambda_{\max}(\mtx{S}_i) \geq t }
	&\geq \Prob{ \exists i : \set{B}_i \ \text{and}\ \lambda_{\max}(\mtx{S}_i) \geq t } \\
	&= \Prob{ \exists i : \set{V}_i \psdle \set{V} \ \text{and}\ \lambda_{\max}(\mtx{Y}_i) \geq t }.
\end{align*}
Indeed, on the event $\set{B}_i$, we have $\mtx{S}_i = \mtx{Y}_i$ because $\kappa \geq i$.

Combine the two bounds from the last paragraph %
to reach the relation
\[
\Prob{ \exists i : \set{V}_i \psdle \set{V} \ \text{and}\ \lambda_{\max}(\mtx{Y}_i) \geq t }
	\leq \inf\nolimits_{\theta > 0}\ \econst^{-\theta t} \cdot \Expect \trace \econst^{g_R(\theta) (\mtx{Z} - \mtx{Y}_0) + \theta \mtx{Y}_0}.
\]
The rest of the argument parallels the proof of Proposition~\ref{prop:maxeig-prob}.
\end{proof}

\textbf{Computational Tools Statement.}  The mathematical content
and exposition in this paper are purely the work of the human author.
Generative language models have not been used for any purpose.

\begin{acks}[Acknowledgments]
The author thanks Otte Hein{\"a}vaara for his insight on Stahl's theorem,
Raphael Meyer for a close reading, and Ramon van Handel for his critical feedback.
The anonymous referees provided a wealth of useful comments,
including the technical insight that supports
the proof of Theorem~\ref{thm:freedman}.
\end{acks}
\begin{funding}
This research was supported in part by ONR Award N-00014-24-1-2223
and the Caltech Carver Mead New Adventures Fund.
\end{funding}

\bibliographystyle{imsart-number} %
\bibliography{Tro25-Comparison-Maximum.bib}       %

\end{document}